\documentclass[a4paper,reqno,english]{amsart}

\usepackage{latexsym}

\usepackage{graphicx}
\usepackage{mathrsfs}
\usepackage{amsfonts}
\usepackage{amssymb}
\usepackage{amsmath}
\usepackage{amsthm}
\usepackage{amscd}
\usepackage[all,2cell]{xy}
\usepackage{color}
\usepackage{xcolor}
\usepackage[pagebackref,colorlinks]{hyperref}
\usepackage{enumerate}
\usepackage{bm}
\usepackage{bbm}
\usepackage[normalem]{ulem}
\usepackage{mathrsfs}

\usepackage{rotating}

\usepackage{comment}
\usepackage{hyperref}
\UseAllTwocells \SilentMatrices

\numberwithin{equation}{section}

\usepackage{comment}

\theoremstyle{plain}
\newtheorem{thm}{Theorem}[section]
\newtheorem{prop}[thm]{Proposition}
\newtheorem{cor}[thm]{Corollary}
\newtheorem{lemma}[thm]{Lemma}
\newtheorem{conj}[thm]{Conjecture}

\theoremstyle{definition}
\newtheorem{deff}[thm]{Definition}
\newtheorem{example}[thm]{Example}

\theoremstyle{remark}
\newtheorem{rmk}[thm]{\bf Remark}

\def\g{\gamma}
\def\d{\delta}
\def\G{\Gamma}

\def\xra{\xrightarrow[]{}}
\def\a{\alpha}
\def\b{\beta}

\def\N{\mathbb N}

\newcommand{\K}{\mathsf{k}}

\newcommand{\Ima}{\operatorname{Im}}
\newcommand{\ima}{\operatorname{i m}}

\newcommand{\Mod}{\operatorname{{\bf Mod}}}

\newcommand{\reg}{\operatorname{reg}}

\newcommand{\GR}{\operatorname{{\bf Gr}}}
\newcommand{\proj}{\operatorname{proj}}

\newcommand{\Mat}{\operatorname{\mathbb{M}}}

\newcommand{\rr}{\operatorname{{\bf r}}}
\newcommand{\s}{\operatorname{{\bf s}}}
\newcommand{\V}{\mathcal{V}}
\newcommand{\diag}{\operatorname{diag}}

\newcommand{\ring}{R\text{-rings}_{\K}}

\newcommand{\ringz}{R_0\text{-rings}_{\K_0}}

\newcommand{\grring}{R\text{-gr-rings}_{\K}}

\newcommand{\T}{{\mathbb{T}}}

\newcommand{\Sets}{\operatorname{Sets}}

\newcommand{\gr}{\operatorname{gr}}

\def \Z{\mathbb Z}
\def\-{\text{-}}

\newcommand{\End}{\operatorname{End}}

\newcommand{\Hom}{\operatorname{Hom}}
\newcommand{\HOM}
{\operatorname{HOM}}

\newcommand{\Ker}{\operatorname{Ker}}

\newcommand{\M}{\mathbb M}

\newcommand{\id}{\operatorname{id}}

\begin{document}

\title[Bergman algebras]{Bergman algebras\\ The graded universal algebra constructions}

\author{R. Hazrat}
\address{Roozbeh Hazrat: Centre for Research in Mathematics and Data Science\\Western Sydney University, Australia} \email{r.hazrat@westernsydney.edu.au}

\author{H. Li}
\address{Huanhuan Li:  Center for Pure Mathematics\\
School of Mathematical Sciences\\
Anhui University, Hefei, China} \email{lhh@ahu.edu.cn}

\author{R. Preusser}
\address{Raimund Preusser: School of Mathematics and Statistics\\
Nanjing University of Information Science and Technology, Nanjing, China}
\email{raimund.preusser@gmx.de}


\keywords{Bergman algebra, Leavitt path algebra, graded ring, $\Gamma$-monoid, graded Grothendieck group}

\date{\today}

\begin{abstract} A half a century ago, George Bergman introduced  stunning machinery which would realise any commutative conical monoid as the non-stable $K$-theory of a ring. The ring constructed is ``minimal" or ``universal". Given the success of graded $K$-theory in classification of algebras and its connections to dynamics and operator algebras, the realisation of $\Gamma$-monoids (monoids with an action of an abelian group $\Gamma$ on them) as non-stable graded $K$-theory of graded rings becomes vital. In this paper, we revisit Bergman's work and develop the graded version of this universal construction. For an abelian group $\Gamma$,  a $\Gamma$-graded ring $R$, and non-zero graded finitely generated projective (left) $R$-modules $P$ and $Q$, we construct a universal $\Gamma$-graded ring extension $S$ such that $S\otimes_R P\cong S\otimes_R Q$ as graded $S$-modules. This makes it possible to bring the graded techniques, such as smash products and Zhang twists into Bergman's machinery. 
Given a commutative conical   $\Gamma$-monoid $M$, we construct a $\Gamma$-graded ring $S$ such that $\mathcal V^{\gr}(S)$ is $\Gamma$-isomorphic to $M$. In fact we show that any finitely generated $\Gamma$-monoid can be realised as the non-stable graded $K$-theory of a hyper Leavitt path algebra. Here $\mathcal V^{\gr}(S)$  is the monoid of isomorphism classes of graded finitely generated projective $S$-modules and the  action of $\Gamma$ on $\mathcal V^{\gr}(S)$ is by shift of degrees. Thus the group completion of $M$ can be realised as 
the graded Grothendieck group $K^{\gr}_0(S)$. We use this machinery to provide a short proof to the fullness of the graded Grothendieck functor $K^{\gr}_0$ for the class of Leavitt path algebras (i.e., Graded Classification Conjecture II). 
\end{abstract}

\maketitle

\begin{center}
\emph{ In memory of Nikolai Vavilov who conveyed the joy of mathematics to us in Chinese, German and Persian}
\end{center}
\tableofcontents

\section{Introduction}

In 1974 in the seminal paper~\cite{bergman74} George Bergman introduced machinery whereby by adding sufficient generators and relations to an algebra, one can lift any pair of non-zero finitely generated projective modules to become isomorphic. To be precise, let $R$ be an algebra over a field $\K$, and let $(P_i, Q_i)_{i\in I}$ be pairs of non-zero finitely generated projective left $R$-modules. Bergman's machinery adjoins certain generators and relations to $R$, so that the resulted extension algebra $S$ obtained provides universal isomorphisms of modules, $S\otimes_R P_i \cong S\otimes_{R} Q_i$, where $i\in I$. Furthermore, the non-stable $K$-theory of the algebra $S$ can be obtained from that of $R$ by identifying the pairs $(P_i, Q_i)_{i\in I}$ in the structure.

Starting from a field $\K$ and choosing the pair of  $\K$-modules $(\K^n, \K^m)$,  $n,m\in \mathbb N^{+}$,  Bergman observed that his machinery  retrieves the celebrated
Leavitt algebra $R=L_\K(n,m)$, where 
$R^n \cong \K^n \otimes_\K R\cong \K^m \otimes_\K R \cong R^m$ as $R$-modules. Many combinatorial algebras constructed in the last 50 years, such as Leavitt path algebras and their generalisations, can be obtained from Bergman's machinery.

Two natural questions immediately arise: 

\begin{enumerate}
    \item {\bf The structural properties}: To describe the structure of the extension algebra $S$ obtained from $R$ relative to the pairs $(P_i,Q_i)_{i\in I}$.

    \item {\bf The classification}: Given pairs $(P_i, Q_i)_{i\in I}$ of $R$-modules and $(P'_i, Q'_i)_{i\in I}$ of $R'$-modules, for $\K$-algebras $R$ and $R'$, respectively, to describe how the extension $\K$-algebras $S$ and $S'$ obtained from the Bergman machinery, respectively, compare.
\end{enumerate}

We highlighted Leavitt path algebras as substantial activity is currently taking place
on both questions (1) and (2).
  In fact, there are several conjectures concering the classification problems and what could be the right invariant for Leavitt path algebras~\cite{lpabook}. Leavitt path algebras associated to graphs have rich $\mathbb Z$-graded structures (arising by assigning degree $1$ to edges) which play a vital role in their studies. Their  non-stable \emph{graded} $K$-theory is  related to invariants of symbolic dynamics and the equivariant $K$-theory of operator algebras and the current  open conjecture is that the graded $K$-theory could be a complete invariant for these algebras (see~\cite{willie,haz2013} and \S\ref{lpaaplicaton}). This is one of the main motivations of this project: to study the Bergman machinery in the graded setting.

Let $\K$ be a commutative ring and let $R$ be a unital $\K$-algebra. Further let $N$ be an $R$-module and $P$ a finitely generated projective $R$-module. Let $\ring$ be the category of $\K$-algebras $T$ equipped with a $\K$-homomorphism $R\rightarrow T$. Bergman's work started by showing that the following functor is representable:
\begin{equation}\label{bergpresnt}
\begin{split}
\mathcal F_{N,P}: \ring &\longrightarrow{} \Sets\\
T&\longmapsto \Hom(T\otimes_R N, T\otimes_R P).
\end{split}
\end{equation}
That is, there exists a universal algebra $S$ (or the initial object $R\rightarrow S$ in an appropriate category), denoted by $R\langle f: \overline N \rightarrow \overline P \rangle$, which provides the $S$-homomorphism $S\otimes_R N \rightarrow S\otimes_R P$. Bergman further showed that there exists a universal algebra $S$, denoted by $R\langle \overline f=0\rangle$,  such that the extension of a given $R$-module homomorphim $f: N\rightarrow P$ vanishes. 

Combining these two constructions, allows for forming very interesting universal rings. For example, given a homomorphism $f:P\rightarrow Q$, one can construct a universal ring $R\langle \overline f^{-1} \rangle$, where the extension of $f$ becomes invertible.  The two important universal constructions needed to realise any conical monoid are the following: Let $P$ and $Q$ be non-zero  finitely generated projective left $R$-modules. Constructing the universal homomorphisms $h:\overline P\rightarrow \overline Q$ and $h': \overline{\overline Q}\rightarrow \overline{\overline P}$ in subsequent extensions and setting $h'\overline{h}-1$ and $\overline{h}h'-1$ to zero, we obtain the universal ring  
\begin{equation}\label{const1}
S:=R\big\langle h,h^{-1}:\overline{P}\cong \overline{Q}\big\rangle,
\end{equation}
where $S\otimes_R P\cong S\otimes_R Q$ as $S$-modules.

On the other hand, for $P$ a non-zero   finitely generated projective left $R$-module, by constructing the universal homomorphism $e:P\rightarrow P$ and setting the homomorphism $e^2-e$ to zero we form a universal ring 
\begin{equation}\label{const2}
    T:=R\big\langle e:\overline{P}\to \overline{P};~e^2=e\big\rangle,
    \end{equation}
such that the extension of $e$ in $T$ is an idempotent endomorphism.

Bergman then proved \cite[Theorems 5.1, 5.2]{bergman74} that $\V(S)\cong \V(R)/\langle [P]=[Q]\rangle$ and 
$\V(T)\cong \langle \V(R), [P_1], [P_2] \rangle /\langle [P]=[P_1]+[P_2]\rangle$, for the universal rings $S$ and $T$ of (\ref{const1}) and (\ref{const2}), respectively. Here $\V(S)$  is the monoid of isomorphism classes of finitely generated projective left $S$-modules. 
These constructions allowed Bergman to build up a machinery to realise any (finitely generated) conical commutative monoid as the non-stable $K$-theory of a ring. 

Graded algebras frequently appear when there is a group acting on an algebraic structure. A graded ring can be a model of an algebraic structure that captures and reflects the time evolution and the dynamics.
A ring $R$ is \emph{graded} by a group $\Gamma$ when, roughly, $R$ can be partitioned by $\Gamma$ in a way that is compatible with the structure of $R$. (See \S\ref{grprelim} for more details.)  Consequently, three categories play a prominent role in this setting: the category of left $R$-modules $R$-$\Mod$, the category of graded left $R$-modules $R$-$\GR$, and the category of left $R_{0}$-modules $R_{0}$-$\Mod$, where $R_{0}$ is the subring of $R$ consisting of the elements in the partition corresponding to the identity element $0$ of $\Gamma$. A substantial portion of the theory of graded rings concerns the relationships between these categories. While the applications of this theory are numerous, one prominent example is the fundamental theorem of $K$-theory, proved by Quillen~\cite{quillen}, using the category of graded modules in a crucial way.

Our starting point is to consider a $\Gamma$-graded $\K$-algebra $R$ over a $\Gamma$-graded commutative ring $\K$,  a $\Gamma$-graded (left) $R$-module $N$ and a $\Gamma$-graded finitely generated projective (left) $R$-module $P$, where $\Gamma$ is an abelian group, and consider the corresponding functors in the graded setting: 
\begin{equation}\label{grbergmp}
\begin{split}
\mathcal F^{\gr}_{N,P}: \grring & \longrightarrow{} \Sets\\
T&\longmapsto \Hom_{T\text{-}\GR}(T\otimes_R N, T\otimes_R P)
\end{split}
\end{equation}

\begin{equation}\label{grbergmp0}
\begin{split}
\mathcal F_{N_0,P_0}: \ringz & \longrightarrow{} \Sets\\
T&\longmapsto \Hom(T\otimes_{R_0} N_0, T\otimes_{R_0} P_0)
\end{split}
\end{equation}

Here $\grring$ is the category of $\Gamma$-graded $\K$-algebras $T$ equipped with a graded $\K$-homomorphism $R\rightarrow T$. In this paper we show that $\mathcal F^{\gr}_{N,P}$ is representable, i.e., there is a $\Gamma$-graded algebra $S$ such that $ \Hom(S,-) \cong \mathcal F^{\gr}_{N,P}$.   Furthermore, if $R$ is strongly graded, then $S \cong R \ast_{R_0} S_0 $, where $S_0$ is the representation for  $\mathcal F_{N_0,P_0}$ and $\ast$ is the coproduct. Forgetting the grading on $S$ gives us the representation for 
$\mathcal F_{N,P}$ of (\ref{bergpresnt}). 

Having these graded universal constructions, we can realise any conical commutative $\Gamma$-monoid as the nonstable graded $K$-theory of a suitable graded ring. Namely, for such a $\Gamma$-monoid $M$, there is a $\Gamma$-graded $\K$-algebra $R$ such that $M\cong \mathcal V^{\gr}(R)$, as $\Gamma$-monoids. Here $\mathcal V^{\gr}(R)$  is the monoid of isomorphism classes of graded finitely generated projective $R$-modules and the  action of $\Gamma$ on $\mathcal V^{\gr}(R)$ is by shift of degrees (Theorem~\ref{mainalibm}). 
In order to carry out our $\Gamma$-monoid realisations, we had two options; either to develop the entire theory of monoids for co-products rings (\cite{bergman74mod}) in the graded setting, or to start with a graded ring, and use the fact that the category of graded modules is equivalent to the category of modules over its corresponding smash product (i.e., to pass from the graded setting to the non-graded case and use the available techniques).  
We opted for the second option; we pass from the graded structure via smash product to the category of modules and  
use the established results in the non-graded setting. The price we had to pay for such a shortcut, 
was to drop some generality, such as working with graded $\K$-algebras, where $\K$ is a field concentrated in degree zero, rather than with $\K$-algebras, where $\K$ is a graded field. 

Our realisations of $\Gamma$-monoids allow us to provide a short proof to one of the Graded Classification Conjectures (See Conjecture~\ref{conjalg}(1)): Let $A$ be a $\mathbb Z$-graded $\K$-algebra where $\K$ is concentrated in degree zero and $\mathcal V^{\gr}(A)$ is cancellative. Then any order preserving $\Z[x,x^{-1}]$-module
homomorphism $\phi: K_0^{\gr}(L_\K(E)) \rightarrow K_0^{\gr}(A)$ with 
$\phi([L_\K(E)])=[A]$ 
is induced by a unital $\mathbb Z$-graded $\K$-homomorphism $\psi: L_\K(E) \rightarrow A$, i.e., $K^{\gr}_0(\psi) = \phi$. Here $L_\K(E)$ is the Leavitt path algebra associated to the graph $E$. Replacing the algebra $A$ with a Leavitt path algebra $L_\K(F)$ provides a positive answer to the fullness conjecture.

In general given two finitely generated projective modules, it is not straightforward to describe Bergman's (localisation) algebras explicitly. Given finitely generated $R$-modules $P$ and $Q$, and a map $f:P\rightarrow Q$, there are few instances where the localisation rings  $R\langle f: \overline P \rightarrow \overline Q \rangle$, or $R\langle \overline f^{-1} \rangle$  are explicitly described. We refer to the paper of Sheiham~\cite{sheiham} where the localisation of triangular matrix rings has been described. In this paper, by representing the finitely generated projective modules by the corresponding idempotent matrices, we first describe the Bergman's constructions in terms of generators and relations.  Thus given idempotent matrices $e\in \M_m(R)$ and $f\in \M_n(R)$, we define the Bergman algebra $B_R(e,f)$ via generators and relations such that $e$ and $f$ become equivalent idempotents in $B_R(e,f)$ in a universal way. 
The construction can be extended to families of pairs of idempotents $(e,f) :=\{(e_i,f_i)\}_{i\in I}$. In fact this construction will be carried out for graded rings and homogeneous idempotents in Section \ref{bergmanalgebrasec}. The algebra $B_R(e,f)$ obtained this way, coincides with 
the universal localisation ring $R\langle f, f^{-1}: \overline P_i \cong \overline Q_i \mid i \in I\rangle$, for pairs of finitely generated projective modules $(P_i,Q_i), i \in I$ which represent $(e_i,f_i), {i\in I}$.

Returning to the graded setting, we show that when $R$ is a $\Gamma$-graded commutative $\K$-semisimple ring which is concentrated in degree zero, then the universal localisation ring $R\langle f, f^{-1}: \overline P_i \cong \overline Q_i \mid i \in I\rangle$, for any family of pairs of finitely generated projective modules $(P_i,Q_i), i \in I$, is graded isomorphic to a Leavitt path algebra of a hypergraph and conversely any such hyper Leavitt path algebra can be realised as a graded Bergman's construction over a $\Gamma$-graded commutative $\K$-semisimple ring  (Theorem~\ref{lemhyperberg}). Built on this, we show that any finitely generated conical $\Gamma$-monoid can be realised as the non-stable graded $K$-theory of a hyper Leavitt path algebra (Theorem~\ref{mainalibm_2}).

As mentioned, starting with a field $\K$, and projective modules $\K$ and $\K^n$, Bergman's machinery would give the Leavitt algebra $L_\K(1,n)$. This ring can be graded by assigning $1$ and $-1$ to the suitable generators. However, we can start with a field $\K$ concentrated in degree zero, and consider the $\mathbb Z$-graded projective modules  $\K$ and $\K(-1)^n$, where $\K(-1)$ is the projective module $\K$ shifted by $-1$. Applying our graded Bergman construction, we obtain the algebra $L_\K(1,n)$ which is naturally $\mathbb Z$-graded from the outset. This allows us to build the graded invariants such as non-stable graded $K$-theory from the machinery we develop.   

The paper is organised as follows: Section~\ref{sectiontwo} provides a background needed for the concepts of $\Gamma$-monoids and graded algebras, as well as the notion of smash products. In Section~\ref{lpalabel}, we recall the notion of Leavitt path algebras of hypergraphs, which generalises several extensions of Leavitt path algebras. This is the generalisation which covers all Bergman algebras constructed from commutative semisimple rings as a base ring. In Section~\ref{bergmanalgebrasec}, we introduce Bergman algebras, via generators and relations, both the non-graded version and the graded version. Section~\ref{bergconscore} which is the core part of this paper, carries out Bergman's localisation construction~\cite{bergman74} in the setting of graded rings. 
Section~\ref{smashbergloc} investigates how the graded Bergman localisation behaves under the smash product. This is needed in Section~\ref{sec7} to compute the non-stable graded $K$-theory of graded Bergman algebras.  Section~\ref{secmainres} develops many consequences of the graded Bergman machinery: we realise any conical $\Gamma$-monoid as the non-stable graded $K$-theory of an appropriate graded algebra. In fact we show that any finitely generated $\Gamma$-monoid can be obtained as the non-stable graded $K$-theory of a hyper Leavitt path algebra. Using our result, we provide a short positive answer to one of the Graded Classification Conjectures (Conjecture~\ref{conjalg}(1)) for Leavitt path algebras in Section~\ref{lpaaplicaton}. 

 In this paper, all rings are unital and all modules are considered as left modules unless otherwise stated.  $\mathbb N$ denotes the natural numbers with $0$ and $\mathbb N^+ = \mathbb N \backslash \{0\}$.

\section{$\Gamma$-monoids and graded algebras}\label{sectiontwo}

\subsection{$\Gamma$-monoids}\label{gammamoni} 

Recall that a {\it monoid} is a semigroup with an identity element. Throughout this paper monoids are commutative, written additively, with the identity element denoted by $0$.  A monoid homomorphism $\phi:M\rightarrow N$ is a map between monoids $M$ and $N$ which respects the structures and satisfies $\phi(0)=0$. Every monoid $M$ is equipped with a natural preordering: $n \leq m$ if $n+p=m$ for some $p\in M$.
We write 
\[n \propto m \text{ if } n \leq k m, \text{ for some } k \in \mathbb N^{+},\] and 
\[n \asymp m \text{ if } n \propto m \text{ and } m \propto n.\] Thus $n \asymp m$ if there is a $k \in \mathbb N^+$ such that $m\leq kn$ and $n \leq km$.
We say $i \in M$ is \emph{an order unit} if $m \propto i$  for any $m\in M$.
We say $M$ is \emph{conical} if $m+n=0$ implies that $m=n=0$, where $m,n \in M$. We say $M$ is \emph{cancellative} if  $m_1+n=m_2+n$ implies $m_1=m_2$, for all $m_1,m_2,n\in M$. 

Given a group $\Gamma$, a \emph{$\Gamma$-monoid} consists of a monoid $M$ equipped with an action of $\Gamma$ on $M$ (by monoid automorphisms). We denote the action of $\alpha\in\Gamma$ on $m\in M$ by ${}^\alpha m$. Throughout this paper $\Gamma$ will be an abelian group and $M$ a commutative monoid. A monoid homomorphism $\phi\colon M \rightarrow N$ between two $\Gamma$-monoids is called $\Gamma$-\emph{monoid homomorphism} if $\phi$ respects the actions of $\Gamma$, i.e., $
\phi({}^\alpha m)={}^\alpha \phi(m)$ for all $m\in M$.  The natural preordering is respected by the action of $\Gamma$. 

An {\it order unit} of a $\Gamma$-monoid $M$ is an element $i\in M$ such that for any $m\in M$, there are  $\gamma_1,\dots,\gamma_t\in \Gamma$ such that 
\begin{equation}\label{orderuni}
    m \leq \sum_{k=1}^{t}{}^{\gamma_k}i.
\end{equation}

We denote the $\Gamma$-monoid $M$ along with the order unit $i$ by  $(M,i)$ and call it a \emph{pointed $\Gamma$-monoid}. A morphism $f:(M,i) \rightarrow (N,j)$ in the category of pointed $\Gamma$-monoids, is a $\Gamma$-monoid morphism $f$ such that $f(i)=j$.

The order unit $i$ is called a \emph{strong order unit} if for any $m \in M$, we can choose all $\gamma_k=0$ in (\ref{orderuni}).   We say $i$ is a \emph{invariant order unit} if ${}^\gamma i =i$ for any $\gamma \in \Gamma$.

A {\it congruence} on a $\Gamma$-monoid $M$ is an equivalence relation $\sim$ on $M$ such that $m\sim m'$ and $n\sim n'$ implies $m+n\sim m'+n'$, and $m\sim m'$ implies $ {}^\gamma m\sim {}^\gamma m'$, for any $\gamma\in \Gamma$ and $m,m',n,n'\in M$. If $\sim$ is a congruence on $M$, then $M/\sim$ is a $\Gamma$-monoid in the obvious way. We call $M/\sim$ a {\it quotient} of $M$. 

 If $M$ is a $\Gamma$-monoid and $(m_i, n_i)$, where $i\in I$, a family of pairs of elements of $M$, then \emph{the congruence on the $\Gamma$-monoid $M$ generated by the relations $m_i = n_i$} will be the congruence generated on $M$ as a monoid by the relations ${}^{\gamma} m_i = {}^{\gamma} n_i$, where  $i\in I$ and $\gamma\in\Gamma$.

 The action of the abelian group $\Gamma$ respects the relations  $\propto$ and $\asymp$ and one can see that $\asymp$ is an equivalence relation. 

 \begin{lemma}
     Let $M$ be a $\Gamma$-monoid and $i\in M$ a $\Gamma$-order unit. Then  the following are equivalent. 
\begin{enumerate}[\upshape(1)]
\item The element $i$ is a strong order unit;

\item For any $\gamma \in \Gamma$ we have $i \asymp {}^\gamma i$; 

\item 
The equivalence class $[i]$ under $\asymp$, which is clearly a
sub-semigroup of $M$, is in fact a $\Gamma$-subsemigroup of $M$.

\item  If $M$ has a strong order-unit, then all
such elements form an equivalence class under $\asymp$.

\end{enumerate}
\end{lemma}
\begin{proof}
    The proofs are straightforward. 
\end{proof}

 Clearly the image of a homomorphism $f:M\to N$, denoted by $\ima f$,  is a $\Gamma$-submonoid of $N$. The {\it kernel} $\ker f$ of $f$ is the congruence $\sim$ on $M$ defined by $m\sim m'$ if  $f(m)=f(m')$, so that $M/\ker f$ is a $\Gamma$-monoid and the canonical homomorphism $M/\ker f \rightarrow N$ is injective (see Lemma~\ref{lemgammahom}).  

Let $X$ be a set. The {\it free  $\Gamma$-monoid  on $X$}, denoted by $F\langle X \rangle$, is constructed as the free monoid on the set $X\times \Gamma$, with the action of $\delta \in \Gamma$ on $(x,\gamma)$ defined by $(x,\gamma+\delta)$. To be precise, 
\[F\langle X \rangle=\Big \{\sum_{x\in X,\gamma\in \Gamma}n_{x,\gamma}(x,\gamma)\mid n_{x,\gamma}\geq 0 \text{, for any }x\in X \text{ and }\gamma\in \Gamma; \text{ almost all }n_{x,\gamma}\text{ are zero}\Big\}\]
which becomes a $\Gamma$-monoid with addition
\[\sum_{x\in X,\gamma\in \Gamma}n_{x,\gamma}(x,\gamma)+\sum_{x\in X,\gamma\in \Gamma}n'_{x,\gamma}(x, \gamma)=\sum_{x\in X,\gamma\in \Gamma}(n_{x,\gamma}+n'_{x,\gamma})(x, \gamma)\]
and $\Gamma$-action
\[{}^{\delta}\Big (\sum_{x\in X,\gamma\in \Gamma}n_{x,\gamma}(x,\gamma)\Big )=\sum_{x\in X,\gamma\in \Gamma} n_{x,\gamma}x(\gamma+\delta).\]
One checks easily that $ F \langle X \rangle $ has the following universal property.

\begin{lemma}\label{lemgammafree}
Let $X$ be a set, $M$ a  $\Gamma$-monoid and $f:X\to M$ a map. Let $g:X\to F\langle X \rangle$ be the map defined by $x\to x(0)$. Then there is a unique homomorphism $h:F\langle X \rangle\to M$ of $\Gamma$-monoids such that $f=hg$.
\end{lemma}

\begin{lemma}[Homomorphism theorem]\label{lemgammahom}
Let $f:M\to N$ be a homomorphism of $\Gamma$-monoids. Then $M/\ker f\cong \ima f$.
\end{lemma}
\begin{proof}
Define a map $\bar f:M/\ker f\to\ima f$ by $\bar f([m])=f(m)$ for any $m\in M$. Clearly $\bar f$ is well-defined and bijective. One checks easily that $f$ is a homomorphism of $\Gamma$-monoids. 
\end{proof}

\begin{prop}\label{propgammaquot}
Any $\Gamma$-monoid is isomorphic to a quotient of a free $\Gamma$-monoid.
\end{prop}
\begin{proof}
Let $M$ be a $\Gamma$-monoid. It follows from Lemma \ref{lemgammafree} that there is a surjective homomorphism $\gamma:F\langle X \rangle\to M$ of $\Gamma$-monoids. Thus $F\langle X \rangle/\ker \gamma\cong \ima \gamma=M$ by Lemma \ref{lemgammahom}.
\end{proof}

 Let $X$ be a set and $R$ a set of relations $m_i=n_i$, where $m_i,n_i\in F\langle X \rangle$, for any $i\in I$. We write $M=\langle X\mid R\rangle$ if $M\cong F\langle X \rangle/\sim$, where $\sim$ is the congruence on $F\langle X \rangle$ generated by $R$. In this case we say that {\it $M$ is a $\Gamma$-monoid  presented by the generating set $X$ and the relations $R$}.

 \subsection{Graded algebras}\label{grprelim}

In this section we collect basic facts that we need about  graded rings and graded algebras. We refer the reader to \cite{haz, NvObook} for the theory of graded rings.

Let $\Gamma$ be an abelian group with identity denoted by $0$. A ring $R$ (possibly without unit)
is called a \emph{$\Gamma$-graded ring} if $ R=\bigoplus_{ \gamma \in \Gamma} R_{\gamma}$
such that each $R_{\gamma}$ is an additive subgroup of $R$ and $R_{\gamma}  R_{\delta}
\subseteq R_{\gamma + \delta}$ for all $\gamma, \delta \in \Gamma$. The group $R_\gamma$ is
called the $\gamma$-\emph{homogeneous component} of $R.$ When it is clear from context
that a ring $R$ is graded by the group $\Gamma,$ we simply say that $R$ is a  \emph{graded
ring}. We denote the set of all homogeneous elements of the graded ring $R$, by $R^h$. 

A unital $\Gamma$-graded ring $R=\bigoplus_{\gamma\in \Gamma}R_\gamma$ is \emph{strongly graded} if $R_{\alpha} R_{\beta}=R_{\alpha+\beta}$ for any $\alpha,\beta\in \Gamma$ and it is \emph{crossed-product} if each component $R_{\alpha}$, $\alpha \in \Gamma$ contains an invertible element. Furthermore, we say $R$ is a \emph{graded field} if it is commutative and all non-zero homogeneous elements are invertible. The notion of strongly graded, in particular, is an important case as it widely generalises the case of group rings, and it has been gradually established that many results on group rings can be extended to strongly graded rings.

Let $\K$ be a $\Gamma$-graded commutative ring and $R$ a unital $\Gamma$-graded ring.  We say $R$ is a \emph{graded $\K$-algebra} if it is equipped with a graded ring homomorphism $\phi:\K\rightarrow Z(R)$, where $Z(R)$ is the centre of $R$. Note that since $\Gamma$ is abelian, $Z(R)$ is a graded ring. Let $S$ be a graded $\K$-algebra. Then we say $S$ is a \emph{$\Gamma$-graded $R\text{-ring}_{\K}$} if it is equipped with a graded $\K$-algebra homomorphism $R\rightarrow S$. We denote the category of 
$\Gamma$-graded $R\text{-rings}_{\K}$ by $\grring$. Considering $\Gamma$ to be trivial we obtain the category of $\ring$. 

For a graded $A$-module $M$, we define the $\a$-\emph{shifted} graded 
$A$-module $M(\a)$ as
\begin{equation*}M(\a)=\bigoplus_{\g\in \G}M(\a)_{\g},
\end{equation*}
where $M(\a)_{\g}=M_{\a+\g}$. That is, as an ungraded module, $M(\alpha)$ is a copy of
$M$, but the grading is shifted by $\alpha$. 

For a $\Gamma$-graded ring $R$, the category of graded $R$-modules is denoted by $R$-$\GR$ and the category of $R$-modules by $R$-$\Mod$.
For $\a\in\G$, the \emph{shift functor}
\begin{equation*}
\mathcal{T}_{\a}: R\-\GR \longrightarrow R\-\GR,\quad M\longmapsto M(\a)
\end{equation*}
is an isomorphism with the property $\mathcal{T}_{\a}\mathcal{T}_{\b}=\mathcal{T}_{\a+\b}$
for $\a,\b\in\G$.

Let $R$ be a strongly $\Gamma$-graded $\K$-algebra. For any $R_0$-module $N$ and any $\gamma\in\Gamma$, we identify the $R_0$-module $R_{\g}\otimes_{R_0} N$ with its image in $R\otimes_{R_0} N$. Since $R=\bigoplus_{\g\in\G} R_{\g}$ and $R_{\g}$ are $R_0$-bimodules, $R\otimes_{R_0} N$ is a $\G$-graded $R$-module, with
$$R\otimes_{R_0}N=\bigoplus_{\g\in\G} R_{\g}\otimes_{R_0}N.$$

Recall that we have the \emph{restriction functor} 
\begin{equation}\label{jetlagwarsaw}
(-)_0: R\-\GR \longrightarrow R_0\text{-}\Mod
\end{equation}
sending $M$
to $M_0$ and $\phi$ to $\phi_{|_{M_0}}$ with $M$ a graded $R$-module and $\phi$ a graded module homomorphism from $M$, and the \emph{induction functor}
\begin{equation}\label{jetlagwarsaw2}
\begin{split}
R\otimes_{R_0}-: R_0\text{-}\Mod &\longrightarrow R\-\GR\\
 N &\longmapsto R\otimes_{R_0}N\\
 \psi &\longmapsto R\otimes \psi.
 \end{split}
\end{equation}

When $R$ is strongly graded,  Dade’s Theorem \cite[Theorem 1.5.1]{haz} (see also
\cite[Theorem 3.1.1]{NvObook}), guarantees that the functors $(-)_0$ and $R\otimes_{R_0}-$ form mutually inverse equivalences of categories. 

 Note that in general $(-)_0: R\text{-}\GR\xrightarrow{} R_0\text{-}\Mod $ does not induce a functor $(-)_0:R\text{-} \GR_{\proj}\xrightarrow[]{} R_0\text{-}\Mod_{\proj}$. But  $R$ is strongly graded if and only if $R\text{-} \GR_{\proj}$ is equivalent to $R\text{-}\Mod_{\proj}$, if and only if $R\text{-}{\bf gr}$ is equivalent $R_0\text{-}{\bf mod}$, both under the functors~(\ref{jetlagwarsaw}) and (\ref{jetlagwarsaw2}). Here $R\text{-}{\bf gr}$ denotes the category of graded finitely generated $R$-modules and $R_0\text{-}{\bf mod}$ is the category of finitely generated $R_0$-modules.

Let $R$, $S$ and $A$ be $\Gamma$-graded $\K$-algebras and  let $i: A\rightarrow R$ and $j: A\rightarrow S$ be $\K$-graded algebra monomorphisms (so we identify $A$ as a graded subring of $R$ and $S$). The coproduct $R*_A S$ can be given a natural $\Gamma$-grading, so that this coproduct is indeed the pushout for the graded maps $i$ and $j$.

\subsection{Matrix form of graded homomorphisms}\label{matrixrep}

Let $\G$ be an abelian group, $\K$ a $\Gamma$-graded commutative ring with identity and $R$ a $\Gamma$-graded $\K$-algebra.

A graded $R$-module is called a {\em graded free} module, if it is a free module with homogeneous basis. Clearly a graded free module is a free module. A graded free $R$-module can be presented as $\bigoplus_{i\in I}R(\gamma_i)$, where $\gamma_i \in \Gamma$. 

A graded $R$-module $P$ is called a {\em graded projective} module if $P$ is a projective $R$-module. One can check that $P$ is graded projective if and only if the functor $\Hom_{R\text{-}\GR}(P,-)$ is an exact functor in the category $R$-$\GR$ if and only if $P$ is graded isomorphic to a direct summand of a graded free $R$-module (\cite[Proposition 1.2.15]{haz}). So a module is  graded projective module if and only if it is graded and projective. This is not the case for graded free module. A module could be graded and free, but not graded free (see~\cite[\S~1.2.4]{haz}). 


In particular, if $P$ is a graded finitely generated projective $R$-module, then there is a graded finitely generated projective $R$-module $Q$ such that 
\begin{equation}\label{projidem1}
 P\bigoplus Q\cong_{\gr} \bigoplus_{i=1}^n R(\alpha_i),   
\end{equation}
with $\alpha_i\in \Gamma$. We denote $\bigoplus_{i=1}^n R(\alpha_i)$  by $R^n (\overline{\alpha})$ with $\overline{\alpha}=(\alpha_1, \cdots, \alpha_n)$.

 One observes that ${\Hom}_{R}(R(\d_j), R(\d_i))\cong_{\gr}R(\d_i-\d_j)$ as graded $R$-modules. Hence we have the following isomorphisms:
 \begin{equation}
 \label{Hom}
 \begin{split}
 \Hom_{R}(\bigoplus_{i\in I}R(\a_i),\bigoplus_{j\in J}R(\b_j))
 &\cong \prod_{i\in I}\Hom_{R}(R(\a_i), \bigoplus_{j\in J}R(\b_j))\\
 &\cong \prod_{i\in I}\bigoplus_{j\in J}\Hom_{R}(R(\a_i), R(\b_j)) \\ 
 &\cong \prod_{i\in I}\bigoplus_{j\in J}R(\b_j-\a_i).	\end{split}
 \end{equation}

 For a $\G$-graded $\K$-algebra $R$, $\overline{\a}=(\a_i)_{i\in I}$, $\overline{\b}=(\b_j)_{j\in J}$ with $\a_i, \b_j\in \G$, set
 \begin{equation}
 \label{zero}
 \mathbb{M}_{I\times J}(R)[\overline{\a}][\overline{\b}]:={\begin{pmatrix}
 	R_{\b_j-\a_i}
 \end{pmatrix}}_{i\in I, j\in J},
 \end{equation}
where for each $i\in I$, only a finite number of entries in $(R_{\b_j-\a_i})$ are non-zero. Thus $\mathbb{M}_{I\times J}(R)[\overline{\a}][\overline{\b}] $ consists of matrices with the $ij$-entry in $R_{\b_j-\a_i}$, with a finite number of non-zero elements in each row.  One observes that $\mathbb{M}_{I\times J}(R)[\overline{\a}][\overline{\b}] $  represents the set $\Hom_{R\text{-}\GR}(\bigoplus_{i\in I}R(\a_i),\bigoplus_{j\in J}R(\b_j))$ as follows:  Suppose that $u:\bigoplus_{i\in I}R(\a_i)\xrightarrow{} 
 \bigoplus_{j\in J} R(\b_j)$ is a graded $R$-module homomorphism. Then for each $i\in I$, we can write $u(1_i)=(u_{ij})_{j\in J}$, where $1_i$ is the element of $\bigoplus_{i\in I}R(\a_i)$ whose $i$-th component is $1$ and whose other components are zero (note that $\deg(1_i)=-\a_i)$), and $u_{ij}\in R(\b_j)_{-\a_i}=R_{\b_j-\a_i}$ are all zero except for a finite number of $j$ in $J$. Then the graded homomorphism $u:\bigoplus_{i\in I}R(\a_i)\xrightarrow{} 
 \bigoplus_{j\in J} R(\b_j)$ is given by multiplying the  matrix $(u_{ij})_{i\in I, j\in J}\in \mathbb M_{I\times J}(R)[\overline{\a}][\overline{\b}]$ from the right. Suppose that $v:\bigoplus_{j\in J}R(\b_j)\xrightarrow{} 
 \bigoplus_{k\in K} R(\g_k)$ is given by multiplying the  matrix $(v_{jk})_{j\in J, k\in K}\in \mathbb M_{J\times K}(R)[\overline{\b}][\overline{\g}]$ from the right. Then the composition $v\circ u$ is given by multiplying the matrix $(u_{ij})_{i\in I, j\in J}  \cdot (v_{jk})_{j\in J, k\in K}$ from the right. Here, 
 \begin{equation}\label{jan16dis}
 (u_{ij})_{i\in I, j\in J}  \cdot (v_{jk})_{j\in J, k\in K}
 \end{equation}is the multiplication of matrices. Notice that the matrix representation of the composition of functions $v \circ u$ is obtained by swapping the orders (\ref{jan16dis}). 

 The above notion of \emph{mixed-shift} was developed for graded right $R$-modules in \cite[\S1.3.4]{haz}.

In fact, any element in $\prod_{i\in I}\bigoplus_{j\in J}R(\b_j-\a_i)$ in the last equation of \eqref{Hom} can be written as $(x_i)_{i\in I}$, where $x_i\in \bigoplus_{j\in J}R(\b_j-\a_i)$ is a row vector  indexed by $J$ ($x_i$ having finitely many non-zero entries for each $i\in I$).

For a $\G$-graded $\K$-algebra $R$, $\overline{\a}=(\a_1, \cdots, \a_m)\in\Gamma^m$, $\overline{\b}=(\b_1,\cdots, \b_n)\in \Gamma^n$, (\ref{zero}) takes the form
 \begin{equation}\label{aidansan}
 \mathbb{M}_{m\times n}(R)[\overline{\a}][\overline{\b}]={\begin{pmatrix}
 	R_{\b_1-\a_1}& R_{\b_2-\a_1}&\cdots& R_{\b_n-\a_1}\\
  R_{\b_1-\a_2}& R_{\b_2-\a_2}&\cdots& R_{\b_n-\a_2}\\
  \vdots& \vdots& \ddots& \vdots\\
  R_{\b_1-\a_m}& R_{\b_2-\a_m}&\cdots& R_{\b_n-\a_m}\\
 \end{pmatrix}},
 \end{equation} and $\mathbb{M}_{m\times n}(R)[\overline{\a}][\overline{\b}]$ represents the set $\Hom_{R\text{-}\GR}(R^m(\overline{\a}),R^n(\overline{\b}))$. When $m=n$, we simply write $\mathbb{M}_{m\times n}(R)[\overline{\a}][\overline{\b}]$ as $\mathbb{M}_{n}(R)[\overline{\a}][\overline{\b}]$.


 By \eqref{Hom} we have the graded matrix ring
 $\End_{R}(R^n(\overline{\a}))\cong \bigoplus_{i=1}^n \bigoplus_{j=1}^n R(\a_j-\a_i)$, denoted by $ \mathbb{M}_{n}(R)(\overline{\a})$. More precisely 
 \begin{equation}
     \mathbb{M}_{n}(R)(\overline{\a})=
     {\begin{pmatrix}
 	R(\a_1-\a_1)& R(\a_2-\a_1)&\cdots& R(\a_n-\a_1)\\
  R(\a_1-\a_2)& R(\a_2-\a_2)&\cdots& R(\a_n-\a_2)\\
  \vdots& \vdots& \ddots& \vdots\\
  R(\a_1-\a_n)& R(\a_2-\a_n)&\cdots& R(\a_n-\a_n)
 \end{pmatrix}}
 \end{equation} 
 For each $\g\in \Gamma$, $\mathbb{M}_{n}(R)(\overline{\a})_{\gamma}$, the $\gamma$-homogeneous elements, are the
$n\times n$-matrices over $R$ with the degree shifted (suspended) as follows 
 \begin{equation}\label{homogenhdgdt}
     \mathbb{M}_{n}(R)(\overline{\a})_{\gamma}=
     {\begin{pmatrix}
 	R_{\gamma+\a_1-\a_1}& R_{\gamma+\a_2-\a_1}&\cdots& R_{\gamma+\a_n-\a_1}\\
  R_{\gamma+\a_1-\a_2}& R_{\gamma+\a_2-\a_2}&\cdots& R_{\gamma+\a_n-\a_2}\\
  \vdots& \vdots& \ddots& \vdots\\
  R_{\gamma+\a_1-\a_n}& R_{\gamma+\a_2-\a_n}&\cdots& R_{\gamma+\a_n-\a_n}
 \end{pmatrix}}.
 \end{equation} 

We identify graded finitely generated projective modules with homogeneous idempotent matrices of degree zero. This is taken from~\cite[Lemma~3.2.3]{haz}.

\begin{lemma} We have the following statements.
    \begin{enumerate}[\upshape(1)]
          \item Any graded finitely generated projective module gives rise to a homogeneous idempotent matrix of degree zero.
        \item Any homogeneous idempotent matrix of degree zero gives rise to a graded finitely generated projective module.
    \end{enumerate}
\end{lemma}

\begin{proof}
For (1),  let $P$ be a graded finitely generated  projective $R$-module. Then
there is a graded module $Q$ such that $P\oplus Q\cong R^n(\overline{\a})$ for some $n\in \mathbb N$ and $\overline{\a}=(\a_1,\cdots, \a_n)\in\Gamma^n$. Define the homomorphism $e \in  \End_{R}(R^n(\overline{\a}))$ 
which sends $Q$ to zero and acts as identity on $P$. Clearly, $e$ is an idempotent and
graded homomorphism of degree zero. Thus $e\in \End_{R}(R^n(\overline{\a}))_0\cong \mathbb{M}_{n}(R)(\overline{\a})_{0}$. 

For (2), let $e\in \mathbb{M}_{n}(R)(\overline{\a})_{0}$ be a homogeneous idempotent matrix of degree zero, where $\overline{\a}=(\a_1,\cdots, \a_n)$ with $\a_i\in\Gamma$. Then $1-e\in \mathbb{M}_{n}(R)(\overline{\a})_{0} $ and $$R^n(\overline{\a})=R^ne(\overline{\a})\oplus R^n(1-e) (\overline{\a}).$$ 
 This shows that $R^ne(\overline{\a})$ is a graded finitely generated projective $R$-module. 
\end{proof}


In the first part of the paper we assume $\K$ to be a commutative graded ring. In the second part, for the realisation of conical $\Gamma$-monoids, we assume $\K$ to be a (graded) field concentrated in degree zero. 

\subsection{Non-stable Graded $K$-theory}

For a $\Gamma$-graded ring $R$ with identity and a graded finitely generated  projective (left) $R$-module $P$, let $[P]$ denote the class of graded $R$-modules graded isomorphic to $P$. Then the set
\begin{equation}\label{zhongshan1}
\mathcal V^{\gr}(R)=\big \{ \, [P] \mid  P  \text{ is a graded finitely generated projective R-module} \, \big \}
\end{equation}
with direct sum as addition constitutes a $\Gamma$-monoid structure defined as follows: $[P]+[Q]= [P\oplus Q]$ and for $\gamma \in \Gamma$ and $[P]\in \mathcal V^{\gr}(R)$, ${}^\gamma [P]=[P(\gamma)]$. The group completion of $\mathcal V^{\gr}(R)$ is called the \emph{graded Grothendieck group} and is denoted by $K^{\gr}_0(R)$, which 
as the above discussion shows is a $\Z[\Gamma]$-module. 

Following the notion of order unit for a $\Gamma$-monoid in~\S\ref{gammamoni}, 
an element $[P]\in \mathcal V^{\gr}(R)$ is an order unit for $\mathcal V^{\gr}(R)$ if for any $[Q] \in \mathcal V^{\gr}(R)$ we have $[Q]  \leq \sum_{1\leq i \leq n} [P(\gamma_i)]$ for some $n\in \mathbb N$ and $\gamma_i \in \Gamma$. 
Since for any graded finitely generated projective module $Q$, there is a graded module $P$ such that $Q\oplus P \cong R(\gamma_1)\oplus R(\gamma_2)\oplus \cdots \oplus R(\gamma_t)$, it follows that that $[R]$ is an order unit (see \cite[\S3]{haz}). 

The following easy proposition shows that the order unit $[R]$ in $\mathcal V^{\gr}(R)$ can determine when $R$ is a strongly graded or a crossed product ring (see also~\cite[Theorems 2.10.1 and 2.10.2]{NvObook}). This will be used in Theorem~\ref{mainalibm}. 

\begin{prop}\label{mnhygh}
Let $R$ be a $\Gamma$-graded ring with identity and $\mathcal V^{\gr}(R)$ the monoid of graded finitely generated projective modules. Then 
\begin{enumerate}[\upshape(1)]


\item The ring $R$ is strongly graded if and only if $[R]$ is an strong order unit.

\item  The  ring $R$ is crossed-product if and only if $[R]$ is an invariant order unit. 

\end{enumerate}
\end{prop}
\begin{proof}

(1) The ring $R$ is strongly graded if and only if, for any $\alpha \in \Gamma$,  $1=\sum_{i=1}^{n_\alpha} r_i s_i$, where $r_i\in R_\alpha$ and $s_i\in R_{-\alpha}$. This is equivalent to having matrices ${\mathbf r}=(r_1,\dots, r_{n_\alpha})\in \M_{1\times n_\alpha}(R)[\overline \alpha][0]$ and 
${\mathbf s}=(s_1,\dots, s_{n_\alpha})^t\in \M_{n_\alpha\times 1}(R)[0][\overline \alpha]$ with ${\mathbf r}{\mathbf s}=1$.
In turn this is equivalent to having an epimorphism $\bigoplus_{n_\alpha}R\rightarrow R(\alpha)\rightarrow 0$, which is equivalent to $[R(\alpha)]\leq n_\alpha[R]$ in $\mathcal V^{\gr}(R)$. This is now equivalent to $[R]$ being a strong order unit. 

(2) This follows from the observation that there is a graded $R$-module isomorphism $R\cong R(\alpha)$ if and only if there is an invertible element in $R_\alpha$. 
\end{proof}

\subsection{Graded hereditary rings}\label{gradedhersec}
Recall that a unital ring $R$ is called a left hereditary ring if any left ideal of $R$ is a projective $R$-module~\cite[Chapter~4]{weibel}. Analogously, we say a $\Gamma$-graded ring $R$ is called a \emph{left graded hereditary ring} if any graded left ideal of $R$ is a graded projective module. 

The following proposition is the graded version of well-known equivalent properties of hereditary rings whose proof follows mutatis mutandis, with an attention given to the grading.

\begin{prop}\label{herher}
Let $R$ be a $\Gamma$-graded ring. The following statements are equivalent.
    \begin{enumerate}[\upshape(1)]
        \item $R$ is left graded hereditary;
        \item Any graded submodule of a graded projective left $R$-module is graded projective;
         \item Any graded submodule of a graded projective left $R$-module is projective;
         \item The graded left global dimension is less than two;
        \item ${\rm Ext}_{R\-\GR}^2(-,-)$ vanishes. 
        \end{enumerate}
\end{prop}

Proposition~\ref{herher}(3) implies that a $\Gamma$-graded ring which is hereditary is indeed graded hereditary. However, the following example shows that a graded hereditary ring is not necessarily hereditary. This is in contrast with the fact that graded regular rings and regular rings which are graded (i.e., their (graded) global dimension is finite) coincide.

Consider a strongly $\Gamma$-graded ring $R$. By Dade's theorem (see Section~\ref{grprelim}), the categories $R$-$\GR$ and  $R_0$-$\Mod$ are equivalent. An application of Proposition~\ref{herher} shows that if $R_0$ is hereditary, then $R$ is a graded hereditary ring. The  
 $\mathbb Z$-graded ring $R=\mathbb Z[x,x^{-1}]$ with $\deg(x)=1$ and $\deg(x^{-1})=-1$ is a strongly graded ring with the zero-component the ring $\mathbb Z$ which is hereditary (PID rings are hereditary). Thus $R$ is a graded hereditary ring. However, $R$ is not hereditary: for example, the (non-graded) ideal $\langle 2,x-1 \rangle$ is not a projective module. 
 
 Note that the graded ring $\K[x,x^{-1}]$, where $\K$ is a field, is indeed a PID ring and thus a hereditary ring.  We will see the graded rings coming out of the graded Bergman constructions are indeed hereditary and thus they are graded hereditary as well. 

\subsection{Smash products of graded algebras}\label{smashprem}

\subsubsection{Smash products}\label{smashsmashhit}
 Let $R$ be a $\Gamma$-graded ring. The {\it smash product} ring $R \# \Gamma$ is defined as the set of all formal sums $\sum\limits_{\gamma\in\Gamma}r^{(\gamma)}p_{\gamma}$, where $r^{(\gamma)}\in R$ for any $\gamma\in\Gamma$, the $p_{\gamma}$'s are symbols, and all but finitely many coefficients $r^{(\gamma)}$ are zero. Addition is defined component-wise and multiplication is defined by linear extension of the rule $(rp_{\alpha})(sp_{\beta})=rs_{\alpha-\beta}p_{\beta}$, where $r,s\in R$ and $\alpha,\beta\in \Gamma$. Instead of $1p_\gamma$ we may write $p_\gamma$. If $R$ is a $\Gamma$-graded $\K$-algebra where $\K$ is a commutative graded ring concentrated in degree $0$, then $R\# \Gamma$ is a (not-necessarily unital) $\K$-algebra with scalar multiplication $\lambda\sum\limits_{\gamma\in\Gamma}r^{(\gamma)}p_{\gamma}=\sum\limits_{\gamma\in\Gamma}\lambda r^{(\gamma)}p_{\gamma}$, for any $\lambda\in \K$ and $\sum\limits_{\gamma\in\Gamma}r^{(\gamma)}p_{\gamma}\in R\#\Gamma$.


\begin{lemma}\label{lempressmash}
 Let $R$ be a $\Gamma$-graded $\K$-algebra where $\K$ is a commutative graded ring concentrated in degree $0$. Then the $\K$-algebra $R \# \Gamma$ has the presentation
 \begin{align*}
 R\#\Gamma=\Big\langle rp_{\gamma}~(r\in R^{h}, \gamma\in \Gamma)\mid{}& rp_\gamma+sp_\gamma=(r+s)p_\gamma\quad(r,s\in R_{\delta};~\gamma,\delta\in \Gamma),\\
 &~rp_{\gamma}sp_{\delta}=rs_{\gamma-\delta} p_{\delta}\quad(r,s\in R^{h};~\gamma,\delta\in\Gamma )\Big\rangle.
 \end{align*}
 \end{lemma}
\begin{proof}
Denote by $F$ the free $\K$-algebra generated by the set $\{rp_{\gamma}~(r\in R^h, \gamma\in \Gamma)\}$, and by $I$ the two-sided ideal of $F$ generated by the set \[\big \{rp_\gamma+sp_\gamma-(r+s)p_\gamma\mid r,s\in R_{\delta};~\gamma,\delta\in \Gamma\big \}\bigcup \big \{rp_{\gamma}sp_{\delta}-rs_{\gamma-\delta} p_{\delta}\mid r,s\in R^h;~\gamma,\delta\in\Gamma\big \}.\] 

Clearly the $\K$-algebra homomorphism $\theta:F\to R\#\Gamma$ defined by $\theta(rp_{\gamma})=rp_\gamma~(r\in R^h,\gamma\in \Gamma)$ induces a $\K$-algebra homomorphism $\bar\theta:F/I\to R\#\Gamma$ (since $I\subseteq \ker(\theta)$). Since the elements $rp_\gamma~(r\in R^h,\gamma\in \Gamma)$ generate the $\K$-algebra $R\#\Gamma$, the homomorphism $\bar\theta$ is surjective. It remains to show that $\bar\theta$ is injective. Suppose that $\bar\theta(x)=0$, where $x\in F/I$ is non-zero. Clearly we can write $x=r_1p_{\gamma_1}+\dots+r_np_{\gamma_n}$,  where $r_1,\dots,r_n\in R^h\setminus \{0\}$ and $\gamma_1,\dots,\gamma_n\in\Gamma$. For $\gamma\in\Gamma$ set $J(\gamma):=\{1\leq i\leq n \mid \gamma_i=\gamma\}$. We may assume that 
\begin{equation}\label{2.1}
\deg(r_i)\neq \deg(r_j), \text{ for any }\gamma\in \Gamma \text{ and }i\neq j\in J(\gamma). 
\end{equation}
It follows from 
\[0=\bar\theta(x)=r_1p_{\gamma_1}+\dots+r_np_{\gamma_n}=\sum_{\gamma\in\Gamma}
\Big (\sum_{i\in J(\gamma)}r_i\Big )p_\gamma\]
that $\sum_{i\in J(\gamma)}r_i=0$, for any $\gamma\in\Gamma$. In view of (\ref{2.1}) this implies $r_i=0$, for any $1\leq i\leq n$, a contradiction. Thus we have shown that $\bar\theta$ is injective.
\end{proof}

\subsubsection{Smash products of quotients and extensions}
 We continue with the assumption that $R$ is a $\Gamma$-graded $\K$-algebra where $\K$ is a commutative graded ring concentrated in degree $0$. If $I$ is a graded ideal of $R$, then we denote by $I\#\Gamma$ the ideal of $R\#\Gamma$ consisting of all elements $\sum\limits_{\gamma\in\Gamma}r^{(\gamma)}p_{\gamma}$ such that $r^{(\gamma)}\in I$, for any $\gamma\in\Gamma$. Note that the quotient ring $R/I$ is $\Gamma$-graded ring. 

 \begin{prop}\label{propsmashquot}
 Let $R$ be a $\Gamma$-graded $\K$-algebra where $\K$ is a commutative graded ring concentrated in degree $0$, and  $I$  a graded ideal of $R$. Then 
 \[(R/I)\# \Gamma\cong R\#\Gamma/I\#\Gamma.\]
 \end{prop}
 \begin{proof}
 We define the map 
 \begin{align*}
     f:(R/I)\# \Gamma&\longrightarrow R\#\Gamma/I\#\Gamma\\
     \sum\limits_{\gamma\in\Gamma}(r^{(\gamma)}+I)p_{\gamma}&\longmapsto\sum\limits_{\gamma\in\Gamma}r^{(\gamma)}p_{\gamma}+I\#\Gamma.
 \end{align*}
 Suppose that $\sum\limits_{\gamma\in\Gamma}(r^{(\gamma)}+I)p_{\gamma}=\sum\limits_{\gamma\in\Gamma}(s^{(\gamma)}+I)p_{\gamma}$ in $(R/I)\# \Gamma$. Then $\sum\limits_{\gamma\in\Gamma}(r^{(\gamma)}-s^{(\gamma)}+I)p_{\gamma}=0$ and hence $r^{(\gamma)}-s^{(\gamma)}\in I$, for any $\gamma\in\Gamma$. It follows that $\sum\limits_{\gamma\in\Gamma}(r^{(\gamma)}-s^{(\gamma)})p_{\gamma}\in I\#\Gamma$ and hence $\sum\limits_{\gamma\in\Gamma}r^{(\gamma)}p_{\gamma}+I\#\Gamma=\sum\limits_{\gamma\in\Gamma}s^{(\gamma)}p_{\gamma}+I\#\Gamma$. Thus $f$ is well-defined. Clearly $f$ is a surjective $\K$-algebra homomorphism. It remains to show $f$ is injective. Suppose that $f(\sum\limits_{\gamma\in\Gamma}(r^{(\gamma)}+I)p_{\gamma})=0$, i.e.,  $\sum\limits_{\gamma\in\Gamma}r^{(\gamma)}p_{\gamma}+I\#\Gamma=0$. Then $\sum\limits_{\gamma\in\Gamma}r^{(\gamma)}p_{\gamma}\in I\#\Gamma$ and hence $r^{(\gamma)}\in I$ for any $\gamma\in\Gamma$. It follows $\sum\limits_{\gamma\in\Gamma}(r^{(\gamma)}+I)p_{\gamma}=0$ and thus $f$ is injective.
 \end{proof}

If $X$ is a set, we denote by $R \langle X \rangle $ the $\K$-algebra obtained from $R$ by adjoining the set $X$ (i.e., $R \langle X \rangle$ is the free product of $R$ and the free $\K$-algebra $k\langle X\rangle$ generated by $X$). If in addition we are given a map $\deg:X\to\Gamma$, then $R \langle X \rangle$ becomes a $\Gamma$-graded $\K$-algebra with the induced grading. We denote by $X\#\Gamma$ the set $\{xp_\gamma\mid x\in X,\gamma\in\Gamma\}$.

 \begin{prop}\label{propsmashext}
Let $R$ be a $\Gamma$-graded $\K$-algebra where $\K$ is a commutative graded ring concentrated in degree $0$, and $X$  a set  and $\deg:X\to\Gamma$ a map. Consider  $R \langle X \rangle$ as a $\Gamma$-graded $\K$-algebra with the induced grading. Then \[R \langle X \rangle\# \Gamma\cong R\#\Gamma \langle X\#\Gamma \rangle /J,\] where $J$ is the ideal of $R\#\Gamma \langle X\#\Gamma \rangle $ generated by the set \[ \big\{(xp_\gamma)(1p_\gamma)-xp_\gamma,~(1p_{\gamma+\deg(x)})(xp_\gamma)-xp_\gamma\mid x\in X,\gamma\in \Gamma\big\}.\]
 \end{prop}
 \begin{proof}
 It follows from Lemma \ref{lempressmash} that $R\#\Gamma \langle X\#\Gamma \rangle /J$ has the presentation
 \begin{equation}\label{2.2}
 \begin{split}
     R\#\Gamma \langle X\#\Gamma\rangle /J=\Big\langle rp_{\gamma}, xp_\gamma~(r\in R^h,x\in X,\gamma\in \Gamma)\mid{}& rp_\gamma+sp_\gamma=(r+s)p_\gamma\quad(r,s\in R_{\delta};~\gamma,\delta\in \Gamma),\\
 &~rp_{\gamma}sp_{\delta}=rs_{\gamma-\delta} p_{\delta}\quad(r,s\in R^h;~\gamma,\delta\in\Gamma ),\\
 &~(xp_\gamma)(1p_\gamma)=xp_\gamma\quad(x\in X,~\gamma\in \Gamma),\\
 &~(1p_{\gamma+\deg(x)})(xp_\gamma)=xp_\gamma\quad (x\in X,~\gamma\in \Gamma)\Big\rangle.
 \end{split}
 \end{equation}
 Since the relations in presentation (\ref{2.2}) above are satisfied in $R \langle X \rangle \# \Gamma$, there is a $k$-algebra homomorphism \[f:R\#\Gamma \langle X\#\Gamma\rangle /J\rightarrow R\langle X \rangle \# \Gamma,\] such that $f( rp_{\gamma})=rp_{\gamma}$ and $f( xp_{\gamma})=xp_{\gamma}$, for any $r\in R^h,x\in X,\gamma\in \Gamma$. Clearly the elements $rp_{\gamma}, xp_\gamma$, where $r\in R^h,x\in X,\gamma\in \Gamma$ generate the $k$-algebra $R(X)\# \Gamma$ and hence $f$ is surjective. It remains to show that $f$ is injective.
 
 Let $w$ be a non-empty and finite word over the alphabet $\{rp_{\gamma}, xp_\gamma~(r\in R^h,x\in X,\gamma\in \Gamma)\}$. Then $w=y_1p_{\gamma_1}\dots y_np_{\gamma_n}$ for some $y_1,\dots, y_n\in R\cup X$ and $\gamma_1,\dots,\gamma_n\in \Gamma$. We call $w$ a {\it path (ending in $\gamma_n$)} if $\gamma_{i-1}-\gamma_i=\deg(y_i)~(i=2,\dots,n)$. 
 We denote the set of of all paths by $W$, and for any $\gamma\in\Gamma$ the set of all paths ending in $\gamma$ by $W_\gamma$. Define a map $\zeta:W\to R \langle X \rangle $ by $\zeta(y_1p_{\gamma_1}\dots y_np_{\gamma_n})=y_1\dots y_n$. 
 Now suppose that $f(a)=0$ for some $a\in R\#\Gamma \langle X\#\Gamma \rangle /J$. It follows from the presentation (\ref{2.2}) that $a=\sum_{w\in W}\lambda_{w}w$ where $\lambda_w\in k~(w\in W)$. Hence
 \begin{equation*}
 0=f(a)=f\Big(\sum_{w\in W}\lambda_{w}w\Big)=f\Big(\sum_{\gamma\in\Gamma}\sum_{w\in W_\gamma}\lambda_{w}w\Big)=\sum_{\gamma\in\Gamma}\sum_{w\in W_\gamma}\lambda_w\zeta(w)p_\gamma.
 \end{equation*}
 It follows that $\sum_{w\in W_\gamma}\lambda_w\zeta(w)=0$ in $R(X)$ for any $\gamma\in\Gamma$. One checks easily that this implies that $\sum_{w\in W_\gamma}\lambda_ww=0$ in $R\#\Gamma\langle X\#\Gamma \rangle /J$,  for any $\gamma\in\Gamma$. Hence $a=0$ and thus $f$ is injective.
 \end{proof}

 \begin{cor}\label{corsmashkey}
 Let $R$ be a $\Gamma$-graded $\K$-algebra where $\K$ is a commutative graded ring concentrated in degree $0$, $X$  a set  and $\deg:X\to\Gamma$ a map. Consider  $R \langle X \rangle$ as a $\Gamma$-graded $\K$-algebra with the induced grading. Let $I$ be a graded ideal of $R\langle X \rangle $. Then \[(R\langle X \rangle /I)\#\Gamma\cong R\#\Gamma \langle X\#\Gamma\rangle /J,\] where $J$ is the ideal of $R\#\Gamma\langle X\#\Gamma\rangle $ generated by the image of $I\#\Gamma$ and the set  
 \begin{equation}\label{hgbdgfuro}
  \big \{(xp_\gamma)(1p_\gamma)-xp_\gamma,~(1p_{\gamma+\deg(x)})(xp_\gamma)-xp_\gamma\mid x\in X,\gamma\in \Gamma\big \}.   
 \end{equation}
  \end{cor}
 \begin{proof}
 It follows from Proposition \ref{propsmashquot} that $(R\langle X \rangle /I)\#\Gamma\cong R\langle X\rangle \#\Gamma/I\#\Gamma$. By Proposition \ref{propsmashext}, there is an isomorphism $g:R\langle X\rangle \#\Gamma\to   R\#\Gamma\langle X\#\Gamma\rangle /J'$, where $J'$ is the ideal of $R\#\Gamma \langle X\#\Gamma\rangle $ generated by (\ref{hgbdgfuro}). Hence 
 \[(R\langle X \rangle /I)\#\Gamma\cong R\langle X\rangle \#\Gamma/I\#\Gamma\cong (R\#\Gamma    \langle X\#\Gamma\rangle/J')/g(I\#\Gamma).\]
 The assertion of the corollary follows.
 \end{proof}

\subsubsection{An isomorphism of categories}\label{subsubseccatiso}
 Recall that $R$-$\GR$ denotes the category of $\Gamma$-graded $R$-modules and $R$-$\GR_{\proj}$ the full subcategory of $R$-$\GR$ whose objects are the projective objects of $R$-$\GR$ that are finitely generated as a $R$-module. Moreover, $R\#\Gamma$-$\Mod$ denotes the category of unital $R\#\Gamma$-modules and $R\#\Gamma$-$\Mod_{\proj}$ the full subcategory of $R\#\Gamma$-$\Mod$ whose objects are the projective objects of $R\#\Gamma$-$\Mod$ that are finitely generated as a $R\#\Gamma$-module. In \cite{ara-hazrat-li-sims} it was shown that there is an isomorphism $\psi:R$-$\GR\to R\#\Gamma$-$\Mod$, and in \cite{Raimund2} it was shown that $\psi$ restricts to an isomorphism $R$-$\GR_{\proj}\to R\#\Gamma$-$\Mod_{\proj}$. The isomorphism $\psi$ maps an object $M$ of $R$-$\GR$ to the object $M\#\Gamma$ of $R\#\Gamma$-$\Mod$ which is defined as follows. As abelian group $M\#\Gamma=M$. The $R\#\Gamma$-action on $M\#\Gamma$ is defined by $(rp_\alpha)m=rm_\alpha$ for any $r\in R$, $\alpha\in \Gamma$ and $m\in M=M\#\Gamma$. On morphisms, $\psi$ is just the identity map.

\section{Leavitt path algebras}\label{lpalabel}

Let $\K$ be a field. Consider the symbols 
$X=(x_{ij})$ and $X^*=(x^*_{ji})$, where $1 \leq i \leq n$ and $1 \leq j \leq n+k$: 
\begin{equation} \label{breaktr}
X=\left( 
\begin{matrix} 
x_{11\phantom{n+{},}} & x_{12\phantom{n+{},}} & \dots & x_{1,n+k}\\ 
x_{21\phantom{n+{},}} & x_{22\phantom{n+{},}} & \dots & x_{2,n+k}\\ 
\vdots & \vdots & \ddots & \vdots\\ 
x_{n,1\phantom{n+{},}} & x_{n,2\phantom{n+{},}} & \dots & x_{n,n+k} 
\end{matrix} 
\right),  
\qquad
X^*=\left( 
\begin{matrix} 
x^*_{11\phantom{n+{},}} & x^*_{12\phantom{n+{},}} & \dots & x^*_{1,n\phantom{n+{},}}\\ 
x^*_{21\phantom{n+{},}} & x^*_{22\phantom{n+{},}} & \dots & x^*_{2,n\phantom{n+{},}}\\ 
\vdots & \vdots & \ddots & \vdots\\ 
x^*_{n+k,1} & x^*_{n+k,2} & \dots & x^*_{n+k,n} 
\end{matrix} 
\right).
\end{equation} 
The \emph{Leavitt algebra} $L_\K(n,n+k)$ is defined as 
\begin{equation}\label{leviialg}
L_\K(n,n+k)=\frac{\K\langle X, X^*\rangle}{\langle XX^*=I_{n}, \, X^*X=I_{n+k} \rangle }.
\end{equation}

Here $\K\langle X, X^*\rangle$ is the free $\K$-algebra generated by the symbols $x_{ij}$ and $x_{ji}^*$'s. Furthermore, the relations $XX^*=I_{n}$ and $X^*X=I_{n+k}$, where $I_n \in \M_n(\K)$ and $I_{n+k} \in \M_{n+k}(\K)$ are the identity matrices,  stand for the collection of relations after multiplying the matrices and comparing the two sides of the equations. 

These rings were first studied by Leavitt \cite{leavitt57} and \cite[p.130, footnote 6]{leavitt62}. Leavitt showed the type of the ring $L_\K(n,n+k)$ is $(n,n+k)$. Furthermore, $L_\K(n,n+k)$, for $n>1$ and $k\geq 1$, is a domain, whereas  $L_\K(1,k+1)$, is a simple ring.

If in (\ref{leviialg}) we only consider the relation $X^*X=I_{n+k}$, then we call the resulting algebra \emph{Cohn algebra}
\begin{equation}\label{cohnialg}
C_\K(n,n+k)=\frac{\K\langle X, X^*\rangle}{\langle  X^*X=I_{n+k} \rangle },
\end{equation}
which was considered by Cohn~\cite{cohn}, who showed that
 $C_\K(n,n+k)$, for $n>1$ and $k\geq 1$, is a domain, whereas  $C_\K(1,k+1)$ has only one non-trivial proper ideal generated by relation $XX^*=I_n$. This result also implies the simplicity of $L_\K(1,k+1)$.

The simple rings $L_\K(1,k+1)$ received significant attentions as they are the discrete versions of Cuntz algebras.   Modeled on this, Leavitt path algebras are introduced \cite{lpabook}, which attach to a directed graph a certain algebra. In the case that the graph has one vertex and $k+1$ loops, it recovers Leavitt algebra $L_\K(1,k+1)$. 

In order to recover the Leavitt algebras $L_\K(n,n+k)$, for $n>1$,  a plethora of generalisations of Leavitt path algebras were introduced, such as Leavitt path algebra of separated graphs, weighted graphs and ultra graphs. 
Here, following~\cite{Raimund2},  we introduce the Leavitt path algebras of hypergraphs. We show that Bergman algebras (see Section~\ref{bergmanalgebrasec}) with coefficients in a commutative semisimple ring are precisely Leavitt path algebras of hypergraphs (Theorem~\ref{lemhyperberg}). 

\subsection{Leavitt path algebras of hypergraphs} \label{hyperraim}

For sets $I$ and $X$, a function $x:I\rightarrow X$, $i\mapsto x_i=x(i)$ is called a {\it family of elements in $X$ indexed by $I$}. A family $x$ of elements in $X$ indexed by $I$ is usually denoted as $(x_i)_{i\in I}$. We call the family $(x_i)_{i\in I}$ {\it non-empty} if $I\neq\emptyset$.

A {\it (directed) hypergraph} is a quadruple $H=(H^0,H^1,s,r)$, where $H^0$ and $H^1$ are sets and $s$ and $r$ are maps associating to each $h\in H^1$ a non-empty family $s(h)=(s(h)_i)_{i\in I_{h}}$, respectively,  $r(h)=(r(h)_j)_{j\in J_{h}}$ of elements in $H^0$. 
The elements of $H^0$ are called {\it vertices} and the elements of $H^1$ {\it hyperedges}. 
In this paper all hypergraphs are assumed to be ${\it regular}$, i.e. for any hyperedge $h$ the sets $I_h$ and $J_h$ are finite. 


\begin{deff}\label{defhyper}
Let $H$ be a hypergraph. The not necessarily unital $\K$-algebra $L_\K(H)$ presented by the generating set 
\[\{v,h_{ij},h_{ij}^*\mid v\in H^0, h\in H^1, i\in I_h,  j\in J_h\}\]
and the relations
\begin{enumerate}[(i)]
\item $uv=\delta_{uv}u\quad(u,v\in H^0)$,

\item $s(h)_ih_{ij}=h_{ij}=h_{ij}r(h)_j,~r(h)_jh_{ij}^*=h_{ij}^*=h_{ij}^*s(h)_i\quad(h\in H^1, i\in I_h,  j\in J_h)$,

\item $\sum\limits_{j\in J_h}h_{ij}h_{i'j}^*= \delta_{ii'}s(h)_i\quad(h\in H^1, ~i,i'\in I_h)$ and

\item $\sum\limits_{i\in I_h}h_{ij}^*h_{ij'}= \delta_{jj'}r(h)_j\quad(h\in H^1,~j,j'\in J_h)$
\end{enumerate}
is called the {\it Leavitt path algebra of the hypergraph $H$}. We will sometimes call a Leavitt path algebra of a hypergraph a {\it hyper Leavitt path algebra}.
\end{deff}

It is not difficult to check that if $H$ is a hypergraph, then $L_\K(H)$ has a set of local units, namely the set of all finite sums of distinct elements
of $H^0$. Furthermore, $L_K (H)$ is a unital ring if and only if  $H^0$ is finite (\cite[Proposition~8]{Raimund2}).

\begin{rmk}\label{rmkhyperrel}
Let $H$ be a hypergraph. We can describe the relations (ii)-(iv) in Definition \ref{defhyper} using matrices as follows. For any hyperedge $h\in H^1$, let $e_h$ be the $I_h\times I_h$-matrix whose entry at position $(i,i')$ is $\delta_{ii'}s(h)_i$, and $f_h$ the $J_h\times J_h$-matrix whose entry at position $(j,j')$ is $\delta_{jj'}r(h)_j$. Moreover, we denote for any hyperedge $h\in H^1$ the $I_h\times J_h$-matrix whose entry at position $(i,j)$ is $h_{ij}$ also by $h$, and the $J_h\times I_h$-matrix whose entry at position $(j,i)$ is $h^*_{ij}$ by $h^*$. Then $L_\K(H)$ is the $\K$-algebra presented by the generating set $\{v,h_{ij},h_{ij}^*\mid v\in H^0, h\in H^1, i\in I_h,  j\in J_h\}$ and relations
\begin{enumerate}[(I)]
\item the elements of $H^0$ are pairwise orthogonal idempotents,

\item $ h=e_h h f_h,\quad  h^*=f_hh^*e_h\quad(h\in H^1)$,

\item $hh^*= e_h\quad(h\in H^1)$ and

\item $h^*h= f_h\quad(h\in H^1)$.
\end{enumerate}
\end{rmk}

\begin{rmk}\label{rmkhyperweight}
Let $H$ be a hypergraph and $\Gamma$ an abelian group. For any $h\in H^1$ choose an $i_h\in I_h$ and a $j_h\in J_h$. A map $w$ which associates to any generator $h_{ij}~(h\in H^1,i\in I_h,j\in J_h)$ an element from $\Gamma$, such that 
\[w(h_{ij})=w(h_{ij_h})-w(h_{i_hj_h})+w(h_{i_hj})\]
for any $h\in H^1$, $i\in I_h$ and $j\in J_h$ is called a {\it $\Gamma$-weight map} for $H$. By \cite[Lemma 70]{Raimund2}, a $\Gamma$-weight map $w$ for $H$ induces a $\Gamma$-grading on $L_\K(H)$ such that $\deg(v)=0$, $\deg(h_{ij})=w(h_{ij})$ and $\deg(h_{ij}^*)=-w(h_{ij})$ for any $v\in H^0$, $h\in H^1$, $i\in I_h$ and $j\in J_h$.  
\end{rmk}

\begin{example}[Leavitt path algebras]
\label{exmlpa}
Suppose that $E$ is a {\it directed graph}, i.e. a quadruple $E=(E^{0}, E^{1}, r, s)$, where $E^{0}$ and $E^{1}$ are
sets and $r,s$ are maps from $E^1$ to $E^0$. The elements of $E^0$ are called \textit{vertices} and the elements of $E^1$ \textit{edges}. We may think of an edge $e\in E^1$ as an arrow from $s(e)$ to $r(e)$. 
Assume that $E$ is \emph{row-finite}, i.e. for each vertex $v\in E^0$ there are at most finitely many edges in $s^{-1}(v)$. 
Denote by $E^{\reg}$ the set of {\it regular} vertices of the graph $E$, i.e. the set of all vertices $v$ for which $s^{-1}(v)$ is not the empty set. The $\K$-algebra $L_{\K}(E)$ presented by the generating set $\{v,e,e^*\mid v\in E^0,e\in E^1\}$ and the relations
\begin{enumerate}[(i)]
\item $uv=\delta_{uv}u\quad(u,v\in E^0)$,
\medskip
\item $s(e)e=e=er(e),~r(e)e^*=e^*=e^*s(e)\quad(e\in E^1)$,
\medskip
\item $\sum_{e\in s^{-1}(v)}ee^*= v\quad(v\in E^{\reg})$ and
\medskip
\item $e^*f= \delta_{ef}r(e)\quad(e,f\in E^1)$
\end{enumerate}
is called the {\it Leavitt path algebra} of $E$ (cf. for example \cite{lpabook}). For each $v\in E^{\reg}$ write $s^{-1}(v)=\{e^{v,j}\mid j\in J_v\}$. Define a hypergraph $H=(H^0,H^1,s',r')$ by $H^0=E^0$, $H^1=\{h^v\mid v\in E^{\reg}\}$, $s'(h^v)=(v)_{i\in\{1\}}$ and $r'(h^v)=(r(e^{v,j}))_{j\in J_v}$. There is a $\K$-algebra isomorphism $\phi:L_{\K}(E)\rightarrow L_{\K}(H)$ such that $\phi(u)=u$, $\phi(e^{v,j})=h^v_{1j}$ and $\phi((e^{v,j})^*)=(h^v_{1j})^*$ for any $u\in E^0$, $v\in E^{\reg}$ and $j\in J_v$.    
\end{example}

\begin{example}[Vertex-weighted Leavitt path algebras]
Suppose that $(E,w)$ is a \textit{vertex-weighted} graph, i.e. $E$ is a directed graph and $w$ is a map associating to each regular vertex a positive integer. We again assume that $E$ is row-finite. The $\K$-algebra $L_{\K}(E,w)$ presented by the generating set $\{v,e_i,e_i^*\mid v\in E^0, e\in E^1, 1\leq i\leq w(s(e))\}$ and the relations
\begin{enumerate}[(i)]
\item $uv=\delta_{uv}u\quad(u,v\in E^0)$,
\medskip
\item $s(e)e_i=e_i=e_ir(e),~r(e)e_i^*=e_i^*=e_i^*s(e)\quad(e\in E^1, 1\leq i\leq w(s(e)))$,
\medskip
\item 
$\sum_{e\in s^{-1}(v)}e_ie_j^*= \delta_{ij}v\quad(v\in E^{\reg};~1\leq i, j\leq w(v))$ and
\medskip 
\item $\sum_{1\leq i\leq w(v)}e_i^*f_i= \delta_{ef}r(e)\quad(v\in E^{\reg};~ e,f\in s^{-1}(v))$
\end{enumerate}
is called the {\it Leavitt path algebra} of $(E,w)$. For each $v\in E^{\reg}$ write $s^{-1}(v)=\{e^{v,j}\mid j\in J_v\}$. Define a hypergraph $H=(H^0,H^1,s',r')$ by $H^0=E^0$, $H^1=\{h^v\mid v\in E^0_{\reg}\}$, $s'(h^v)=(v)_{i\in \{1,\dots,w(v)\}}$ and $r'(h^v)=(r(e^{v,j}))_{j\in J_v}$. There is a $\K$-algebra isomorphism $\psi:L_{\K}(E,w)\rightarrow L_{\K}(H)$ such that $\psi(u)=u$, $\psi(e^{v,j}_i)=h^v_{ij}$ and $\psi((e^{v,j}_i)^*)=(h^v_{ij})^*$ for any $u\in E^0$, $v\in E^{\reg}$, $i\in \{1,\dots,w(v)\}$ and $j\in J_v$.  
\end{example}

\section{Bergman algebras}\label{bergmanalgebrasec}

 Let $R$ be a ring and $e\in \M_m(R)$ and $f\in \M_n(R)$ be idempotent matrices. Recall that $e$ and $f$ are called \emph{equivalent}, denoted $e\sim f$, if there are matrices $h\in \M_{m\times n}(R)$ and $h^* \in M_{n\times m}(R)$ such that $e=hh^*$ and $f=h^*h$.  

\begin{deff}\label{bergmanalgebra}
Let $R$ be a $\K$-algebra, where $\K$ is a commutative ring. Let $e\in \M_m(R)$ and   $f \in \M_n(R)$ be idempotent matrices. Let $h=(h_{ij})$ and $h^*=(h^*_{ji})$, where $1\leq i \leq m$ and $1\leq j \leq n$, be a collection of symbols. Then we define a chain of   \emph{Bergman algebras} 
\[B^1_{R}(e,f) \longrightarrow B^2_{R}(e,f) \longrightarrow B^3_{R}(e,f)\longrightarrow B_R(e,f),\]
as follows:
\begin{align*}
B^1_{R}(e,f)&:=  R\big\langle h\big\rangle \big /  \big\langle 
ehf =h \big\rangle\\
B^2_{R}(e,f) & := R \big\langle h, h^*\big\rangle \big /  \big\langle 
ehf =h, fh^*e =h^* \big\rangle\\
B^3_{R}(e,f) & := R\big\langle h, h^*\big\rangle\big /  \big\langle  ehf =h, fh^*e =h^*, hh^*=e \big\rangle
\end{align*}
and 
\begin{align}\label{berglagebra}
B_R(e,f) = B^4_{R}(e,f)& := R\big\langle h, h^*\big\rangle\big / \big \langle 
ehf =h, fh^*e =h^*, h h^* = e, h^*h=f \big\rangle.
\end{align}
\end{deff}
We further define, 
$$B_R(e)=R\big\langle h\big\rangle \big/ \big \langle 
ehe =h, h^2=h \big\rangle.$$

Here $R\langle h, h^*\rangle$ is the free ring generated by symbols $h_{ij}$ and $h^*_{ji}$'s with coefficients from the $\K$-algebra $R$. We assume the symbols commute with $\K$.  In fact $R\langle h, h^*\rangle$ is 
the coproduct  $R *_{\K} \K \langle h_{ij}, h^*_{ji}\mid 1\leq i \leq m, 1\leq j \leq n\rangle$. Furthermore, the relations such as  $ehf=h$  stand for the collection of relations after multiplying the matrices and comparing the two sides of the equation.  Clearly there is a canonical $\K$-algebra homomorphim $R\rightarrow B_R^i(e,f)$, where the images of $e$ and $f$ in $B_R(e,f)$ become equivalent. 

In this note, we particularly concentrate on the Bergman algebra $B_R(e,f)$ defined in (\ref{bergmanalgebra}), as many of combinatorial algebras can be realised as such Bergman algebras.

\begin{example}
We can recover classical Cohn and Leavitt algebras (\ref{cohnialg}) and  (\ref{leviialg}), respectively, as follows. For a field $\K$, and idempotent matrices $1\in \K$ and $I_n\in \M_n(\K)$, we have 
\begin{align*}
  B^1_\K(1, I_n)  &\cong \K \langle x_1,x_2,\dots,x_n \rangle \\
  B^2_\K(1, I_n)  &\cong \K \langle x_1,x_2,\dots,x_n, x_1^*,x_2^*,\dots,x_n^*\rangle \\
  B^3_\K(I_n, 1)  &\cong C_\K(1,n)\\
  B_\K(1, I_n) = B^4_\K(1, I_n)  &\cong L_\K(1,n),
\end{align*}
furthermore
\begin{align*}
B(I_n, I_{n+k}) = B^4(I_n, I_{n+k})&\cong L_\K(n,n+k)\\
B^3(I_{n+k}, I_{n})&\cong C_\K(n,n+k)
\end{align*}

\end{example}

\begin{rmk}
    In this note we use categories of left modules over rings. One can also use categories of right modules. In that case, one would have $B^3_\K(1, I_n) \cong C_\K(1,n)$ returning to right $R$-modules.
\end{rmk}

\begin{example}
This example demonstrates how rich the structure of seemingly easy to construct Bergman algebras of (\ref{berglagebra}) could be. Consider the Bergman algebra $B_{\M_2(\K)}(1,1)$. Here $1$ is the identity matrix of the $\K$-algebra $\M_2(\K)$.   One can immediately see that 
\[B_{\M_2(\K)}(1,1) \cong \M_2(\K) *_\K \K[x,x^{-1}].
\]
Furthermore, the discussion in Section~4 of \cite{bergman74} shows that the Leavitt algebra $L_\K(2,2)$ coincides with a corner of this Bergman algebra, specifically, 
\[L_\K(2,2) \cong e_{11} B_{\M_2(\K)}(1,1) e_{11}.
\]
\end{example}

\begin{lemma}\label{bergequiiso}
Let $R$ be a $\K$-algebra, where $\K$ is a commutative ring. Let $e\in \M_m(R)$ and   $f \in \M_n(R)$ be idempotent matrices. If idempotent matrices $e'\in \M_k(R)$ and   $f' \in \M_l(R)$ are equivalent to $e$ and $f$, respectively, then $B_R^i(e,f)$ is $\K$-algebra isomorphic to $B_R^i(e',f')$, where $1\leq i \leq 4$.
\end{lemma}
\begin{proof}
   We show that if $e\sim e'$ then $B_R^i(e,f) \cong B_R^i(e',f)$. The general result follows by symmetry. 
   
   Since $e\sim e'$, there are matrices $x\in \M_{m\times k}(R)$, $y\in \M_{k\times m}(R)$ such that
$e=xy$, $e'=yx$. 
    Define the map $\phi:B_R^i(e,f)\rightarrow B_R^i(e',f)$ by sending $h \mapsto x h'$ and $h^*$$ 
    \mapsto  {h^*}'y$ (if $h^*$ is present in the definition of $B_R^i(e,f)$). The converse map $\psi:B_R^i(e',f)\rightarrow B_R^i(e,f)$ is defined by sending $h' \mapsto  yh$ and ${h^*}'  \mapsto h^* x$. We leave it to the reader to check these assignments respect the defining relations and give a $\K$-algebra isomorphism between the two algebras. 
\end{proof}

\begin{example}

We note that if the idempotents $e\in \M_m(R)$ and $f \in \M_n(R)$ are equivalent, i.e., $e=hh^*$ and $f=h^*h$, then by considering $h_1= ehf$ and $h_1^*=fhe$ we have $e=h_1h_1^*$ and $f=h_1^*h_1$, with, $eh_1f=h_1$ and $fh_1^*e=h_1^*$. However the algebras 
\[
B'(e,f):= R\big\langle h, h^*\big \rangle \big /  \big \langle 
h h^* = e, h^*h=f \big \rangle.
\]
and 
\[B_R(e,f) = R\big \langle h, h^*\big \rangle \big /  \big \langle 
ehf =h, fh^*e =h^*, h h^* = e, h^*h=f \big \rangle.
\] are not necessarily isomorphic, as the following example shows. 

Suppose that $\K$ is a field (or more generally an integral domain). Let $e=f\in \M_2(\K)$ be the idempotent matrix which has a one at position $(1,1)$ and zeros elsewhere (note that the projective $\K$-module $P$ defined by $e$ is just $\K$, the free $\K$-module of rank 1).
 
Let $R$ be the free $\K$-algebra $\K \langle h,h^* \rangle$ subject to the relations $\{ h h^*=e, h^* h=e, ehe=h, eh^*e=h^* \}$. It follows from the relations $ehe=h$, $eh^*e=h^*$ that all entries of $h$ and $h^*$ that are not at position $(1,1)$ are zero. So $R$ is just the Laurent polynomial ring $\K[x,x^{-1}]$.
 
Let $S$ be the free $\K$-algebra $\K \langle h,h^* \rangle$ subject to the relations  $\{ h h^*=e, h^* h=e \}$. Since $$e=hh^*=hh^*hh^*=heh^*=heeh^*$$ implies $e=eheeh^*e,$ it follows that $h_{11}h^*_{11}=1$. But we also have
 $$h_{11}h^*_{11}+h_{12}h^*_{21}=(h h^*)_{11}=e_{11}=1$$
 and hence $h_{12}h^*_{21}=0$. If both $h_{12}$ and $h^*_{21}$ are non-zero in $S$, then the ring $S$ has zero divisors, while $R=\K[x,x^{-1}]$ has not. Suppose now that one of the entries $h_{12}$ or $h^*_{21}$ is zero in $S$. Clearly
 $$h^*_{21}h_{12}+h^*_{22}h_{22}=(h^*h)_{22}=e_{22}=0$$
 and hence $h^*_{22}h_{22}=0$ since $h_{12}=0$ or $h^*_{21}=0$. We show that  $h^*_{22}$ and $h_{22}$ are non-zero in $S$, which implies that $S$ has zero divisors.
 
Let $R'$ be the ring obtained from $R=\K[x,x^{-1}]$ by adjoining two elements $y$ and $z$ subject to $yz=0$ and $zy=0$. Then $y$ and $z$ are non-zero in $R'$. Since $$\diag(x,y)\diag(x^{-1},z)=\diag(1,0)=\diag(x^{-1},z)\diag(x,y),$$
there is a ring homomorphism $S\rightarrow R'$ mapping any $r\in R$ to itself, and
 $$h_{11} \mapsto x, h^*_{11} \mapsto x^{-1}, h_{22} \mapsto y,  h^*_{22} \mapsto z,$$
 and the other entries of $h$ and $h^*$ to zero. Since the image of $h_{22}$ under this homomorphism is non-zero, $h_{22}$ must be non-zero as well. Similarly it follows that $h^*_{22}$ is non-zero. 
 Thus $S$ has zero divisors while $R$ has not.
 \end{example}

The rings introduced in Definition~\ref{bergmanalgebra}, come directly from Bergman's machinery of universal constructions, by working with idempotent matrices instead of projective modules. Let $P$ and $Q$ be non-zero finitely generated projective $R$-modules. In~\cite{bergman74}, Bergman constructed a ring $$S_1=R\big\langle h:\overline{P}\rightarrow \overline{Q}\big\rangle$$  
such that there is a ``universal'' $S_1$-module homomorphism 
$h: S_1 \otimes_R P \rightarrow S_1 \otimes_R Q$. He further constructed universal algebras 
$$S_2=R\big\langle h, h^*:\overline{P}\rightarrow \overline{Q}, \overline{Q}\rightarrow \overline{P}\big\rangle$$ and 
$$S_3=R\big\langle h, h^*:\overline{P}\rightarrow \overline{Q}, \overline{Q}\rightarrow \overline{P}, h^*h=1\big\rangle$$ and finally 
$$S = S_4=R\big\langle h, h^*:\overline{P}\cong \overline{Q}\big\rangle.$$ Thus, over $S$ we have $S\otimes_R P \cong S\otimes_R Q$. 

 Representing $P$ and $Q$ as idempotent matrices $e$ and $f$, respectively, in \S\ref{ghostagain} we will establish that the  rings $S_i$,  above coincide with the Bergman algebras $B^i_{R}(e,f)$, $1\leq i \leq 4$, in Definition~\ref{bergmanalgebra}, which explains our terminology. In fact in \S\ref{ghostagain} we will carry out the graded version of these constructions. By setting the grade group trivial we obtain the above statements. We note that the paper \cite{bergman74} works with right modules whereas we work here with left modules.

We can extend the definition of Bergman algebras to families of pairs of idempotents. Let 
$(e,f) :=\{(e_i,f_i)\}_{i\in I}$, be a collection of pairs of idempotents. Then we define 
\begin{equation}\label{bergcolec}
B_R(e,f) := R\big\langle h_i, h_i^*\mid i\in I\big\rangle \big/  \big\langle 
e_ih_if_i =h_i, 
f_ih_i^*e_i =h_i^*, 
h_i h_i^* = e_i, h_i^*h_i=f_i\mid i\in I\big\rangle.
\end{equation}

Similarly we define $B^i_R(e,f)$, for $1\leq i \leq 3$ as well as $B_R(e)$. 

\begin{example}
Let $(e,f)$ be the collection $\{(1,1),(1,1)\}$. Then 
$B_\K(e,f)=\K\langle x,x^{-1},y,y^{-1} \rangle ,$ the (non-commutative) Laurent polynomial ring with two variables. 
\end{example}

Next we define the graded version of the Bergman algebras of Definition~\ref{berglagebra}.

\begin{deff}\label{berggralgebra2}
Let $R$ be a $\Gamma$-graded $\K$-algebra, where $\K$ is a $\Gamma$-graded commutative ring. Let $e\in \M_m(R)(\overline{\beta})_0$ and   $f \in \M_n(R)(\overline \gamma)_0$ be idempotent matrices, where $\overline \beta=(\beta_1, \dots, \beta_m)$ and  $\overline \gamma=(\gamma_1, \dots, \gamma_n)$. Furthermore, let 
$h=(h_{ij})$ and 
$h^*=(h^*_{ji})$, where $1\leq i \leq m$ and $1\leq j \leq n$,  be a collection of symbols. To the symbols 
 $h_{ij}$ assign the degree $\gamma_j-\beta_i$ and $h^*_{ji}$ the degree $\beta_i-\gamma_j$ for $1\leq i \leq m$ and $1\leq j \leq n$. Then the Bergman algebra $B_R^i(e,f)$ defined in Definition~\ref{bergmanalgebra} form $\Gamma$-graded $\K$-algebras. For
$(e,f) :=\{(e_i,f_i)\}_{i\in I}$, a collection of pairs of homogeneous idempotents, the graded Bergman algebra $B^i_R(e,f)$ is defined similarly as in (\ref{bergcolec})
\end{deff}

It has already been established in the literature that Leavitt path algebras can be realised as Bergman algebras (using the module theoretical language). This identification was then used to describe the non-stable $K$-theory of Leavitt path algebras (Theorem~3.2.5 in~\cite{lpabook}).

\begin{lemma}\label{lemprehyperberg}
Suppose that $R= \Pi_{t\in T}\K$, where $\K$ is field and $T$ some finite set. For any $t\in T$, let $\epsilon_t$ be the element of $R$ whose $t$-component is $1$ and whose other components are $0$. Then any idempotent matrix with entries in $R$ is equivalent to a matrix of the form $\diag(\epsilon_{t_{1}},\dots,\epsilon_{t_{n}})$, where $t_1,\dots,t_n\in T$.
\end{lemma}
\begin{proof}
Recall that there is a $1-1$ correspondence between equivalence classes of idempotent matrices over $R$ and isomorphism classes of finitely generated projective $R$-modules. For any $t\in T$, let $P_t=R\epsilon_t$. It is well-known that any finitely generated projective $R$-module $P$ is isomorphic to a direct sum $P=P_{t_1}\oplus\dots\oplus P_{t_n}$, where $t_1,\dots,t_n\in T$. It follows that the corresponding idempotent matrix equals (up to equivalence) $(\epsilon_{t_{1}})\oplus \dots\oplus (\epsilon_{t_{n}})=\diag(\epsilon_{t_{1}},\dots,\epsilon_{t_{n}})$.
\end{proof}
    
The following theorem shows that if the $\K$-algebra $R$ is  commutative $\K$-semisimple, then the Bergman algebras $B_R(e,f)$ of (\ref{bergcolec}) coincide with the class of hyper Leavitt path algebras of Section~\ref{hyperraim}.  

\begin{thm}\label{lemhyperberg}
Let $\K$ be a field, $\Gamma$ an abelian group and $R$ a commutative semisimple $\K$-algebra considered as $\Gamma$-graded algebra concentrated in degree zero. Let $(e,f)$ be a family of pairs of homogeneous idempotents. Then the Bergman algebra $B_R(e,f)$ as defined in {\upshape (\ref{bergcolec})} is $\Gamma$-graded isomorphic to a unital hyper Leavitt path algebra $L_\K(H)$ whose grading is induced by a weight map. Conversely, any unital hyper Leavitt path algebra $L_\K(H)$ whose grading is induced by a weight map is graded isomorphic to a Bergman algebra $B_R(e,f)$, where $R$ is a commutative semisimple $\K$-algebra concentrated in degree zero. 
\end{thm}
\begin{proof}
 First we show that if $R$ is a $\Gamma$-graded commutative semisimple $\K$-algebra concentrated in degree zero, then $B_R(e,f)$ is graded isomorphic to a unital hyper Leavitt path algebra $L_\K(H)$ whose grading is induced by a weight map. Write $R= \Pi_{t\in T}\K$, where $T$ is some finite set. For any $t\in T$ we let $\epsilon_t$ be the element of $R$ whose $t$-component is $1$ and whose other components are $0$. Moreover, write $(e,f)=\{(e_u,f_u)\}_{u\in U}$, where for any $u\in U$, $e_u\in \Mat_{m_u}(R)(\overline \alpha_u)_0$ and $f_u\in \Mat_{n_u}(R)(\overline \beta_u)_0$, $\overline \alpha_u \in \Gamma^{m_u}$ and $\overline \beta_u \in \Gamma^{n_u}$, are homogeneous idempotent matrices. Since $R= \Pi_{t\in T}\K$, by Lemma~\ref{bergequiiso} and \ref{lemprehyperberg}, we may assume that for any $u\in U$, $e_u=\diag(\epsilon_{t_{u,1}},\dots,\epsilon_{t_{u,m_u}})$ and $f_u=\diag(\epsilon_{t'_{u,1}},\dots,\epsilon_{t'_{u,n_u}})$, for some $t_{u,1},\dots, t_{u,m_u}, t'_{u,1},\dots, t'_{u,n_u}\in T$. Define a hypergraph $H=(H^0, H^1, s,r)$ by $H^0=\{v_t\mid t\in T\}$, $H^1=\{h_u\mid u\in U\}$, $I_{h_u}=\{1,\dots,m_u\}~(h_u\in H^1)$, $J_{h_u}=\{1,\dots,n_u\}~(h_u\in H^1)$, $s(h_u)_i=v_{t_{u,i}}~(h_u\in H^1, i\in I_{h_u})$ and $r(u)_j=v_{t'_{u,j}}~(h_u\in H^1, j\in J_{h_u})$. It follows from Remark \ref{rmkhyperrel} that $B_R(e,f)\cong L_\K(H)$ as $\K$-algebras. Since $H^0=T$ is finite, $L_\K(H)$ is unital (see \cite[Proposition 8]{Raimund2}). Now define a $\Gamma$-weight map $w$ for $H$ by $w((h_u)_{ij})=(\beta_u)_j-(\alpha_u)_i$ for any $h_u\in H^1$, $i\in I_{h_u}$ and $j\in J_{h_u}$. One checks easily that $w$ is indeed a $\Gamma$-weight map for $H$, and that the isomorphism $B_R(e,f)\cong L_\K(H)$ is graded with respect to the $\Gamma$-grading on $L_\K(H)$ induced by $w$.

 Now let $L_\K(H)$ be a unital hyper Leavitt path algebra whose grading is induced by a $\Gamma$-weight map $w$. Since $L_\K(H)$ is unital, $H^0$ is finite. Let $R:= \Pi_{v\in H^0}\K$. For any $v\in H^0$ we let $\epsilon_v$ be the element of $R$ whose $v$-component is $1$ and whose other components are $0$. We may assume that for any $h\in H^1$, $I_h=\{1,\dots, m_h\}$ and $J_h=\{1,\dots,n_h\}$, and moreover that $i_h=1=j_h$. Since $w$ is a $\Gamma$-weight map for $H$, we have $w(h_{ij})=w(h_{i1})-w(h_{11})+w(h_{1j})$ for any $h\in H^1$. For any $h\in H^1$, $1\leq i\leq m_h$ and $1\leq j\leq n_h$ set $\alpha_{h,i}:=-w(h_{i1})+w(h_{11})$ and $\beta_{h,j}:=w(h_{1j})$. Moreover, set $\overline \alpha_h:=(\alpha_{h,1},\dots, \alpha_{h,m_h})$ and $\overline \beta_h:=(\beta_{h,1},\dots, \beta_{h,n_h})$. For any $h\in H^1$ let $e_h\in \Mat_{m_h}(R)(\overline \alpha_h)_0$ be the idempotent matrix whose entry at position $(i,i')$ is $\delta_{ii'}\epsilon_{s(h)_i}$, and $f_h\in \Mat_{n_h}(R)(\overline \beta_h)_0$ the idempotent matrix whose entry at position $(j,j')$ is $\delta_{jj'}\epsilon_{r(h)_j}$. It follows from Remark \ref{rmkhyperrel} that $L_\K(H)\cong B_R(e,f)$ as $\K$-algebras, where $(e,f)=\{(e_h,f_h)\}_{h\in H^1}$. One checks easily that this isomorphism is graded as well.
\end{proof}

\section{Graded Bergman universal construction}\label{bergconscore}

In this section we start by extending Bergman's work to the graded setting. We will see that if one wants to realise a conical $\Gamma$-monoid as a non-stable $K$-theory of a ring, that ring needs to be $\Gamma$-graded and the type of $K$-theory to be considered should be the graded $K$-theory. 

Throughout this section, $\Gamma$ is an abelian group, $\K$ a commutative $\Gamma$-graded ring and $R$ a $\Gamma$-graded $\K$-algebra, i.e., $R$ is equipped with a graded homomorphism $\K\rightarrow Z(A)$.

\subsection{Graded universal morphisms}\label{bergconst}

We start with a simple lemma which will be used throughout the paper.  

\begin{lemma}
\label{lemma-resolution}
Let $F, M$ and $X, Y$ be $\G$-graded $R$-modules with the exact sequence of $\G$-graded $R$-modules \[\xymatrix{Y\ar[r]^{u} &X \ar[r]^{\varepsilon} &M \ar[r] &0,} \] and a graded idempotent endomorphism  $e: F\xrightarrow{} F$ (i.e., $e\circ e=e$). Set  $P := {\rm Im}\, e$. Then for any graded module homomorphism $g: M\xrightarrow{} P$, there is a unique graded module homomorphism $h:X\xrightarrow{} F$ in the following diagram   \[\xymatrix{Y\ar[r]^{u}& X \ar[r]^{\varepsilon} \ar[d]^{h}&M\ar[r]\ar[d]^g &0\\ & F \ar[r]^{e} \ar[d]^{1-e}& P \ar[r]&0\\ & F&& }\] satisfying $h\circ u=0$, $(1-e)\circ h=0$ and $e\circ h=g\circ \varepsilon$. On the other hand, for any graded module homomorphism $h:X\xrightarrow{} F$ with $h\circ u=0=(1-e)\circ h$, there is a unique graded module homomorphism $g:M\xrightarrow{} P$ such that $e\circ h=g\circ \varepsilon$. Furthermore, we have that $g: M\xrightarrow[]{} P$ is the zero map if and only if the corresponding $h:X\xrightarrow[]{} F$ is the zero map.

\end{lemma}

\begin{proof} For the idempotent map $e:F\xrightarrow{} F$, it is a fact that 
\begin{align*}
    \begin{cases}\Ker e={\rm Im} (1-e), &\\
        {\rm Im}\,e=\Ker (1-e),&\\
        F=\Ker e\oplus {\rm Im} \,e.&
    \end{cases}
\end{align*} Define $h(x)=(g\circ \varepsilon)(x)$ for any $x\in X$. As $e(p)=p$ for any $p\in P$, we have $e\circ h=h=g\circ\varepsilon$. Hence $(1-e)\circ h=(1-e)\circ e\circ h=0$ and $h\circ u=g\circ\varepsilon\circ u= 0$. For the uniqueness of $h$, suppose that there exists $h':X\xrightarrow{} F$ such that $e\circ h'=g\circ \varepsilon$, $h'\circ u=0$ and $(1-e)\circ h'=0$. Then $h'(x)=h'(x)-((1-e)\circ h')(x)=e\circ h')(x)=(g\circ \varepsilon) (x)=h(x)$ for any $x\in X$. Therefore $h'=h$.  One observes that if $g: M\xrightarrow[]{} P$ is the zero map, one takes $h$ to be the zero map which is the desired map. Obviously if $h:X\xrightarrow[]{} F$ is the zero map, then $g=0$. 

On the other hand, when $h\circ u=0$ by the universal property of cokernel there exists $\tau: M\xrightarrow{} F$ such that $h=\tau \circ \varepsilon$. Set $g=e\circ\tau$. Then $e\circ h=g\circ\varepsilon$. For the uniqueness of $g$, suppose there exists another $g':M\xra{}P$ satisfying $e\circ h=g'\circ \varepsilon$. Then we have $g\circ\varepsilon=e\circ h=g'\circ\varepsilon$. Since $\varepsilon$ is surjective, we have $g=g'$. 
\end{proof}

We are in position to extend two crucial theorems of Bergman on universal ring constructions~\cite[Theorems~3.1 and~3.2]{bergman74} to the setting of graded rings.

\begin{thm} \label{thm-1} Let $R$ be a $\Gamma$-graded $\K$-algebra. Suppose that $M$ is a $\Gamma$-graded $R$-module and $P$ a graded finitely generated projective $R$-module. Then there exists a $\Gamma$-graded $R\text{-ring}_{\K}$, $S$, with a universal graded module homomorphism $f: S\otimes_R M \rightarrow{} S\otimes_R P$; that is, given any $\G$-graded $R\text{-ring}_{\K}$, $T$, and any graded $T$-module homomorphism $g:T\otimes_R M\rightarrow{} T\otimes_R P$, there exists a unique graded homomorphism $S\rightarrow{} T$ of $R\text{-rings}_{\K}$ such that $g=T\otimes_S f$. 
\end{thm}

\begin{proof}  We first write $M$ as a cokernel of the graded $R$-module homomorphism $u:\bigoplus_{i\in I}R(\a_i)\xrightarrow{} 
 \bigoplus_{j\in J} R(\b_j)$  of graded free $R$-modules;  that is $$\bigoplus_{i\in I}R(\a_i)\stackrel{u}{\longrightarrow}\bigoplus_{j\in J} R(\b_j)
\stackrel{\varepsilon}{\longrightarrow} M \longrightarrow 0,$$ is an exact sequence of graded $R$-modules. Using the matrix representation of  \S\ref{matrixrep}, we have $$u=(u_{ij})_{i\in I, j\in J}\in \mathbb M_{I\times J}(R)[\overline{\a}][\overline{\b}].$$  Here $\overline{\a}=(\a_i)_{i\in I}$ and $\overline{\b}=(\b_j)_{j\in J}$. 
We also write $P$ as a direct summand of a graded free $R$-module $\bigoplus_{m\in K} R(\g_m)$ (see~(\ref{projidem1})). Hence, $P$ can be written as the image of a graded idempotent endomorphism $e: \bigoplus_{m\in K} R(\g_m) \xrightarrow{} \bigoplus_{m\in K} R(\g_m)$.  Similarly, by~\S\ref{matrixrep} we have $1-e=(v_{mm'})_{m,m'\in K}\in \mathbb M_{K\times K}(R)[\overline{\g}][\overline{\g}] $ with $\overline{\g}=(\g_m)_{m\in K}$.

For any $\G$-graded $R\text{-ring}_{\K}$ $T$ and any graded module homomorphism $g:T\otimes_R M \xrightarrow{} T\otimes_R P$, we have the following diagram with the first two rows exact
\begin{equation}
	\label{diagram}
\xymatrix{T\otimes_R (\bigoplus_{i\in I}R(\a_i))\ar[r]^{\overline{u}}& T\otimes_R (\bigoplus_{j\in J} R(\b_j))\ar[r]^>>>>>{\overline{\varepsilon}} \ar[d]^{h}&T\otimes_R M\ar[r]\ar[d]^{g} &0\\ & T\otimes_R (\bigoplus_{m\in K}R(\g_m)) \ar[r]^>>>>>{\overline{e}} \ar[d]^{\overline{1-e}}& T\otimes_R P\ar[r]&0\\ & T\otimes_R (\bigoplus_{m\in K}R(\g_m))&& },
\end{equation} where by Lemma \ref{lemma-resolution}, $g$ is uniquely determined by a graded module homomorphism $h$ with $h\circ\overline{u}=0=\overline{1-e}\circ h$.
But in matrix form, $$h=(h_{jm})_{j\in J, m\in K}\in \mathbb M_{J\times K}(T)[\overline{\b}][\overline{\g}],$$ with $h_{jm}\in T_{\g_m-\b_j}$. 
The condition $h\circ \overline{u}=0$  is equivalent to
\begin{equation}\label{matrixrelation-one}(u_{ij})_{i\in I, j\in J}\cdot (h_{jk})_{j\in J, m\in K}=0
\end{equation}  in $\mathbb M_{I\times K}(T)[\overline{\a}][\overline{\g}]$, that is,  \begin{equation}\label{relation-one}
\sum_{j\in J}u_{ij}h_{jm}=0 
\end{equation} for $i\in I, m\in K$.  Similarly, the condition  $\overline{1-e}\circ h=0$ is equivalent to 
\begin{equation}
\label{matrixrelation-two}(h_{jm})_{j\in J, m\in K}\cdot (v_{mm'})_{m,m'\in K} =0\end{equation} in $\mathbb M_{J\times K}(T)[\overline{\b}][\overline{\g}]$, that is 
 \begin{equation}
 \label{relation-two}
 \sum_{m\in K}h_{jm}v_{mm'}=0	
 \end{equation} for $j\in J, m'\in K$.


Let us define $S_0$ as the $R\text{-ring}_{\K}$ obtained by adjoining generators $h_{jm}$, for $j\in J, m\in K$ to the ring $R$.  Set the degrees of $h_{jm}$ to $\g_m-\b_j$, where $j\in J$ and $m\in K$. Then $S_0$ is a $\G$-graded $R\text{-ring}_{\K}$. Define $S$ to be the quotient of $S_0$ subject to the two relations \eqref{relation-one} and \eqref{relation-two} which are homogeneous of degree $\gamma_m - \alpha_i$ and $\gamma_{m'}-\beta_j$, respectively, in the graded ring $S_0$. Then $S$ is the desired $\G$-graded $R\text{-ring}_{\K}$.
\end{proof}

\begin{rmk}
\label{rmkthm}
\begin{enumerate}[\upshape(1)]
\item  The $\G$-graded $R\text{-ring}_{\K}$ $S$ is independent of the choices of resolutions for $M$ and $P$. One observes that $S$ has the universal property for fixed resolutions of $M$ and $P$. 
\item We prove that applying the forgetful functor $\mathcal{F}$ to the $\G$-graded ring $S$ with a universal graded module homomorphism in Theorem \ref{thm-1} we have the ring in Bergman's Theorem 3.1 \cite{bergman74}. Actually when we apply the forgetful functor $\mathcal{F}$ to the graded projective resolution $\bigoplus_{i\in I}R(\a_i)\xrightarrow{u} 
 \bigoplus_{j\in J} R(\b_j)
\xrightarrow{\varepsilon} M \xrightarrow{} 0$, we have the non-graded projective resolution $\bigoplus_{i\in I}R\xrightarrow{u} 
 \bigoplus_{j\in J} R
\xrightarrow{\varepsilon} M \xrightarrow{} 0$. Similarly if we apply the forgetful functor $\mathcal{F}$ to the idempotent endomorphism $e: \bigoplus_{m\in K} R(\g_m) \xrightarrow{} \bigoplus_{m\in K} R(\g_m) $, then we have the idempotent endomorphism $e: \bigoplus_{K} R\xrightarrow{} \bigoplus_{K}R $. 

Given us any $R\text{-}\text{ring}_{\K}$ $T'$ and any module homomorphism $g':T'\otimes_R M\xrightarrow{} T'\otimes_R P$,  
\begin{equation}
	\label{dd}
\xymatrix{T'\otimes_R(\bigoplus_{i\in I}R)\ar[r]^{\overline{u}}& T'\otimes_R (\bigoplus_{j\in J} R) \ar[r]^>>>>>{\overline{\varepsilon}} \ar[d]^{h'}&T'\otimes_R M\ar[r]\ar[d]^{g'} &0\\ & T'\otimes_R (\bigoplus_{K}R)\ar[r]^>>>>>{\overline{e}} \ar[d]^{\overline{1-e}}& T'\otimes_R P \ar[r]&0\\ & T'\otimes_R (\bigoplus_{K}R)&& },
\end{equation} where $h'$ is uniquely determined by $g'$; see Lemma \ref{lemma-resolution}. This $h'$ gives us a ring homomorphism $S\xrightarrow{} T'$ . Hence $S$ satisfies the universal property as the ring in Bergman's Theorem 3.1. Therefore $S$ is the ring in Bergman's Theorem 3.1.

\item[(3)] The desired $\G$-graded $\K$-algebra $S$ in Theorem \ref{thm-1} has the graded $\K$-homomorphism $R\xra S$ which sends $r$ to $\overline{r}$ for $r\in R$.
\end{enumerate}
\end{rmk}

\begin{thm}
\label{thm-2}
Let $R$ be a $\G$-graded $\K$-algebra, $M$ a graded $R$-module, $P$ a graded projective $R$-module, and $f:M\xrightarrow{} P$ a graded module homomorphism.  Then there exists a $\G$-graded $R\text{-ring}_{\K}$ $S$ such that $S\otimes_R f=0$, and $S$ is universal for this property: Given any $\G$-graded $R\text{-ring}_{\K}$ $T$ with $T\otimes_R f=0$, there exists a unique graded homomorphism of $R\text{-rings}_{\K}$, $S\xrightarrow{} T$.
\end{thm}

\begin{proof}
We write $P$ as a direct summand of a graded free module $\bigoplus_{m\in K} R(\g_m)$. Then for any graded $R\text{-ring}_{\K}$ $T$, the map $\overline{i}: T\otimes_R P\xrightarrow{}\bigoplus_{m\in K}T(\g_m)$ induced by the inclusion $i: P\xrightarrow{} \bigoplus_{m\in K}R(\g_m)$ is again an inclusion. Hence  for each homogeneous elements $x\in M$, we write $$f(x)=\sum_{m\in K} f_{m x}$$ where $f_{m x}$ is a homogeneous element in $R(\g_m)$. Observe that for any $\G$-graded $R\text{-ring}_{\K}$ $T$, $T\otimes_R f=0$ if and only if for any homogeneous elements $x\in \bigcup_{\a\in \G} M_{\a}$ all the $f_{m x}$ is zero in $T$. So the desired universal ring $S$ is  the quotient of $R$ by the two-sided ideal generated by these homogeneous elements $f_{m x}$.   \end{proof}

The ring constructed in Theorem~\ref{thm-1} will be denoted by $R\langle f:\overline{M}\xrightarrow{} \overline{P}\rangle$ (or $R\langle \overline{M}\xrightarrow{} \overline{P}\rangle$), and that of Theorem~\ref{thm-2} by $R\langle \overline{f}=0\rangle$. Note that the graded $\K$-homomorphism from $R$ to $R\langle \overline{f}=0\rangle$ in Theorem \ref{thm-2} is a canonical projection.

\subsubsection{}\label{bgdhyu}

Combining Theorems~\ref{thm-1} and~\ref{thm-2} we can construct graded rings with rich structures. 
Let $f:P \rightarrow Q$ be a module homomorphism  between graded finitely generated projective 
$R$-modules $P$ and $Q$. By Theorem~\ref{thm-1} we construct a graded ring with a homomorphism $g: \overline Q\rightarrow \overline P$ and then by two applications of Theorem~\ref{thm-2} we set $1-g \overline f$ and $1-\overline fg$ to zero. We denote the graded ring obtained by $R\langle {\overline f}^{-1}\rangle$ where the homomorphism $\overline f$ is now invertible.

In a similar fashion, for graded finitely generated  projective $R$-modules $P$ and $Q$, we can adjoin a universal graded isomorphism between $\overline{P}$ and $\overline{Q}$ by first freely adjoining a graded map $\overline{P}\xrightarrow[]{}\overline{Q}$, then adjoining an inverse. We denote the resulting $\Gamma$-graded $R\text{-ring}_{\K}$ by $$R\langle i, i^{-1} :\overline{P}\cong_{\gr} \overline{Q}\rangle.$$ 

Given a single graded finitely generated projective module $P$, one may obtain a $\Gamma$-graded $R\text{-ring}_{\K}$ by adjoining to $R$ a universal idempotent graded endomorphism of $\overline{P}$. This ring will be denoted by $$R\langle i :\overline{P}\xra \overline{P}; i^2=i\rangle.$$

\subsection{Universal morphisms for strongly graded rings}
Let $\K$ be a commutative $\Gamma$-graded ring and $R$ a $\G$-graded $\K$-algebra. As mentioned in the introduction, in certain sections, such as \S\ref{smashbergloc} and \S\ref{gdhdhdh} we need to assume that the graded ring $\K$ is concentrated in degree zero. Here we also need to make this assumption, namely, $R$ is a $\G$-graded $\K$-algebra, where $\K$ is concentrated in degree zero. Note that, by our assumption $\K \subseteq R_0$.
Let $M$ be a $\Gamma$-graded $R$-module and $P$ a $\Gamma$-graded finitely generated projective $R$-module. It is natural to compare the Bergman algebras $R\langle f:\overline{M}\xra \overline{P} \rangle$ and $R_0\langle f_0: \overline{M_0}\xra \overline{P_0}\rangle$, in case $P_0$ is a finitely generated projective $R_0$-module.  We will show that if $R$ is strongly graded, then there is a $\Gamma$-graded isomorphism of algebras, 
 \[R\langle f:\overline{M}\longrightarrow \overline{P} \rangle \cong  R \ast_{R_0} R_0\langle f_0: \overline{M_0}\longrightarrow \overline{P_0}\rangle.\] This ties the two representable functors~(\ref{grbergmp}) and (\ref{grbergmp0}) together.

We first consider the grading for coproducts of two $\Gamma$-graded $\K$-algebras. Let $R$ be a $\G$-graded $\K$-algebra and $S$ an $R_0\text{-ring}_{\K}$ with $\Gamma$-grading. Then the coproduct $R\ast_{R_0} S$ has a $\Gamma$-grading. The coproduct is, as a graded vector space, the space of all words in $R$ and $S$, and the homogeneous elements of pure degree are words where the letters are homogeneous, and the degree is given by the sum of the degrees. One observes that when $S$ is an $R_0\text{-ring}_{\K}$   with trivial grading, the zeroth component of the coproduct $R\ast_{R_0} S$ contains $R_0\ast_{R_0} S\cong S$. And the zeroth component of $R\ast_{R_0} S$   can contain elements such as $rsr'$, where $r,r'\in R$ are homogeneous and the sum of the degrees are zero.

\begin{prop}
\label{propcoproduct}
    Let $\K$ be a commutative $\Gamma$-graded ring concentrated in degree zero, and $R$ a strongly $\G$-graded $\K$-algebra.  Suppose that $M$ is a $\Gamma$-graded $R$-module and $P$ a graded finitely generated projective $R$-module. Then we have the following isomorphism of  $\Gamma$-graded $\K$-algebras 
    $$R\ast_{R_0} R_0\langle f_0: \overline{M_0}\longrightarrow \overline{P_0}\rangle \cong R\langle f:\overline{M}\longrightarrow \overline{P} \rangle.$$ 
    
\end{prop}

\begin{proof}
    Choose a projective resolution  $$\bigoplus_{i\in I}R_0\stackrel{u_0}{\longrightarrow}\bigoplus_{j\in J} R_0
\stackrel{\varepsilon_0}{\longrightarrow} M_0 \longrightarrow 0$$ for $M_0$ as $R_0$-module. Since $P$ is a graded finitely generated projective $R$-module, and $R$ is strongly graded, its zeroth component $P_0$ is a finitely generated projective $R_0$-module. We write $P_0$ as the image of the idempotent morphism $e_0: \bigoplus_{m\in K} R_0\xra \bigoplus_{m\in K} R_0$ of $R_0$-modules.

For any $R_0\text{-ring}_{\K}$ $T_0$ and any module homomorphism $g_0:T_0\otimes_{R_0} M_0 \xrightarrow{} T_0\otimes_{R_0} P_0$, we have the following diagram with the first two rows exact
\begin{equation*}
\xymatrix{T_0\otimes_{R_0}  (\bigoplus_{i\in I}R_0)\ar[r]^{\overline{u_0}}& T_0\otimes_{R_0} (\bigoplus_{j\in J} R_0)\ar[r]^>>>>>{\overline{\varepsilon_0}} \ar[d]^{h_0}&T_0\otimes_{R_0} M_0\ar[r]\ar[d]^{g_0} &0\\ & T_0\otimes_{R_0}  (\bigoplus_{m\in K}R_0) \ar[r]^>>>>>{\overline{e_0}} \ar[d]^{\overline{1-e_0}}& T_0\otimes_{R_0}  P_0\ar[r]&0\\ &T_0\otimes_{R_0}  (\bigoplus_{m\in K}R_0)&& },
\end{equation*} where by Lemma \ref{lemma-resolution}, $g_0$ is uniquely determined by a module homomorphism $h_0$ with $h_0\circ\overline{u_0}=0=\overline{1-e_0}\circ h_0$. Then we have $$R_0\langle f_0: \overline{M_0}\longrightarrow\overline{P_0}\rangle\cong R_0\langle h_0 \rangle/ \langle u_0h_0=0=h_0(1-e_0)\rangle.$$ Here $ u_0h_0=0$ and $h_0(1-e_0)=0$ follows from the matrix relations \eqref{matrixrelation-one} and \eqref{matrixrelation-two}; compare Remark \ref{rmkthm} (2).

Since the graded ring $R$ is strongly graded, $M\cong R\otimes_{R_0} M_0$ and we have a projective resolution of $M$ as follows: $$\bigoplus_{i\in I} R\stackrel{1\otimes_{R_0} u_0}{\longrightarrow}\bigoplus_{j\in J} R
\stackrel{1\otimes_{R_0}\varepsilon_0}{\longrightarrow} R\otimes_{R_0} M_0 \longrightarrow 0.$$

For any $\G$-graded $R\text{-ring}_{\K}$ $T$ and any graded module homomorphism $g:T\otimes_{R} M \xrightarrow{} T\otimes_{R} P$, we have the following diagram with the first two rows exact
\begin{equation*}
\xymatrix{T \otimes_R R \otimes_{R_0}  (\bigoplus_{i\in I}R_0)\ar[r]^{\overline{u_0}}& T \otimes_R R \otimes_{R_0} (\bigoplus_{j\in J} R_0)\ar[r]^>>>>>{\overline{\varepsilon_0}} \ar[d]^{h}&T \otimes_R R \otimes_{R_0} M_0\ar[r]\ar[d]^{g} &0\\ & T \otimes_R R \otimes_{R_0}  (\bigoplus_{m\in K}R_0) \ar[r]^>>>>>{\overline{e_0}} \ar[d]^{\overline{1-e_0}}& T \otimes_R R \otimes_{R_0}  P_0\ar[r]&0\\ &T \otimes_R R \otimes_{R_0}  (\bigoplus_{m\in K}R_0)&& },
\end{equation*} where by Lemma \ref{lemma-resolution}, $g$ is uniquely determined by a graded module homomorphism $h$ with $h\circ\overline{u_0}=0=\overline{1-e_0}\circ h$. Then we have \begin{equation}
\label{isosg}
\begin{split}
R\langle f: \overline{M}\longrightarrow\overline{P}\rangle
&\cong R\langle h\rangle/ \langle u_0h=0=h(1-e_0)\rangle 
\end{split}
\end{equation} By Theorem \ref{thm-1}, $R\langle f: \overline{M}\xra \overline{P}\rangle$  is  a $\Gamma$-graded $\K$-algebra. One observes that by the proof of Theorem \ref{thm-1} the elements of the matrix $h$ as additional generators of $R\langle f: \overline{M}\xra \overline{P}\rangle$ are of zero degree. By Lemma~\ref{lemma-resolution}, it follows that the universal homomorphism of $R_0\langle f_0: \overline{M_0}\xra \overline{P_0} \rangle $ is the restriction of the universal homomorphism of $R\langle f:  \overline{M}\xra \overline{P}\rangle$ to the zeroth component. 

 Now we prove that 
    $R\ast_{R_0} R_0\langle f_0: \overline{M_0}\xra \overline{P_0}\rangle \cong R\langle f:\overline{M}\xra \overline{P} \rangle$ as $\K$-algebras. One notes that we have the natural inclusion $\iota:R_0\xra R$, and the natural homomorphisms $\pi:R_0\xra R_0\langle f_0: \overline{M_0}\xra \overline{P_0} \rangle$ sending $r_0$ in $R_0$ to $\overline{r_0}$, as well as  $\pi':R\xra R\langle f: \overline{M}\xra \overline{P} \rangle$ sending $r$ in $R$ to $\overline{r}$.  The homomorphism $\varphi: R_0\langle f_0: \overline{M_0}\xra \overline{P_0} \rangle \xra R\langle f: \overline{M}\xra \overline{P} \rangle$ is induced by the homomorphism $R_0\langle h_0\rangle\xra R\langle h\rangle$ since the matrix $h$ is induced by $h_0$, equivalently, we have $\varphi(r_0)=r_0$ for $r_0\in R_0$ and $\varphi(h_{ij})=h_{ij}$ for $h_0=(h_{ij})$. It follows directly that in the diagram below,  $\varphi\circ \pi=\pi'\circ \iota$. In order to prove that $R\langle f: \overline{M}\xra \overline{P}\rangle$ is the coproduct $R*_{R_0} R_0\langle f_0: \overline{M_0}\xra \overline{P_0} \rangle$, one only needs to check that $R\langle f: \overline{M}\xra \overline{P}\rangle$ has the universal property. That is, taking any $R_0\text{-ring}_{\K}$ $S$ and any homomorphisms $\alpha: R_0\langle f_0: \overline{M_0}\xra \overline{P_0} \rangle\xra S$ and $\beta: R\xra S$ with $\alpha\circ \pi=\beta\circ \iota$, we need to show that there exists a unique homomorphism $\theta: R\langle f: \overline{M}\xra \overline{P} \rangle \xra S$ in the following diagram 
\begin{equation*}
        \xymatrix{
        R_0 \ar[r]^{\iota} \ar[d]_{\pi} & R \ar[d]_{\pi'} \ar@/^/[rdd]^{\beta}&\\
        R_0\langle f_0: \overline{M_0}\xra \overline{P_0} \rangle \ar@/_/[rrd]_{\alpha} \ar[r]^>>>>>>{\varphi} &  R\langle f: \overline{M}\xra \overline{P} \rangle \ar@.[rd]^{\theta}&\\
        &&S
        }
    \end{equation*} such that $\alpha=\theta\circ \varphi$ and $\beta=\theta\circ \pi'$. First we define $\theta: R\langle f: \overline{M}\xra \overline{P}\rangle \xra S$, where $\theta(r)=\beta(r)$, for $r\in R$ and $\theta(h_{ij})=\alpha(h_{ij})$ for $h=(h_{ij})$. We observe that $\theta$ is well defined as $\alpha$ sends the generators of the two-sided ideal $\langle u_0h_0=0=h_0(1-e_0)\rangle$ to zero. And we can check directly that $\alpha=\theta\circ \varphi$ and $\beta=\theta\circ \pi'$. The uniqueness of $\theta$  follows directly from the commutativity of the two triangles in the above diagram. One observes that $\theta$ preserves the grading in the graded algebras setting. Therefore, we obtain that $R\ast_{R_0} R_0\langle f_0: \overline{M_0}\xra \overline{P_0}\rangle \cong R\langle f:\overline{M}\xra \overline{P} \rangle$ as $\K$-algebras.
    \end{proof}

\begin{prop}
\label{propcoproducttwo}
    Let $R$ be a strongly $\G$-graded $\K$-algebra. Suppose that $M$ is a $\Gamma$-graded $R$-module, $P$ a graded finitely generated projective $R$-module and $f:M\xrightarrow{} P$ is a graded module homomorphism. Then we have the following isomorphism of  $\Gamma$-graded $\K$-algebras 
    $$R\ast_{R_0} R_0\langle \overline{f_0}=0\rangle \cong R\langle \overline{f}=0 \rangle.$$ Here, we have $f_0=(-)_0(f)$. Furthermore, the zeroth component of the $\G$-graded $\K$-algebra $R\langle \overline{f}=0 \rangle$ is isomorphic to $R_0\langle \overline{f_0}=0\rangle$.
\end{prop}

\begin{proof}
As $R$ is a strongly $\G$-graded $\K$-algebra and $P$ is a graded finitely generated projective $R$-module, $P_0$ is a finitely generated projective $R_0$-module. We write $P_0$ as the summand of the free $R_0$-module $\bigoplus_{c\in C} R_0$. The fact that it is a direct summand insures that for any
$R_0\text{-}{\rm ring}_{\K_0}$  $T$, the map $\overline{i}:T\otimes_{R_0} P_0 \xra T\otimes_{R_0}(\bigoplus_{c\in C} R_0)$ induced by the inclusion $i: P_0\xra \bigoplus_{c\in C} R_0$ is again an inclusion. Hence if we write for each $m_0\in M_0$ $$f_0(m_0)=\sum_{c\in C} f_{m_0, c}$$ with $f_{m_0,c}$ in the component $R_0$ corresponding to $c$. We can see that for any $R_0\text{-}{\rm ring}_{\K}$  $T$, $T\otimes_{R_0} f_0$ will be zero if and only
if all the $f_{m_0,c}$ go to zero in $T$. Hence $$R_0\langle \overline{f_0}=0\rangle\cong R_0/\langle f_{m_0,c}\rangle_{m_0\in M_0, c\in C}. $$ Here $\langle f_{m_0,c}\rangle_{m_0\in M_0, c\in C}$ is the two-sided ideal of $R_0$ generated by the elements $f_{m_0,c}$ for $m_0\in M_0, c\in C$.

One observes that $R\otimes_{R_0} P_0$ is a direct summand of the free $R$-module $R\otimes_{R_0}(\bigoplus_{c\in C} R_0)\cong \bigoplus_{c\in C} R$. Actually we have the maps
$$R\otimes_{R_0} M_0\longrightarrow{} R\otimes_{R_0} P_0\longrightarrow{} R\otimes_{R_0}(\bigoplus_{c\in C} R_0)\cong \bigoplus_{c\in C}R.$$
We write for each $m_0\in M_0$ $$R\otimes_{R_0} f_0(r\otimes m_0)=\sum_{c\in C} rf_{m_0, c}$$ with $r\in R$ a homogeneous element, and $f_{m_0,c}$ in the component $R_0$ corresponding to $c$. For any $\G$-graded $R\text{-}{\rm ring}_{\K}$  $T'$, $T'\otimes_{R} f$ will be zero if and only
if $T'\otimes_{R} (R\otimes_{R_0} f_0)$ will be zero, if and only if all the $rf_{m_0,c}$ go to zero in $T'$ for any homogeneous elements $r\in R$, $c\in C$ and $m_0\in M_0$. Hence $R\langle \overline{f}=0\rangle\cong R/\langle rf_{m_0,c}\rangle_{r\in R_h, m_0\in M_0, c\in C}$. Here $R_h$ is the set consisting of all homogeneous elements of $R$ and $\langle rf_{m_0,c}\rangle_{m_0\in M_0, c\in C}$ is the two-sided ideal of $R$ generated by the homogeneous elements $rf_{m_0,c}$ for $m_0\in M_0, c\in C$. Hence we have $$R\langle \overline{f}=0\rangle\cong R/\langle rf_{m_0,c}\rangle_{r\in R_h, m_0\in M_0, c\in C}.$$  
One observes that we obtain the following commutative diagram with $\varphi_1(r)=r$, for $r\in R$ and $\psi_1(\overline{r_0})=\overline{r_0}$, for $r_0\in R_0$. 
\begin{equation*}
        \xymatrix{
        R_0 \ar[r]^>>>>>{\pi} \ar[d]_{\iota} &R_0\langle \overline{f_0}=0\rangle \ar[d]_{\psi_1}&\\
        R \ar[r]^>>>>>>{\varphi_1} & R\langle \overline{f}=0\rangle & &&
        }
    \end{equation*} Here $\iota:R_0\xra R$ is the inclusion and $\pi: R_0\xra R_0\langle \overline{f_0}=0\rangle$ is the natural projection. One can check directly that $R\langle \overline{f}=0\rangle$ satisfies the universal property of $R\ast_{R_0} R_0\langle \overline{f_0}=0\rangle$ as $\Gamma$-graded $\K$-algebras. The proof is completed.
\end{proof}

Let $R$ be a strongly $\G$-graded $\K$-algebra. Suppose that $M$ and $P$ be $\Gamma$-graded finitely generated projective $R$-modules. Then we have the following two universal algebras: 
  $$R_0\langle i_0, i^{-1}_0: \overline M_0 \cong  \overline P_0\rangle \text{ and }  R\langle i, i^{-1}: \overline M \cong \overline P\rangle.$$ 
    
It would be interesting to establish a relation between these two algebras and $ R\langle i, i^{-1}: \overline M \cong \overline P\rangle_0$ similar to Proposition~\ref{propcoproduct}.

\subsection{Matrix forms for three types of Bergman algebras}\label{ghostagain}
Let $\Gamma$ be an abelian group, $\K$ a $\Gamma$-graded commutative ring and  $R$ a $\Gamma$-graded $\K$-algebra. 

Let $P$ and $Q$ be non-zero graded finitely generated projective $R$-modules. Applying Theorem~\ref{thm-1} we obtain a $\Gamma$-graded ring  $R\langle i:\overline{P}\xrightarrow{} \overline{Q}\rangle$. Recall from \S\ref{bgdhyu} that we can also construct the graded rings $R\langle i, i^{-1} :\overline{P}\cong_{\gr} \overline{Q}\rangle$ and $R\big\langle i:\overline{P}\to\overline{P};~i^2=i\big\rangle$. 
Further if $g:P\rightarrow Q$ is a graded $R$-module homomorhism, then $R\langle \overline{g}^{-1} \rangle$ is the universal ring which $g$ becomes invertible.  In this section, by representing the graded finitely generated projective modules by idempotent matrices, 
we describe these $\Gamma$-graded rings via generators and relations, 
in  Lemmas \ref{lempresS}, \ref{lempresS*} and \ref{lemma-univ}.  Once the grade group $\Gamma$ is set to be trivial, we obtain the Bergman algebras of Definition~\ref{berglagebra}.

 Since $P$ is graded finitely generated $R$-module, using~(\ref{projidem1}), we can write $P$ as the image of a  graded idempotent endomorphism $e$ of a  graded free $R$-module of finite rank $\bigoplus_{j=1}^n R(\beta_j)$ with $\beta_j\in \Gamma$. Similarly, we write  $Q$ as the image of a  graded idempotent endomorphism $f$ of a  graded free $R$-module of finite rank $\bigoplus_{k=1}^m R(\gamma_k)$ with $\gamma_k\in\Gamma$.

By Lemma \ref{lemma-resolution} there is a unique graded homomorphism $h_{g}: \bigoplus_{j=1}^n R(\b_j)\xra \bigoplus_{k=1}^m R(\g_k) $ corresponding to the given graded module homomorphism $g: P\xra Q$ such that we have the following diagram 
\begin{equation}\label{howinfhf}
	\xymatrix{\bigoplus_{j=1}^n R(\b_j)\ar[r]^{1-e}& \bigoplus_{j=1}^n R(\b_j)\ar[r]^>>>>>{e} \ar[d]^{h_{g}}&P\ar[r]\ar[d]^{g} &0\\ & \bigoplus_{k=1}^m R(\g_m) \ar[r]^>>>>>{f} \ar[d]^{1-f}&  Q\ar[r]&0\\ & \bigoplus_{k=1}^mR(\g_m)&& },
\end{equation} with the first two rows exact, the right square commutative and $h_{g}\circ (1-e)=0$ and $(1-f)\circ h_{g}=0$. One can represent $h_g$ as a matrix  $h_g \in \mathbb{M}_{m\times n}(R)[\overline{\beta}][\overline{\gamma}]$, where $\overline \beta=(\beta_1, \dots, \beta_n)$ and  $\overline \gamma=(\gamma_1, \dots, \gamma_m)$ (see(~\ref{aidansan})).
So  Lemma \ref{lemma-resolution}, one observes that there is a one-to-one correspondence between the set of graded $R$-module homomorphisms $g: P\xra Q$ and the set of  $R$-module homomorphisms  $h_{g}: \bigoplus_{j=1}^n R(\b_j)\xrightarrow{} \bigoplus_{k=1}^m R(\g_m)$ satisfying $eh_{g}f=h_{g}$.

\begin{lemma}
\label{lemma-matrix} 
Let $R$ be a $\Gamma$-graded $\K$-algebra and let $P$ and $Q$ be non-zero graded finitely generated  projective $R$-modules. 
\begin{enumerate}[\upshape(1)]
   \item We have that $R\langle i: \overline{P}\xrightarrow{} \overline{Q}\rangle\cong R\langle h\rangle/{\langle ehf=h\rangle} $ as $\Gamma$-graded algebras for some $n\times m$-matrix $h=(h_{jk})$ with $h_{jk}$ symbols of degree $\gamma_k-\beta_j\in \Gamma$, for $j=1,\cdots, n$ and $k=1, \cdots, m$ and homogeneous idempotents $e\in \M_n(R)(\overline \beta)_0$ and $f\in \M_m(R)(\overline \gamma)_0$, where $\overline \beta=(\beta_1, \dots, \beta_n)$ and $\overline \gamma=(\gamma_1, \dots, \gamma_m)$. Moreover, there is a graded $\K$-algebra homomorphism $R\xra R\langle \overline{P}\xrightarrow{} \overline{Q}\rangle$ sending $r$ to $\overline{r}$ for any $r\in R$.

   \item Let $g: P\xrightarrow{} Q$ be a graded $R$-module homomorphism. Then 
   $R\langle \overline{g}=0\rangle \cong R/\langle h_g\rangle $, where $h_g$ is the corresponding matrix from $\bigoplus_{j=1}^n R(\b_j)$ to $\bigoplus_{k=1}^m R(\g_k)$ of {\upshape (\ref{howinfhf})} whose $jk$-th entry is of degree $\gamma_k-\beta_j\in \Gamma$ and $\langle h_g\rangle$ is the two-sided ideal generated by the entries of $h_g$. 
   \end{enumerate}
\end{lemma}

\begin{proof}
Writing, $P\bigoplus P'\cong_{\gr} \bigoplus_{i=1}^n R(\beta_i)$, we choose $$\bigoplus_{j=1}^n R(\b_j)\xrightarrow{1-e}\bigoplus_{j=1}^n R(\b_j)
\xrightarrow{e}P \xrightarrow{} 0$$ as the projective resolution of $P$ as graded $R$-module. Writing $Q\bigoplus Q'\cong_{\gr} \bigoplus_{i=1}^m R(\gamma_i)$,
by the proof of Theorem \ref{thm-1}, for any $\G$-graded $R\text{-ring}_{\K}$ $T$ and any graded module homomorphsm $g':T\otimes_R P \xrightarrow{} T\otimes_R Q$, we have the following diagram with the first two rows exact 
\begin{equation}\label{ggffddsss}
	\xymatrix{T\otimes_R (\bigoplus_{j=1}^n R(\b_j))\ar[r]^{\overline{1-e}}& T\otimes_R (\bigoplus_{j=1}^n R(\b_j))\ar[r]^>>>>>{\overline{e}} \ar[d]^{h}&T\otimes_R P\ar[r]\ar[d]^{g'} &0\\ & T\otimes_R (\bigoplus_{k=1}^m R(\g_m)) \ar[r]^>>>>>{\overline{f}} \ar[d]^{\overline{1-f}}& T\otimes_R Q\ar[r]&0\\ & T\otimes_R (\bigoplus_{k=1}^mR(\g_m))&& },
\end{equation}

For (1), by matrix relations \eqref{matrixrelation-one} and \eqref{matrixrelation-two}, we have the multiplication of matrices  
$(1-e)h=0$ and $h(1-f)=0$, which is equivalent to that $ehf=h$. Thus $R\langle i: \overline{P}\xrightarrow{} \overline{Q}\rangle$ is isomorphic to $R\langle h\rangle/{\langle ehf=h\rangle}$. 

For (2), by Lemma \ref{lemma-resolution}, for $g:P\rightarrow Q$,  there is a unique $h_g$  from $\bigoplus_{j=1}^n R(\b_j)$ to $\bigoplus_{k=1}^m R(\g_m)$, and we have $\overline{g}=0$ if and only if $h_g=0$. By the proof of Theorem \ref{thm-2} the desired universal ring $R\langle \overline{g}=0\rangle $ is  the quotient of $R$ by the two-sided ideal generated by entries of the matrix $h_g$.
\end{proof}

We have the following two consequences which will be used in the following subsections.

\begin{lemma}\label{lempresS} Let $R$ be a $\Gamma$-graded $\K$-algebra and let $P$ and $Q$ be non-zero graded finitely generated projective $R$-modules. 
 We have the $\Gamma$-graded algebra isomorphism $$R\big\langle i,i^{-1}:\overline{P}\cong \overline{Q}\big\rangle \cong R\big\langle h, h^*\big\rangle /\big\langle ehf=h, fh^*e=h^*, hh^*=e, h^*h=f\big\rangle,$$ where $h=(h_{jk})$ is a $n\times m$-matrix with $h_{jk}$ symbols of degree $\gamma_k-\beta_j$ and $h^*=(h^*_{kj})$ is a $m\times n$-matrix with $h^*_{kj}$ symbols of degree $\beta_j-\gamma_k$, for $j=1,\cdots, n$ and $k=1,\cdots, m$, and homogeneous idempotents $e\in \M_n(R)(\overline \beta)_0$ and $f\in \M_m(R)(\overline \gamma)_0$, where $\overline \beta=(\beta_1, \dots, \beta_n)$ and $\overline \gamma=(\gamma_1, \dots, \gamma_m)$.
\end{lemma}

\begin{proof}
By Lemma~\ref{lemma-matrix}(1) we have that $R\langle i_0:\overline{P}\xra \overline{Q}, i_0': \overline{Q}\xra \overline{P}\rangle \cong R\langle h, h^*\rangle /\langle ehf=h, fh^*e=h^*\rangle$. Now for the $\grring$ $S=R\langle i_0:\overline{P}\xra \overline{Q}, i_0':\overline{Q}\xra \overline{P}\rangle $, we have the following two diagrams (see~(\ref{ggffddsss}))
\begin{equation*}
	\xymatrix{S\otimes_R (\bigoplus_{j=1}^n R(\b_j))\ar[r]^{\overline{1-e}}& S\otimes_R (\bigoplus_{j=1}^n R(\b_j))\ar[r]^>>>>>{\overline{e}} \ar[d]^{h}&S\otimes_R P\ar[r]\ar[d]^{i_0} &0\\
 & S\otimes_R (\bigoplus_{k=1}^m R(\g_k)) \ar[r]^>>>>>{\overline{f}} \ar[d]^{\overline{1-f}}& S\otimes_R Q\ar[r]&0\\ 
 & S\otimes_R (\bigoplus_{k=1}^m R(\g_k)&& }
\end{equation*} with the first two rows exact, the right square commutative and $(1-e)h=0$ and $h(1-f)=0$, and 
\begin{equation*}
	\xymatrix{S\otimes_R (\bigoplus_{k=1}^m R(\g_k))\ar[r]^{\overline{1-f}}& S\otimes_R (\bigoplus_{k=1}^m R(\g_k))\ar[r]^>>>>>{\overline{f}} \ar[d]^{h^*}&S\otimes_R Q\ar[r]\ar[d]^{i_0'} &0\\
 & S\otimes_R (\bigoplus_{j=1}^n R(\b_j)) \ar[r]^>>>>>{\overline{e}} \ar[d]^{\overline{1-e}}& S\otimes_R P\ar[r]&0\\ 
 & S\otimes_R (\bigoplus_{j=1}^n R(\b_j)&& }
\end{equation*} with the first two rows exact, the right square commutative and $(1-f)h^*=0$ and $h^*(1-e)=0$. Combining the above two diagrams, we have the following diagram 
\begin{equation*}
	\xymatrix{S\otimes_R (\bigoplus_{j=1}^n R(\b_j))\ar[r]^{\overline{1-e}}& S\otimes_R (\bigoplus_{j=1}^n R(\b_j))\ar[r]^>>>>>{\overline{e}} \ar[d]^{e-h^*\circ h}&S\otimes_R P\ar[r]\ar[d]^{\id_{S\otimes P}-i_0'\circ i_0} &0\\ & S\otimes_R (\bigoplus_{j=1}^n R(\b_j)) \ar[r]^>>>>>{\overline{e}} \ar[d]^{\overline{1-e}}& S\otimes_R P\ar[r]&0\\ & S\otimes_R (\bigoplus_{j=1}^n R(\b_j))&& }
\end{equation*} with the first two rows exact, the right square commutative and the matrix form $(1-e)(e-hh^*)=0$ and $(e-hh^*)(1-e)=0$. By Lemma \ref{lemma-resolution} $e-h^* \circ h$ is the corresponding map for $\id_{S\otimes P}-i_0'\circ i_0$. The corresponding matrix representation of $e-h^* \circ h$ is $e-hh^*$ (see~\ref{jan16dis}). Similarly one can show that $f-h^*h$ is the corresponding matrix of $\id_{S\otimes Q}-i_0\circ i_0'$. By Lemma \ref{lemma-matrix}(2), we then have $$R\big\langle i,i^{-1}:\overline{P}\cong \overline{Q}\big\rangle \cong R\big\langle h, h^*\big\rangle /\big\langle ehf=h, fh^*e=h^*, hh^*=e, h^*h=f\big\rangle.$$
This completes the proof. 
\end{proof}

\begin{lemma}\label{lempresS*}
 Let $R$ be a $\Gamma$-graded $\K$-algebra and $P$ a non-zero graded finitely generated projective $R$-module. We have the $\Gamma$-graded algebra isomorphism $$R\big\langle i:\overline{P}\to\overline{P};~i^2=i\big\rangle \cong R\big\langle h\big\rangle /\big\langle ehe=h, h^2=h\big\rangle,$$ where $h=(h_{kl})$ is a $n\times n$-matrix with $h_{kl}$ symbols of degree $\beta_l-\beta_k$, for $k,l=1,\cdots, n$ and $e\in \M_n(R)(\overline \beta)_0$, where $\overline \beta= (\beta_1, \dots, \beta_n)$  is a homogeneoues idempotent. 
\end{lemma}

\begin{proof}
    By Lemma \ref{lemma-matrix}(1) we have that $R\langle i': \overline{P}\xra \overline{P}\rangle \cong R\langle h\rangle /\langle ehe=h\rangle$. Now for the graded $R\text{-ring}_{\K}$ $S'=R\langle i':\overline{P}\xra \overline{P}\rangle $, we have the following diagram
\begin{equation*}
	\xymatrix{S'\otimes_R (\bigoplus_{j=1}^n R(\b_j))\ar[r]^{\overline{1-e}}& S'\otimes_R (\bigoplus_{j=1}^n R(\b_j))\ar[r]^>>>>>{\overline{e}} \ar[d]^{h}&S'\otimes_R P\ar[r]\ar[d]^{i'} &0\\
 & S'\otimes_R (\bigoplus_{j=1}^n R(\b_j)) \ar[r]^>>>>>{\overline{e}} \ar[d]^{\overline{1-e}}& S'\otimes_R P\ar[r]&0\\ 
 & S'\otimes_R (\bigoplus_{j=1}^n R(\b_j)&& }
\end{equation*} with the first two rows exact, the right square commutative and $(1-e)h=0$ and $h(1-e)=0$. Observe that $h-h^2$ corresponds to the matrix of $i'-i'\circ i'$ (Lemma \ref{lemma-resolution}). Now by Lemma \ref{lemma-matrix}(2), the proof is completed.
\end{proof}

Given a graded homomorphism $g:P\xrightarrow{} Q$ between graded finitely generated projective $R$-modules $P$ and $Q$, we may construct a $\Gamma$-graded $R\text{-ring}_{\K}$ $R\langle \overline{g}^{-1} \rangle$ in which $\overline{g}$ become an isomorphism, by adjoining a map $g: \overline{Q}\xrightarrow{} \overline{P}$, then setting $1_{\overline{P}}-g\circ \overline{g}$ and $1_{\overline{Q}}-\overline{g}\circ g$ equal to zero.

Recall that by Lemma \ref{lemma-resolution} there is a unique graded homomorphism $h_{g}: \bigoplus_{j=1}^n R(\b_j)\xra \bigoplus_{k=1}^m R(\g_k) $ corresponding to $\g: P\xra Q$ satisfying $eh_{g}f=h_{g}$.
We can describe the universal localisation algebra of the $R$-module homomorphism $g:P \rightarrow Q$ as follows (see~\cite[page 293]{bergman78}).

\begin{lemma} \label{lemma-univ} 
Let $R$ be a $\Gamma$-graded $\K$-algebra and $g: P\xrightarrow{} Q$ a graded $R$-module homomorphism. 
We have the $\Gamma$-graded algebra isomorphism $$R\big\langle \overline{g}^{-1} \big\rangle \cong R\big\langle h^*\big\rangle /\big\langle fh^*e=h, h_{g}h^*=e,h^*h_{g}=f\big\rangle,$$ where $h^*=(h^*_{kj})$ is a $m\times n$-matrix with $h^*_{kj}$ symbols of degree $\beta_j-\g_k$, for $j=1,\cdots, n$ and $k=1,\cdots, m$.
\end{lemma}

\begin{proof}
By  Lemma \ref{lemma-matrix}(1) we have $S_1 = R\langle i: \overline{Q}\xra \overline{P}\rangle \cong R\langle h^*\rangle/ fh^*e=h^*\rangle$, for some $m\times n$-matrix $h^*$. Similarly as in  Lemma \ref{lempresS}, one has that $h^*\circ h_{g}-e$ is the unique graded homomorphism corresponding to $i\circ (S_1\otimes_R g) -1_{S_1\otimes_R P}$ and that $h_{g}\circ h^*-f$ is the unique graded homomorphism corresponding to $(S_1\otimes_R g)\circ i-1_{S_1\otimes_R Q}$. We omit the details here.
\end{proof}

\subsection{Universal morphisms and representable functors}

We recall some facts on representable functors and universal morphisms.

Let $\mathcal{C}$ be a small category and $F:\mathcal{C}\xrightarrow{} \Sets$ a functor, where $\Sets$ is the category of sets with set functions as morphisms.  We say that $F$ is \emph{representable} if there exists a natural isomorphism $\Hom_{\mathcal{C}}(A, -)\xrightarrow{} F$ for some object $A$ in $\mathcal{C}$. For an element $A$ in $\mathcal{C}$, we say that an element $u\in F(A)$ is \emph{universal} if for any object $B$ in $\mathcal{C}$ and $v\in F(B)$, there is a unique morphism $f: A\xrightarrow{} B$ with $v=F(f)(u)$. In this case, this property is called the universal property of $(A, u)$. 

\begin{lemma}
\label{lemma-bijection}
Let $A$ be an object in $\mathcal{C}$. The representations $\alpha: \Hom_{\mathcal{C}}(A, -)\xrightarrow{} F$ are in bijection to universal elements of the form $(A, u)$ via $\a\xrightarrow{} \a_A(1_A)$.
\end{lemma}

\begin{proof}
	By Yoneda Lemma, the natural transformations $\a: \Hom_{\mathcal{C}}(A, -)\xrightarrow{} F$ correspond to elements in $F(A)$ via $\a\mapsto \a_A(1_A)$. It remains to check that $\a$ is an isomorphism if and only if $\a_A(1_A)$ in $F(A)$ is an universal element. Actually, for any object $B$ in $\mathcal{C}$, we have $\a_B: \Hom_{\mathcal{C}}(A,B)\xrightarrow{} F(B)$. By the naturality  of $\a$, for any morphism $f: A\xrightarrow{} B$ 
in $\mathcal{C}$, we have $\a_B(f)=\a_B(f\circ 1_A)=F(f)(\a_A(1_A))$. Thus $\a$ is a natural isomorphism if and only if for any object $B$ in $\mathcal{C}$, the map $\Hom_{\mathcal{C}}(A, B)\xrightarrow{} F(B), f\mapsto F(f)(\a_A(1_A))$ is bijection, which is a restatement of saying that $\a_A(1_A)$ in $F(A)$ is a universal element.
\end{proof}

In what follows, we continue with the assumption that $\Gamma$ is an abelian group, $\K$ a $\Gamma$-graded commutative ring and $R$ a $\Gamma$-graded $\K$-algebra. Suppose that $M$ is a $\Gamma$-graded $R$-module and $P$ a graded finitely generated projective $R$-module.

Theorem \ref{thm-1} now says that the functor from $\G$-graded $R\text{-}\text{rings}_{\K}$ to $\Sets$, associating to $T$ the set $\Hom_{T\text{-}\GR}(T\otimes_R M, T\otimes_R P)$, is representable. More precisely, we have the functor 
\begin{equation*}
\begin{split}
\mathcal F^{\gr}_{M,P}: \grring & \longrightarrow{} \Sets\\
T&\longmapsto \Hom_{T\text{-}\GR}(T\otimes_R M, T\otimes_R P)
\end{split}
\end{equation*}
 such that $F^{\gr}_{M,P}(f)=T'\otimes_T -$ for any graded $\K$-algebra homomorphism $f:T\xrightarrow{} T'$. Then,   there is a natural isomorphism $\a: \Hom_{\grring}(S, -)\xrightarrow{} F^{\gr}_{M,P}$ with $S$ the graded algebra constructed in Theorem \ref{thm-1}.

We recall from \cite[Chapter~4]{Schofield} the notion of universal localization of algebras. Let $R$ be
an algebra and $\sigma$ be a family of morphisms between finitely generated projective $R$-modules. A homomorphism $\theta:R\xrightarrow[]{} S$ of algebras is called \emph{$\sigma$-inverting} if for each morphism $\xi\in\sigma$, the morphism $S\otimes_R \xi$ in $S$-${\mathbf{Mod}}$ is invertible; $\theta$ is called a \emph{universal localization 
with respect to $\sigma$} if in addition each $\sigma$-inverting homomorphism $\theta':R\xrightarrow[]{} S'$ factors
uniquely through $\theta$.

We adopt this to the setting of graded algebras.  Let $\Gamma$ be an abelian group, $R,$
$S$ two $\G$-graded $\K$-algebras and $\Sigma$ a family of morphisms between $\Gamma$-graded finitely generated  projective $R$-modules. A homomorphism $\theta:R\xrightarrow[]{} S$ of $\Gamma$-graded algebras is \emph{graded $\Sigma$-inverting}, if for all $\xi\in \Sigma$, the morphism $S\otimes_R \xi$ in $S$-$\GR$ is invertible; $\theta$ is called a \emph{graded universal localization with repect to $\Sigma$} if in addition every $\Sigma$-inverting homomorphism $\theta':R\xrightarrow[]{} S'$ of $\Gamma$-graded algebras factors uniquely through $\theta$. In \cite[Proposition~3.1]{cy}, it was shown that, the $\mathbb Z$-graded universal localisation exists. Using our graded constructions, here we show that, for any arbitrary abelian group $\Gamma$, the $\Gamma$-graded universal localisation exists.

Recall from \S\ref{bgdhyu} that given a graded homomorphism $f:P \rightarrow Q$  between graded finitely generated projective $R$-modules $P$ and $Q$, we can construct a graded ring $R\langle {\overline f}^{-1}\rangle$ where the homomorphism $\overline f$ becomes invertible.

\begin{lemma} Let $R$ and $S$ be $\Gamma$-graded $\K$-algebras and $\Sigma$ a family of morphisms between $\Gamma$-graded finitely generated projective  $R$-modules.   If $\theta:R\xra{}S$ is a $\Gamma$-graded universal localization with respect to $\Sigma$, then $S\cong R\langle \overline{\xi}^{-1}\rangle_{\xi\in \Sigma}$.
\end{lemma}

\begin{proof}  For each $\xi:P_{\xi}\xra{}Q_{\xi}$ in $\Sigma$, collect the pair of graded finitely generated projective $R$-modules $(Q_{\xi}, P_{\xi})$.  
By a repeated application of Theorem~\ref{thm-1} for each pair $(Q_{\xi}, P_{\xi})$, the $\Gamma$-graded ring 
$$S_{\Sigma} = R \big \langle \Delta_{\xi}: \overline Q_{\xi} \rightarrow \overline P_\xi \mid \xi \in \Sigma \big \rangle, $$
is the graded $R\text{-}\text{ring}_{\K}$ with a universal family of graded module homomorphisms  
$$\big \{\Delta_{\xi}:S_{\Sigma}\otimes_R Q_{\xi}\xra{}S_{\Sigma}\otimes_R P_{\xi}\, |\, (Q_{\xi}, P_{\xi}), \xi \in  \Sigma\big\}.$$  As $S_{\Sigma}$ is a graded $R\text{-}\text{ring}_{\K}$, we have the graded algebra homomorphism $\theta_1: R\xra S_{\Sigma}$.  Let $R_{\Sigma}$  be the graded $S_{\Sigma}\text{-}\text{ring}_{\K}$  such that 
\begin{equation}
\label{stwo}
\begin{cases}
	R_{\Sigma}\otimes_{\theta_2}(1-\Delta_{\xi}\circ S_{\Sigma}\otimes_{\theta_1}\xi)=0;\\\
	R_{\Sigma}\otimes_{\theta_2}(1-S_{\Sigma}\otimes_{\theta_1}\xi\circ \Delta_{\xi})=0.
	\end{cases}
\end{equation}  Here $\theta_2: S_{\Sigma}\xra{}R_{\Sigma}$ is the graded algebra homomorphism such that $R_{\Sigma}$ is a graded $S_{\Sigma}\text{-}\text{ring}_{\K}$.  From the construction, we have  $R_{\Sigma}= R\langle \overline{\xi}^{-1}\rangle_{\xi\in \Sigma} $.

Since $S$ is a graded localisation ring, for each $\xi: P_{\xi}\xra{}Q_{\xi}\in \Sigma$, we have the isomorphism $S\otimes_R \xi:S\otimes_R P_{\xi}\xra{}S\otimes_R Q_{\xi}$. Now by the universal property of $\Delta_{\xi}$, there exists a unique graded algebra homomorphism $\lambda: S_{\Sigma}\xra{}S$ such that $(S\otimes_R \xi)^{-1}=S\otimes_{\lambda}\Delta_{\xi}$. Hence $S$ satisfies \eqref{stwo}, that is, \begin{equation*}
\begin{cases}
	1-(S\otimes_{\lambda}\Delta_{\xi})\circ (S\otimes_R \xi)
 =S\otimes_{\lambda}(1-\Delta_{\xi}\circ (S_{\Sigma}\otimes_{\theta_1}\xi))=0;\\
	1-(S\otimes_R \xi)\circ(S\otimes_{\lambda}\Delta_{\xi})=S\otimes_{\lambda}(1-(S_{\Sigma}\otimes_{\theta_1}\xi)\circ \Delta_{\xi})=0.
	\end{cases}
\end{equation*}

Next we prove that $S$ has the universal property of $R_{\Sigma}$. For any graded $S_{\Sigma}\text{-}\text{ring}_{\K}$ $T$ with $	T\otimes (1-\Delta_{\xi}\circ (S_{\Sigma}\otimes_{\theta_1}\xi))=0 $ and $	T\otimes (1-(S_{\Sigma}\otimes_{\theta_1} \xi)\circ \Delta_{\xi})=0, $ there exists a unique graded algebra homomorphism $\eta:S\xra{}T$ in the following diagram 
\[\xymatrix{R\ar[r]^{\theta_1}\ar[d]_{\theta}& S_{\Sigma} \ar[ld]_{\lambda}\ar[r]^{\theta_2^T} &T
	 \\
  S\ar[rru]_{\eta} & & }\] such that $\theta_2^T\circ \theta_1=\eta\circ \theta$. The proof is completed. 
\end{proof}

We are in a position to extend  \cite[Proposition 3.1 (1)]{cy}, which only worked for $\mathbb Z$-graded algebras. The existence of a $\mathbb Z$-graded universal localization follows from the argument in \cite[Theorem 4.1]{Schofield} adapted for categories with automorphisms. Here we prove that the existence of any $\Gamma$-graded universal localization follows from graded Bergman's constructions where $\Gamma$ is any abelian group.

Recall that a $\Gamma$-graded ring $R$ (with unit) is called left graded hereditary if any graded submodule of a graded projective left $R$-module is (graded) projective, which is equivalent to that the graded global dimension for $R$ being less than or equal to $1$ (see~\S\ref{gradedhersec}).

\begin{prop}
\label{propuniversalloc}
  Let $\Gamma$ be an abelian group, $R$ a $\G$-graded $\K$-algebra and $\Sigma$ a family of morphisms between $\Gamma$-graded finitely generated  projective $R$-modules. We have the following statements.
  \begin{enumerate}[\upshape(1)]
  \item The induced homomorphism $R\xrightarrow{} R\langle {\overline{\xi}}^{-1} \rangle_{\xi\in \Sigma}$, $r\mapsto \overline{r}$, is a   $\Gamma$-graded universal localization of $R$ with respect to $\Sigma$, denoted by $\theta_{\Sigma}$, and is unique up to isomorphism. 
  
  \item Moreover, the underlying homomorphism $\theta_{\Sigma}: R\xrightarrow{} R\langle {\overline{\xi}}^{-1} \rangle_{\xi\in \Sigma}$ of ungraded algebras is a universal localization of $R$ with respect to $\sigma=U(\Sigma)$ which is the image of $\Sigma$ under the forgetful functor $U$. 
  
  \item If $R$ is graded hereditary, so is $R\langle {\overline{\xi}}^{-1} \rangle_{\xi\in \Sigma}$.
  \end{enumerate}
\end{prop}
\begin{proof}
(1) Recall from \eqref{stwo} that $S_{\Sigma}=R\langle \Delta_{\xi}: \overline{Q_{\xi}}\xra \overline{P_{\xi}} | \xi:P_{\xi}\xra Q_{\xi}\in \Sigma\rangle$ and that $R_{\Sigma}$ satisfies that  $$R_{\Sigma}\otimes_R \xi:R_{\Sigma}\otimes_R P_{\xi}\xrightarrow[]{\sim} R_{\Sigma}\otimes_R Q_{\xi},$$ is an isomorphism for each $\xi: P_{\xi}\xra Q_{\xi} \in \Sigma$.

For any $\Sigma$-inverting graded homomorphism $\theta': R\xra S'$, there exists a unique graded $R$-algebra homomorphism $\lambda: S_{\Sigma}\xra{}S'$ such that $(S'\otimes_R \xi)^{-1}=S'\otimes_{\lambda}\Delta_{\xi}$. Hence we have
\begin{equation}
\label{firstcom}
    \theta'=\lambda\circ \theta_1.
\end{equation}
For each $\xi$ in $\Sigma$, we have the isomorphism $S'\otimes_R \xi:S'\otimes_R P_{\xi}\xra{}S'\otimes_R Q_{\xi}$, and  
\begin{equation*}
\begin{cases}
	1-(S'\otimes_{\lambda}\Delta_{\xi})\circ (S'\otimes_R \xi)
 =S'\otimes_{\lambda}(1-\Delta_{\xi}\circ (S_{\Sigma}\otimes_{\theta_1}\xi))=0;\\
	1-(S'\otimes_R \xi)\circ(S'\otimes_{\lambda}\Delta_{\xi})=S'\otimes_{\lambda}(1-(S_{\Sigma}\otimes_{\theta_1}\xi)\circ \Delta_{\xi})=0.
	\end{cases}
\end{equation*} There exists a unique graded $S_{\Sigma}$-algebra homomorphism $\eta: R_{\Sigma}\xra S'$ in the following diagram.
\[\xymatrix{R\ar[r]^{\theta_1}\ar[d]_{\theta'}& S_{\Sigma} \ar[ld]_{\lambda}\ar[r]^{\theta_2} &R_{\Sigma}\ar[lld]^{\eta}
	 \\
  S' & & }\] Thus we have 
\begin{equation}
\label{secondcom}
    \eta\circ\theta_2=\lambda.
\end{equation} Therefore we have $\theta_2\circ \theta_1\circ\eta=\theta'$. Suppose that there exists $\eta': R_{\Sigma}\xra S'$ such that $\theta_2\circ \theta_1\circ \eta'=\theta'$. Combining with \eqref{firstcom} we have that $\theta_2\circ \theta_1\circ \eta'=\theta'=\lambda\circ \theta_1$. By the uniqueness of $\lambda$ we have $\lambda=\eta'\circ \theta_2$. By comparing with \eqref{secondcom} we have $\eta=\eta'$. Then we prove that $\theta_2\circ\theta_1: R\xra R_{\Sigma}$ is a graded universal localization of $R$ with respect to $\Sigma$, which will be denoted by  $\theta_{\Sigma}$. The uniqueness of $\theta_{\Sigma}$ up to isomorphism follows from the universal property of Bergman algebras.

(2) Take $\Gamma$ to be the trivial group. Then we have the second statement.

(3) By Bergman's construction $R\langle {\overline{\xi}}^{-1} \rangle_{\xi\in \Sigma}$ is hereditary. As it is a graded ring, it is graded hereditary. 
    \end{proof}

\subsection{Path algebras as universal algebras}
     
Bergman's proof of Theorem~5.3 in \cite{bergman74} indicates that any lower triangle matrix ring can be realised as the universal algebra of the form $R \langle  f_i:\overline P_i \rightarrow \overline Q_i \rangle$, where $R$ can be chosen as an appropriate
 commutative $\K$-semisimple ring. Furthermore, inverting these maps, i.e.,  $R \langle  f_i, f_i^{-1}:\overline P_i  \cong \overline Q_i \rangle$ would then produce the whole matrix ring. As the lower triangle matrix rings are path algebras of certain kind of line graphs, it is natural to ask whether all path algebras of graphs can be obtained as universal algebras. Here we show this is in fact the case, namely for any finite graph $E$, the path algebra $P_\K(E) \cong  R \langle  f_i:\overline P_i  \rightarrow \overline Q_i \rangle$, where $R$ is an commutative $\K$-semisimple ring and $(P_i,Q_i)$ are finitely generated projective $\K$-modules determined by the shape of the graph.  Although it is usually quite difficult to describe the universal ring obtained by inverting morphisms (see for example \cite{sheiham}), here we show that the Leavitt path algebra $L_\K(E)$ is the universal algebra, which inverts the maps $f_i: P_i  \rightarrow Q_i $. 

Although the results of this section might be known (or predictable) to the experts, we provide the details to demonstrate what it means to build a graded ring from the inset and not to assign the  grading from the outset, as it has been done so far in the theory of combinatorial algebras. Namely, previously we would construct (Leavitt) path algebras and then assign a grading to them. By using our graded Bergman machinery, the algebras obtained are already equipped with a grading.

Let $E$ be a finite graph. We use $E^{\reg}$ to denote the set of vertices in $E$ at which there are at least one edge starting from it.
Recall that the \emph{path algebra} of the graph $E$ is defined as 
\begin{equation}\label{gfhfyhr}
P_\K(E) = \K \big \langle E^0, E^1 \big \rangle \big / \big \langle u v = \delta_{u,v}, s(e)e=er(e)=e, u,v \in E^0, e\in E^1 \big \rangle, 
\end{equation}
where $\K \big \langle E^0, E^1 \big \rangle$ is the free (non-unital) algebra with vertices and edges of the graph as the generators. 
Assigning zero to vertices and $1$ to edges, since the defining relations in (\ref{gfhfyhr}) are homogeneous, $P_\K(E)$ becomes a $\mathbb Z$-graded ring (with positive support). 
In Lemma \ref{lempathalg} we will realise $P_\K(E)$ as a $\mathbb Z$-graded universal algebra.

Consider the semisimple $\K$-algebra $R=\prod_{E^0}\K$, the product of $|E^0|$-copies  of the field $\K$ as a $\mathbb Z$-graded ring which is concentrated in the zero component.  
Denote $p_v(i)$, where $v\in E^0$ and $i\in \mathbb Z$, the graded finitely generated projective  $R$-module with $\K(i)$ appears in $v$-th component of $R$ and zero elsewhere. Throughout, we write $p_v$ for $p_v(0)$. Recall that $p_v(i), i \in \mathbb Z$, is the shift of the projective $R$-module $p_v$ by $i$.  Note that $\bigoplus_{v\in E^0}p_v \cong R$ as graded $R$-modules.

 \begin{lemma}
 \label{lempathalg}
 	Let $E$ be a finite graph and $\K$ a field. Let $R=\prod_{E^0}\K$ be the $\mathbb{Z}$-graded algebra concentrated in degree zero, with graded finitely generated projective modules $p_v(i)$ as above. Then we have a $\mathbb Z$-graded $\K$-algebra isomorphism $$P_\K(E)\cong R\,\big \langle g^v:\overline{p_v}\longrightarrow{} \overline{\bigoplus_{e\in s^{-1}(v)}{p_{r(e)}(1)}}, \, {v\in E^{\reg}} \big \rangle.$$   
 \end{lemma}

\begin{proof}
We observe that for any regular vertex $v\in E^0$, we have the following diagram with exact rows of $R$-modules
	 \[\xymatrix{\bigoplus_{w\neq v}p_w\ar[r]^{u}& R \ar[r]^{\varepsilon} \ar[d]^{h^v}&p_v\ar[r] &0
	 \\ Q(1)\ar[r]& \bigoplus\limits_{e\in s^{-1}(v)}R(1)\ar[r]^{1-f^v} \ar[d]^{f^v}& \bigoplus\limits_{e\in s^{-1}(v)}p_{r(e)}(1) \ar[r]&0\\ 
	 & \bigoplus\limits_{e\in s^{-1}(v)}R(1)&& }\] where $Q=\bigoplus_{e\in s^{-1}(v)}(R/p_{r(e)})$ with $$\bigoplus_{e\in s^{-1}(v)}p_{r(e)}\bigoplus Q=\bigoplus_{e\in s^{-1}(v)}R, $$ that is, the second row of the above diagram splits. And $f^v$ is an idempotent endomorphism of $\bigoplus_{e\in s^{-1}(v)}R(1)$ such that $f^v|_{Q(1)}=\id_{Q(1)}$ and $\Ker f^v=\bigoplus_{e\in s^{-1}(v)}p_{r(e)}(1)$. Note that the graded module homomorphism $h^v$ can be written as the matrix 
	 $$\begin{pmatrix}{e_1^v}\\ e_2^v\\
	 \vdots\\
	 e_{n_v}^v
	 \end{pmatrix}$$ such that $e_1^v, e_2^v, \cdots, e_{n_v}^v$ are elements in $R$ corresponding to the index set of all the edges starting from $v$ in $E$. By Theorem \ref{thm-1}, the universal algebra $S$ with universal graded module homomorphisms $$g^v: S \otimes_R p_v\longrightarrow S \otimes_R \bigoplus_{e\in s^{-1}(v)}p_{r(e)}(1)$$ for regular vertices $v\in E^0$ is obtained via adjoining to $R$ generators $e_i^v$ for $v\in E^{\reg}$ and $1\leq i\leq n_v$ subject to the relations $h^v \circ u=0=f^v\circ h^v$. More precisely, we have $h^v\circ u=0$ implying 
	 \begin{equation}
	 \label{pathalgebrarelation1}
{\mathbf e}_{w} e^v_i=0
\end{equation}
for any $w\neq v$ and all $1\leq i\leq n_v$. Here, $$\mathbf{e}_w=\begin{pmatrix}
	 	0,\cdots, 0, 1,0\cdots, 0
	 \end{pmatrix}\in p_w.$$ Hence by the universal property of the cokernel $\varepsilon$, there exists a unique graded homomorphism $\tau: p_v\xra \bigoplus_{e\in s^{-1}(v)}R(1)$ such that $h^v=\tau \circ \varepsilon$. This deduce that \begin{equation}
	 \label{pathalgebrarelation2}
	 	\mathbf{e}_v e_i^v=e^v_i
	 \end{equation} for all $1\leq i\leq n_v$. Also we have $f^v\circ h^v=0$ implying 
	 \begin{equation}
	 	\label{pathalgebrarelation3}
e^v_i \mathbf{e}_{r(e^v_i)}=e^v_i,\qquad e^v_i {\mathbf e}_w=0 \text{~for~} w\neq r(e^v_i),
\end{equation}
 for each $1\leq i\leq n_v$. One can see that $S$ is isomorphic to the path algebra $P_\K(E)$ when one observes that $\mathbf e_v$ corresponds to the vertex $v\in E^0$, and for each $v\in E^{\reg}$,  $e^v_i$ corresponds to the edge starting from $v$ for $1\leq i\leq n_v$. 
\end{proof}

One observes that the universal graded module homomorphism $$g^v: P_\K( E)\otimes_R p_v\longrightarrow P_\K(E)\otimes_R \big (\bigoplus_{e\in s^{-1}(v)}p_{r(e)}(1)\big)$$ in the following diagram corresponding to $P_\K( E)\otimes_R h^v$  is given by $g^v(xv)=\sum_{e\in s^{-1}(v)}xe$ for any $x\in P_\K( E)$. 
 \[\xymatrix{
 P_\K( E) \otimes_R p_v\ar[d]^{\cong}\ar[r]^>>>>{g^v}& P_\K( E)\otimes_R (\bigoplus_{e\in s^{-1}(v)}p_{r(e)}(1))\ar[d]^{\cong}\\
 P_\K( E) \, v \ar[r]^>>>>>>>>>{g^v} & \bigoplus_{e\in s^{-1}(v)}P_\K( E) \,r(e)}(1)\]

We have the following consequence (compare with \cite[Proposition 4.3(1)]{cy}). We mention that in \cite{cy} the authors constructed the maps $g^v$ which are exactly the graded universal homomorphisms of Bergman algebras. 

\begin{thm} \label{levbalcon}
The $\mathbb Z$-graded algebra homomorphism $\iota: P_\K(E)\xrightarrow{} L(E), \, p\mapsto p$, for $p$ a path in $E$,  is a graded universal localization with respect to $$\Sigma=\Big \{g^v: P_\K(E) v\longrightarrow \bigoplus_{e\in s^{-1}(v)} P_\K(E) r(e)(1) \, |\, v\in E^{\reg}\Big \},$$ where $g^v(x)=\sum_{e\in s^{-1}(v)}xe$ for $x\in P_\K(E) v$.
\end{thm}

\begin{proof}
    By Proposition \ref{propuniversalloc}(1) we have that $P_{\K}(E)\xrightarrow[]{}P_{\K}(E)\langle \xi^{-1}\rangle_{\xi\in \Sigma}$,  $p\mapsto p$, is a $\mathbb Z$-graded universal localization of $P_{\K}(E)$ with respect to $\Sigma$. One observes that \begin{equation}
    \label{com}
        \begin{split}
            P_{\K}(E)\langle \xi^{-1}\rangle_{\xi\in \Sigma}
            &\cong R\big \langle \overline{p_v}
            \cong \overline{\oplus_{e\in s^{-1}(v)}p_{r(e)}(1)}, v\in E^{\rm reg}\big\rangle\\
            &\cong B_R(e, f)\\
            &\cong L(H)\\
            &\cong L(E)
        \end{split}
    \end{equation} where $R=\prod_{E^0}\K$ and $(e, f)=\{(e_{h^v}, f_{h^v})\, |\, v\in E^{\rm reg}\}$ with $e_{h^v}=(\epsilon_{v})\in R_0$ and $f_{h^v}=\in \Mat_{|J_v|\times |J_v|}(R)(1,1,\cdots, 1)_0$ the idempotent matrix whose entry at position $(j,j')$ is $\delta_{jj'}\epsilon_{r(e^{v,j})}$. Here in the equation \eqref{com} the first isomorphism follows by Lemma \ref{lempathalg}, the second isomorphism follows by Lemma \ref{lempresS}, the third isomorphism is given in Theorem \ref{lemhyperberg} and the last isomorphism is given in Example \ref{exmlpa}. Moreover all the four isomorphisms preserve elements. Hence the composition $P_{\K}(E)\xrightarrow[]{}P_{\K}(E)\langle \xi^{-1}\rangle_{\xi\in \Sigma}\xra L(E)$, $p\mapsto p$ is a $\mathbb Z$-graded universal localization of $P_{\K}(E)$ with respect to $\Sigma$.
\end{proof}



 \section{Smash products of graded Bergman algebras}\label{smashbergloc}

 Throughout this section, $\Gamma$ is an abelian group, $\K$ a commutative $\Gamma$-graded ring concentrated in degree $0$, and $R$ is a $\Gamma$-graded $k$-algebra, i.e., $R$ is equipped with a graded homomorphism $\K\rightarrow Z(R)$.  Note that, by our assumptions $\K \subseteq R_0$.

\subsection{Smash products of graded Bergman algebras I}\label{typeabergsmash}

Let $P$ and $Q$ be non-zero graded  finitely generated projective $R$-modules and set 
\[S:=R\big\langle h,h^{-1}:\overline{P}\cong \overline{Q}\big\rangle.\]
We write $P$ as the image of an idempotent graded endomorphism $d$ of a graded free $R$-module of finite rank $\bigoplus_{j\in J}R(\beta_j)$, and similarly $Q$ as the image of an idempotent graded endomorphism $e$ of a graded free $R$-module of finite rank $\bigoplus_{m\in M}R(\gamma_m)$. Clearly $\deg(d_{ij})=\beta_j-\beta_i$, for any $i,j\in J$, and $\deg(e_{km})=\gamma_m-\gamma_k$, for any $k,m\in M$. Recall from Lemma \ref{lempresS} that $S$ can be obtained from $R$ by adjoining homogeneous generators $h_{jm},h'_{mj}~(j\in J,m\in M)$ of degree $\deg(h_{jm})=\gamma_m-\beta_j$, respectively $\deg(h'_{mj})=\beta_j-\gamma_m$, and relations $dh=h=he,~eh'=h'=h'd,~hh'=d,~h'h=e$. Hence $S\cong R\langle X\rangle/I$, where $X=\{h_{jm},h'_{mj}\mid j\in J,m\in M\}$ and $I$ is the ideal of $R\langle X\rangle$ generated by the relations 
\begin{align*}
\sum_{i\in J}d_{ji}h_{im}&=h_{jm}=\sum_{k\in M}h_{jk}e_{km}&(j\in J,m\in M),\\
\sum_{k\in M}e_{mk}h'_{kj}&=h'_{mj}=\sum_{i\in J}h'_{mi}d_{ij}&(j\in J,m\in M),\\
\sum_{m\in M}h_{im}h'_{mj}&=d_{ij}&(i,j\in J),\\
\sum_{j\in J}h'_{kj}h_{jm}&=e_{km}&(k,m\in M).
\end{align*}
It follows from Lemma \ref{lempressmash} and Corollary \ref{corsmashkey} that $S\#\Gamma$ has the presentation 
 \begin{equation}\label{eqn6.1}
 \begin{split}
     S\#\Gamma=\Big\langle& rp_{\gamma},h_{jm}p_\gamma,h'_{mj}p_\gamma\quad(r\in R^{h}, \gamma\in \Gamma,j\in J,m\in M)\mid{}\\
 &rp_\gamma+sp_\gamma=(r+s)p_\gamma\quad(r,s\in R_{\delta};~\gamma,\delta\in \Gamma),\\
 &rp_{\gamma}sp_{\delta}=rs_{\gamma-\delta} p_{\delta}\quad(r,s\in R^{h};~\gamma,\delta\in\Gamma ),\\
 &(h_{jm}p_\gamma)(1p_\gamma)=h_{jm}p_\gamma=(1p_{\gamma+\gamma_m-\beta_j})(h_{jm}p_\gamma)\quad(j\in J,m\in M, \gamma\in \Gamma),\\
  &(h'_{mj}p_\gamma)(1p_\gamma)=h'_{mj}p_\gamma=(1p_{\gamma+\beta_j-\gamma_m})(h'_{mj}p_\gamma)\quad(j\in J,m\in M, \gamma\in \Gamma)\\
&\sum_{i\in J}d_{ji}p_{\gamma+\gamma_m-\beta_i}h_{im}p_\gamma=h_{jm}p_\gamma=\sum_{k\in M}h_{jk}p_{\gamma+\gamma_m-\gamma_k}e_{km}p_\gamma\quad(j\in J,m\in M,\gamma\in\Gamma)\\
&\sum_{k\in M}e_{mk}p_{\gamma+\beta_j-\gamma_k}h'_{kj}p_\gamma=h'_{mj}p_\gamma=\sum_{i\in J}h'_{mi}p_{\gamma+\beta_j-\beta_i}d_{ij}p_\gamma\quad(j\in J,m\in M,\gamma\in\Gamma)\\
&\sum_{m\in M}h_{im}p_{\gamma+\beta_j-\gamma_m}h'_{mj}p_\gamma=d_{ij}p_\gamma\quad(i,j\in J;~\gamma\in\Gamma)\\
&\sum_{j\in J}h'_{kj}p_{\gamma+\gamma_m-\beta_j}h_{jm}p_\gamma=e_{km}p_\gamma\quad(k,m\in M;~\gamma\in\Gamma)\Big\rangle.
\end{split}
 \end{equation}
For any $\gamma\in \Gamma$, $i,j\in J$ and $k,m\in M$, define
\begin{align*}
d_{ij}^{(\gamma)}&:=d_{ij}p_{\gamma-\beta_j},\\
e^{(\gamma)}_{km}&:=e_{km}p_{\gamma-\gamma_m},\\
h^{(\gamma)}_{jm}&:=h_{jm}p_{\gamma-\gamma_m},\\
(h')^{(\gamma)}_{mj}&:=h'_{mj}p_{\gamma-\beta_j}.
\end{align*}
Using these conventions it follows from (\ref{eqn6.1}) that 

\begin{equation}\label{eqn6.2}
\begin{split}
 S\#\Gamma=\Big\langle& rp_{\gamma},h^{(\gamma)}_{jm},(h')^{(\gamma)}_{mj}\quad(r\in R^{h}, \gamma\in \Gamma,j\in J,m\in M)\mid{}\\
 &rp_\gamma+sp_\gamma=(r+s)p_\gamma\quad(r,s\in R_{\delta};~\gamma,\delta\in \Gamma),\\
 &rp_{\gamma}sp_{\delta}=rs_{\gamma-\delta} p_{\delta}\quad(r,s\in R^{h};~\gamma,\delta\in\Gamma ),\\
 &h_{jm}^{(\gamma)}(1p_{\gamma-\gamma_m})=h_{jm}^{(\gamma)}=(1p_{\gamma-\beta_j})h_{jm}^{(\gamma)}\quad(j\in J,m\in M, \gamma\in \Gamma),\\
  &(h')_{mj}^{(\gamma)}(1p_{\gamma-\beta_j})=(h')_{mj}^{(\gamma)}=(1p_{\gamma-\gamma_m})(h')_{mj}^{(\gamma)}\quad(j\in J,m\in M, \gamma\in \Gamma),\\
&\sum_{i\in J}d^{(\gamma)}_{ji}h_{im}^{(\gamma)}=h_{jm}^{(\gamma)}=\sum_{k\in M}h_{jk}^{(\gamma)}e_{km}^{(\gamma)}\quad(j\in J,m\in M,\gamma\in\Gamma),\\
&\sum_{k\in M}e_{mk}^{(\gamma)}(h'_{kj})^{(\gamma)}=(h'_{mj})^{(\gamma)}=\sum_{i\in J}(h')_{mi}^{(\gamma)}d_{ij}^{(\gamma)}\quad(j\in J,m\in M,\gamma\in\Gamma),\\
&\sum_{m\in M}h_{im}^{(\gamma)}(h')^{(\gamma)}_{mj}=d_{ij}^{(\gamma)}\quad(i,j\in J;~\gamma\in\Gamma),\\
&\sum_{j\in J}(h')_{kj}^{(\gamma)}h_{jm}^{(\gamma)}=e_{km}^{(\gamma)}\quad(k,m\in M;~\gamma\in\Gamma)\Big\rangle.
\end{split}
 \end{equation}

Next, for any finite subset $A\subseteq \Gamma$ we define the $k$-algebra $T_A$ by

\begin{equation}\label{eqn6.3}
\begin{split}
T_A=\Big\langle& rp_{\beta}~~(r\in R_\alpha;~\alpha+\beta,\beta\in\Gamma_A);\quad h^{(\gamma)}_{jm},(h')^{(\gamma)}_{mj}~~(\gamma\in A,j\in J,m\in M)\mid{}\\
 &rp_\beta+sp_\beta=(r+s)p_\beta\quad(r,s\in R_{\alpha};~\alpha+\beta,\beta\in\Gamma_A),\\
 &rp_{\beta}sp_{\beta'}=rs_{\beta-\beta'} p_{\beta'}\quad(r\in R_{\alpha};~s\in R_{\alpha'};~\alpha+\beta,\beta, \alpha'+\beta',\beta'\in\Gamma_A),\\
 &h_{jm}^{(\gamma)}(1p_{\gamma-\gamma_m})=h_{jm}^{(\gamma)}=(1p_{\gamma-\beta_j})h_{jm}^{(\gamma)}\quad(j\in J,m\in M, \gamma\in A),\\
  &(h')_{mj}^{(\gamma)}(1p_{\gamma-\beta_j})=(h')_{mj}^{(\gamma)}=(1p_{\gamma-\gamma_m})(h')_{mj}^{(\gamma)}\quad(j\in J,m\in M, \gamma\in A),\\
&\sum_{i\in J}d^{(\gamma)}_{ji}h_{im}^{(\gamma)}=h_{jm}^{(\gamma)}=\sum_{k\in M}h_{jk}^{(\gamma)}e_{km}^{(\gamma)}\quad(j\in J,m\in M,\gamma\in A),\\
&\sum_{k\in M}e_{mk}^{(\gamma)}(h'_{kj})^{(\gamma)}=(h'_{mj})^{(\gamma)}=\sum_{i\in J}(h')_{mi}^{(\gamma)}d_{ij}^{(\gamma)}\quad(j\in J,m\in M,\gamma\in A),\\
&\sum_{m\in M}h_{im}^{(\gamma)}(h')^{(\gamma)}_{mj}=d_{ij}^{(\gamma)}\quad(i,j\in J;~\gamma\in A),\\
&\sum_{j\in J}(h')_{kj}^{(\gamma)}h_{jm}^{(\gamma)}=e_{km}^{(\gamma)}\quad(k,m\in M;~\gamma\in A)\Big\rangle,
\end{split}
 \end{equation}
where $\Gamma_A=\{\gamma-\gamma_m,\gamma-\beta_j\mid \gamma\in A,j\in J,m\in M\}$. Clearly $T_A$ is a unital algebra with identity $\sum_{\gamma\in\Gamma_A}p_\gamma$. Moreover, \[S\#\Gamma=\varinjlim_A T_A,\] where $A$ runs through all finite subsets of $\Gamma$. We will show that the $T_A$'s are isomorphic to Bergman algebras.

Let $A\subseteq \Gamma$ a finite subset. We denote by $R\# \Gamma_A$ the subalgebra of $R\# \Gamma$ consisting of all sums of elements $rp_\beta$, where $r\in R_\alpha$ and $\alpha+\beta,\beta\in \Gamma_A$. Note that $R\# \Gamma_A$ is is a unital algebra with identity $\sum_{\gamma\in\Gamma_A}p_\gamma$.

\begin{lemma}\label{lemrelpress}
The $\K$-algebra $R\#\Gamma_A$ has the presentation 
 \begin{equation}\label{eqn6.4}
 \begin{split}
 R\#\Gamma_A=\Big\langle& rp_{\beta}\quad(r\in R_\alpha;~\alpha+\beta,\beta\in\Gamma_A)\mid{}\\
 &rp_\beta+sp_\beta=(r+s)p_\beta\quad(r,s\in R_{\alpha};~\alpha+\beta,\beta\in\Gamma_A),\\
 &rp_{\beta}sp_{\beta'}=rs_{\beta-\beta'} p_{\beta'}\quad(r\in R_{\alpha};~s\in R_{\alpha'};~\alpha+\beta,\beta, \alpha'+\beta',\beta'\in\Gamma_A)\Big\rangle.
 \end{split}
 \end{equation}   
\end{lemma}
\begin{proof}
Denote by $Z$ the $\K$-algebra presented by (\ref{eqn6.4}). It follows from Lemma \ref{lempressmash} that there is a $\K$-algebra homomorphism $\phi:Z\to R\#\Gamma$ such that $\phi(rp_\gamma)=rp_\gamma$, where $r\in R_\alpha$ and $\alpha+\beta,\beta\in\Gamma_A$. Clearly one can write any element $z\in Z$ in the form $z=\sum_{\beta\in\Gamma_A}\sum_{\alpha+\beta\in\Gamma_A}r_\beta^{(\alpha)}p_\beta$, where each $r^{(\alpha)}_\beta\in R_\alpha$. It follows that $\phi(Z)=R\# \Gamma_A$. It remains to show that $\phi$ is injective. Suppose that $\phi(z)=0$, for some $z\in Z$. Write $z=\sum_{\beta\in\Gamma_A}\sum_{\alpha+\beta\in\Gamma_A}r_\beta^{(\alpha)}p_\beta$, where each $r^{(\alpha)}_\beta\in R_\alpha$. Since $\phi(z)=0$, it follows that each $r^{(\alpha)}_\beta=0$. Thus $\phi$ is injective and therefore $Z\cong \phi(Z)=R\#\Gamma_A$.
\end{proof}

For any $\gamma\in A$ define the matrices 
\begin{align*}
d^{(\gamma)}&:=(d_{ij}^{(\gamma)})_{i,j\in J}=(d_{ij}p_{\gamma-\beta_j})_{i,j\in J}\in \Mat_{J\times J}(R\#\Gamma_A),\\
e^{(\gamma)}&:=(e_{km}^{(\gamma)})_{k,m\in K}=(e_{km}p_{\gamma-\gamma_m})_{k,m\in K}\in \Mat_{M\times M}(R\#\Gamma_A).
\end{align*}
A straightforward computation shows that $d^{(\gamma)}$ and $e^{(\gamma)}$ are idempotent matrices. Hence they define idempotent endomorphisms of $\bigoplus_{j\in J}R\# \Gamma_A$ and $\bigoplus_{m\in M}R\# \Gamma_A$, respectively. It follows that $P_A^{(\gamma)}:=(\bigoplus_{j\in J}R\# \Gamma_A)d^{(\gamma)}$ and $Q_A^{(\gamma)}:=(\bigoplus_{m\in M}R\# \Gamma_A)e^{(\gamma)}$ are non-zero finitely generated projective $R\#\Gamma_A$-modules. Define the $\K$-algebra
\begin{equation*}
B_A:=R\#\Gamma_A\Big\langle g^{(\gamma)}, (g^{(\gamma)})^{-1}:\overline{P_A^{(\gamma)}}\cong\overline{Q_A^{(\gamma)}}~(\gamma\in A)\Big\rangle.
\end{equation*}
We will show that $T_A\cong B_A$. By Lemma \ref{lempresS} and (\ref{eqn6.4}) we have
\begin{equation}\label{eqn6.5}
\begin{split}
B_A=\Big\langle& rp_{\beta}~(r\in R_\alpha;~\alpha+\beta,\beta\in\Gamma_A);~~~~g^{(\gamma)}_{jm},(g')^{(\gamma)}_{mj}~~(\gamma\in A,j\in J,m\in M)\mid{}\\
 &rp_\beta+sp_\beta=(r+s)p_\beta\quad(r,s\in R_{\alpha};~\alpha+\beta,\beta\in\Gamma_A),\\
 &rp_{\beta}sp_{\beta'}=rs_{\beta-\beta'} p_{\beta'}\quad(r\in R_{\alpha};~s\in R_{\alpha'};~\alpha+\beta,\beta, \alpha'+\beta',\beta'\in\Gamma_A),\\
&\sum_{i\in J}d^{(\gamma)}_{ji}g_{im}^{(\gamma)}=g_{jm}^{(\gamma)}=\sum_{k\in M}g_{jk}^{(\gamma)}e_{km}^{(\gamma)}\quad(j\in J,m\in M,\gamma\in A),\\
&\sum_{k\in M}e_{mk}^{(\gamma)}(g'_{kj})^{(\gamma)}=(g'_{mj})^{(\gamma)}=\sum_{i\in J}(g')_{mi}^{(\gamma)}d_{ij}^{(\gamma)}\quad(j\in J,m\in M,\gamma\in A),\\
&\sum_{m\in M}g_{im}^{(\gamma)}(g')^{(\gamma)}_{mj}=d_{ij}^{(\gamma)}\quad(i,j\in J;~\gamma\in A),\\
&\sum_{j\in J}(g')_{kj}^{(\gamma)}g_{jm}^{(\gamma)}=e_{km}^{(\gamma)}\quad(k,m\in M;~\gamma\in A)\Big\rangle,
\end{split}
 \end{equation}

\begin{lemma}\label{lemmonomial}
In $B_A$ the relations 
\begin{align}
 &g_{jm}^{(\gamma)}(1p_{\gamma-\gamma_m})=g_{jm}^{(\gamma)}=(1p_{\gamma-\beta_j})g_{jm}^{(\gamma)}\quad(j\in J,m\in M, \gamma\in A),\label{eqn6.6}\\
  &(g')_{mj}^{(\gamma)}(1p_{\gamma-\beta_j})=(g')_{mj}^{(\gamma)}=(1p_{\gamma-\gamma_m})(g')_{mj}^{(\gamma)}\quad(j\in J,m\in M, \gamma\in A),\label{eqn6.7}
\end{align}
hold.
\end{lemma}
\begin{proof}
We only prove relation (\ref{eqn6.6}) and leave relation (\ref{eqn6.7}) to the reader. 
Clearly \begin{align*}
    g_{jm}^{(\gamma)}(1p_{\gamma-\gamma_m})
        =&\Big (\sum_{k\in M}g_{jk}^{(\gamma)}e_{km}^{(\gamma)}\Big)(1p_{\gamma-\gamma_m})\\
    =&\sum_{k\in M}g_{jk}^{(\gamma)}(e_{km}p_{\gamma-\gamma_m})(1p_{\gamma-\gamma_m})\\
=&\sum_{k\in M}g_{jk}^{(\gamma)}(e_{km}p_{\gamma-\gamma_m})\\
=&\sum_{k\in M}g_{jk}^{(\gamma)}e_{km}^{(\gamma)}\\
    =&g_{jm}^{(\gamma)}.
\end{align*}
Similarly
\begin{align*}
    (1p_{\gamma-\beta_j})g_{jm}^{(\gamma)}
=& (1p_{\gamma-\beta_j})\sum_{i\in J}d_{ji}^{(\gamma)}g^{(\gamma)}_{im}\\
=& (1p_{\gamma-\beta_j})\sum_{i\in J}(d_{ji}p_{\gamma-\beta_i})g^{(\gamma)}_{im}\\
=&\sum_{i\in J}(1p_{\gamma-\beta_j})(d_{ji}p_{\gamma-\beta_i})g^{(\gamma)}_{im}\\
=&\sum_{i\in J}((d_{ji})_{\gamma-\beta_j-(\gamma-\beta_i)}p_{\gamma-\beta_i})g^{(\gamma)}_{im}\\
=&\sum_{i\in J}(d_{ji}p_{\gamma-\beta_i})g^{(\gamma)}_{im}\\
=&\sum_{i\in J}d_{ji}^{(\gamma)}g^{(\gamma)}_{im}\\
=&g_{jm}^{(\gamma)}
\end{align*}
since $\deg(d_{ji})=\beta_i-\beta_j$.
\end{proof}

\begin{prop}\label{propbergsmash}
Let $A\subseteq \Gamma$ be a finite subset. Then $T_A\cong B_A$.
\end{prop}
\begin{proof}
The proposition follows from (\ref{eqn6.3}), (\ref{eqn6.5}), (\ref{eqn6.6}) and (\ref{eqn6.7}).
\end{proof}

The theorem below is the main result of this subsection. It follows directly from Proposition \ref{propbergsmash} and the fact that $S\#\Gamma=\varinjlim_A T_A$, where $A$ runs through all finite subsets of $\Gamma$.
\begin{thm}\label{thmbergsmash}
The $\K$-algebra $R\big\langle h,h^{-1}:\overline{P}\cong \overline{Q}\big\rangle\#\Gamma$ is the direct limit of the Bergman algebras
\[R\#\Gamma_A\Big\langle g^{(\gamma)}, (g^{(\gamma)})^{-1}:\overline{P_A^{(\gamma)}}\cong\overline{Q_A^{(\gamma)}}~(\gamma\in A)\Big\rangle,\]
where $A$ runs through all finite subsets of $\Gamma$.
\end{thm}

We finish this subsection by proving a lemma which will be used in \S 7. For $\gamma\in \Gamma$ we define the $R\#\Gamma$-modules $P^{(\gamma)}:=(\bigoplus_{j\in J}R\# \Gamma)d^{(\gamma)}$ and $Q^{(\gamma)}:=(\bigoplus_{m\in M}R\# \Gamma)e^{(\gamma)}$.

\begin{lemma}\label{lemtech_1}
Let $\gamma\in \Gamma$. Then $P^{(\gamma)}\cong P(-\gamma)\#\Gamma$ and $Q^{(\gamma)}\cong Q(-\gamma)\#\Gamma$ as $R\#\Gamma$-modules.
\end{lemma}
\begin{proof}
We only prove that $P(-\gamma)\#\Gamma\cong P^{(\gamma)}$. The proof that $Q(-\gamma)\#\Gamma\cong Q^{(\gamma)}$ is similar. Recall that $$P(-\gamma)\#\Gamma=P(-\gamma)=P=\Ima(d)$$ as a set. Hence any element of $ P(-\gamma)\#\Gamma$ can be written as $xd$ where $x\in \bigoplus_{j\in J}R(\beta_j)$. We define the map 
\begin{align*}
\theta: P(-\gamma)\#\Gamma&\longrightarrow P^{(\gamma)}\\   
 xd&\longmapsto  x^{(\gamma)}d^{(\gamma)},
\end{align*}
where $x^{(\gamma)}\in \bigoplus_{j\in J}R\#\Gamma$ is defined by $x^{(\gamma)}_j=x_jp_{\gamma-\beta_j}~(j\in J)$. Clearly 
\begin{align*}
&xd=yd\\
\Leftrightarrow~~&(x-y)d=0\\
\Leftrightarrow~~&\sum_{i\in J}(x_i-y_i)d_{ij}=0~\forall j\in J~(\text{in }R)\\\
\Leftrightarrow~~&(\sum_{i\in J}(x_i-y_i)d_{ij})p_{\gamma-\beta_j}=0~\forall j\in J~(\text{in }R\#\Gamma)\\
\Leftrightarrow~~&\sum_{i\in J}(x_i-y_i)d_{ij}p_{\gamma-\beta_j}=0~\forall j\in J\\
\Leftrightarrow~~&\sum_{i\in J}(x_i-y_i)p_{\gamma-\beta_i}d_{ij}p_{\gamma-\beta_j}=0~\forall j\in J\\
\Leftrightarrow~~&\sum_{i\in J}(x^{(\gamma)}_i-y^{(\gamma)}_i)d^{(\gamma)}_{ij}=0~\forall j\in J\\
\Leftrightarrow~~&(x^{(\gamma)}-y^{(\gamma)})d^{(\gamma)}=0\\
\Leftrightarrow~~&x^{(\gamma)}d^{(\gamma)}=y^{(\gamma)}d^{(\gamma)}
\end{align*}
for any $x,y\in \bigoplus_{j\in J}R(\beta_j)$. Hence $\theta$ is well-defined and injective. 

Next we show that $\theta$ is surjective. Let $z\in P^{(\gamma)}=\Ima(d^{(\gamma)})$. Then there is a $y\in \bigoplus_{j\in J}R\#\Gamma$ such that $z=yd^{(\gamma)}$. For any $i\in J$ write $y_i=\sum_{\alpha\in\Gamma}y_{i,\alpha}p_\alpha\in R\#\Gamma$. Then
\begin{multline*}
(yd^{(\gamma)})_j
=\sum_{i\in J}y_id^{(\gamma)}_{ij}
=\sum_{i\in J}(\sum_{\alpha\in\Gamma}y_{i,\alpha}p_\alpha)d_{ij}p_{\gamma-\beta_j}
=\sum_{i\in J}\sum_{\alpha\in\Gamma}y_{i,\alpha}(d_{ij})_{\alpha-\gamma+\beta_j}p_{\gamma-\beta_j}
=\sum_{i\in J}y_{i,\gamma-\beta_i}d_{ij}p_{\gamma-\beta_j},
\end{multline*}
since $\deg(d_{ij})=\beta_j-\beta_i$. Hence we may assume that $y_{i,\alpha}=0$, for any $i\in J$ and $\alpha\neq \gamma-\beta_i$. Define $x\in \bigoplus_{j\in J}R(\beta_j)$ by $x_j=y_{j,\gamma-\beta_j}~(j\in J)$. Then clearly $y=x^{(\gamma)}$ and hence $z=yd^{(\gamma)}=x^{(\gamma)}d^{(\gamma)}=\theta(xd)$. Thus $\theta$ is surjective. 

It remains to show that $\theta$ is a module homomorphism. Clearly
\begin{equation}\label{EQ5.3}
    \theta(xd+yd)=\theta((x+y)d)=(x+y)^{(\gamma)}d^{(\gamma)}=x^{(\gamma)}d^{(\gamma)}+y^{(\gamma)}d^{(\gamma)}=\theta(xd)+\theta(yd),
\end{equation}
for any $x,y\in \bigoplus_{j\in J}R(\beta_j)$. Let now $x\in \bigoplus_{j\in J}R(\beta_j)$, $r\in R$ and $\alpha\in\Gamma$. Then
\begin{equation}\label{EQ5.4}
    \theta(rp_\alpha.xd)=\theta(r(xd)_{\alpha}).
\end{equation}
For any $j\in J$ we have
\begin{equation}\label{EQ5.5}
((xd)_{\alpha})_j=((xd)_j)_{\alpha}=\Big (\sum_{i\in J}x_id_{ij}\Big )_{\alpha}=\sum_{i\in J}(x_i)_{\alpha-\gamma+\beta_i}d_{ij}
\end{equation}
since $\deg((x_i)_{\delta}d_{ij})=\delta-\beta(i)+\gamma$ in $R(\beta(j)-\gamma)$. Define $y\in \bigoplus_{j\in J}R(\beta_j)$ by $y_j=r(x_j)_{\alpha-\gamma+\beta_j}~(j\in J)$. It follows from (\ref{EQ5.4}) and (\ref{EQ5.5}) that  
\begin{equation}\label{EQ5.6}
    \theta(rp_\alpha.xd)=\theta(yd)=y^{(\gamma)}d^{(\gamma)}.
\end{equation}
On the other hand we have
\begin{equation}\label{EQ5.7}
    rp_\alpha.\theta(xd)=(rp_\alpha).x^{(\gamma)}d^{(\gamma)}.
\end{equation}
Clearly 
\begin{equation}\label{EQ5.8}
    ((rp_\alpha).x^{(\gamma)})_j=(rp_\alpha).(x^{(\gamma)})_j=(rp_\alpha)(x_jp_{\gamma-\beta_j})=r(x_j)_{\alpha-\gamma+\beta_j}p_{\gamma-\beta_j}=y^{(\gamma)}_j,
\end{equation}
for any $j\in J$. Hence 
\begin{equation}\label{EQ5.9}
    (rp_\alpha).x^{(\gamma)}=y^{(\gamma)}.
\end{equation}
It follows from (\ref{EQ5.6}), (\ref{EQ5.7}) and (\ref{EQ5.8}) that 
\begin{equation}\label{EQ5.10}
    \theta(rp_\alpha.xd)=rp_\alpha.\theta(xd).
\end{equation}
In view of (\ref{EQ5.3}) and (\ref{EQ5.10}) we have shown that $\theta$ is a module homomorphism.
\end{proof}

\subsection{Smash products of graded Bergman algebras II}
Let $P$ be a non-zero graded  finitely generated projective $R$-module and set 
\[S:=R\big\langle e:\overline{P}\to \overline{P};~e^2=e\big\rangle.\]
We write $P$ as the image of an idempotent graded endomorphism $d$ of a free graded $R$-module of finite rank $\bigoplus_{j\in J}R(\beta_j)$. Clearly $\deg(d_{ij})=\beta_j-\beta_i$ for any $i,j\in J$. Recall from Lemma \ref{lempresS*} that $S$ can be obtained from $R$ by adjoining homogeneous generators $e_{ij}~(i,j\in J)$ of degree $\deg(e_{ij})=\beta_j-\beta_i$ and relations $de=e=ed$ and $ee=e$. Hence $S\cong R\langle X\rangle/I$ where $X=\{e_{ij}\mid i,j\in J\}$ and $I$ is the ideal of $R\langle X\rangle$ generated by the relations 
\begin{align*}
\sum_{k\in J}d_{ik}e_{kj}&=e_{ij}=\sum_{k\in J}e_{ik}d_{kj}&(i,j\in J),\\
\sum_{k\in J}e_{ik}e_{kj}&=e_{ij}&(i,j\in J).
\end{align*}
It follows from Lemma \ref{lempressmash} and Corollary \ref{corsmashkey} that $S\#\Gamma$ has the presentation 
 \begin{equation}\label{eqn6.16}
 \begin{split}
 S\#\Gamma=\Big\langle& rp_{\gamma},e_{ij}p_\gamma\quad(r\in R^{h};~ \gamma\in \Gamma;~i,j\in J)\mid{}\\
 &rp_\gamma+sp_\gamma=(r+s)p_\gamma\quad(r,s\in R_{\delta};~\gamma,\delta\in \Gamma),\\
 &rp_{\gamma}sp_{\delta}=rs_{\gamma-\delta} p_{\delta}\quad(r,s\in R^{h};~\gamma,\delta\in\Gamma ),\\
 &(e_{ij}p_\gamma)(1p_\gamma)=e_{ij}p_\gamma=(1p_{\gamma+\beta_j-\beta_i})(e_{ij}p_\gamma)\quad(i,j\in J;~ \gamma\in \Gamma),\\
&\sum_{k\in J}d_{ik}p_{\gamma+\beta_j-\beta_k}e_{kj}p_\gamma=e_{ij}p_\gamma=\sum_{k\in J}e_{ik}p_{\gamma+\beta_j-\beta_k}d_{kj}p_\gamma\quad(i,j\in J;~ \gamma\in \Gamma),\\
&\sum_{k\in J}e_{ik}p_{\gamma+\beta_j-\beta_k}e_{kj}p_\gamma=e_{ij}p_\gamma(i,j\in J;~ \gamma\in \Gamma)\Big\rangle.
\end{split}
\end{equation}
For any $\gamma\in \Gamma$ and $i,j\in J$ define $d_{ij}^{(\gamma)}:=d_{ij}p_{\gamma-\beta_j}$ and $e^{(\gamma)}_{ij}:=e_{ij}p_{\gamma-\beta_j}$. Using these conventions it follows from (\ref{eqn6.16}) that 

\begin{equation}\label{eqn6.17}
\begin{split}
    S\#\Gamma=\Big\langle& rp_{\gamma},e_{ij}^{(\gamma)}\quad(r\in R^{h};~ \gamma\in \Gamma;~i,j\in J)\mid{}\\
 &rp_\gamma+sp_\gamma=(r+s)p_\gamma\quad(r,s\in R_{\delta};~\gamma,\delta\in \Gamma),\\
 &rp_{\gamma}sp_{\delta}=rs_{\gamma-\delta} p_{\delta}\quad(r,s\in R^{h};~\gamma,\delta\in\Gamma ),\\
 &e_{ij}^{(\gamma)}(1p_{\gamma-\beta_j})=e_{ij}^{(\gamma)}=(1p_{\gamma-\beta_i})e_{ij}^{(\gamma)}\quad(i,j\in J;~ \gamma\in \Gamma),\\
&\sum_{k\in J}d_{ik}^{(\gamma)}e_{kj}^{(\gamma)}=e_{ij}^{(\gamma)}=\sum_{k\in J}e_{ik}^{(\gamma)}d_{kj}^{(\gamma)}\quad(i,j\in J;~ \gamma\in \Gamma),\\
&\sum_{k\in J}e_{ik}^{(\gamma)}e_{kj}^{(\gamma)}=e_{ij}^{(\gamma)}(i,j\in J;~ \gamma\in \Gamma)\Big\rangle.
\end{split}
 \end{equation}

For any finite subset $A\subseteq \Gamma$ we define the $k$-algebra $T_A$ by

\begin{equation}\label{eqn6.18}
\begin{split}
    T_A=\Big\langle& rp_{\beta}~(r\in R_\alpha;~\alpha+\beta,\beta\in\Gamma_A);~~~~e_{ij}^{(\gamma)}~(\gamma\in A;~i,j\in J)\mid{}\\
 &rp_\beta+sp_\beta=(r+s)p_\beta\quad(r,s\in R_{\alpha};~\alpha+\beta,\beta\in\Gamma_A),\\
 &rp_{\beta}sp_{\beta'}=rs_{\beta-\beta'} p_{\beta'}\quad(r\in R_{\alpha};~s\in R_{\alpha'};~\alpha+\beta,\beta, \alpha'+\beta',\beta'\in\Gamma_A),\\
 &e_{ij}^{(\gamma)}(1p_{\gamma-\beta_j})=e_{ij}^{(\gamma)}=(1p_{\gamma-\beta_i})e_{ij}^{(\gamma)}\quad(i,j\in J;~ \gamma\in A),\\
&\sum_{k\in J}d_{ik}^{(\gamma)}e_{kj}^{(\gamma)}=e_{ij}^{(\gamma)}=\sum_{k\in J}e_{ik}^{(\gamma)}d_{kj}^{(\gamma)}\quad(i,j\in J;~ \gamma\in A),\\
&\sum_{k\in J}e_{ik}^{(\gamma)}e_{kj}^{(\gamma)}=e_{ij}^{(\gamma)}\quad(i,j\in J;~ \gamma\in A)\Big\rangle,
\end{split}
 \end{equation}
where $\Gamma_A=\{\gamma-\beta_j\mid \gamma\in A,j\in J\}$. Clearly $T_A$ is a unital algebra with identity $\sum_{\gamma\in\Gamma_A}p_\gamma$. Moreover, $S\#\Gamma=\varinjlim_A T_A$, where $A$ runs through all finite subsets of $\Gamma$. We will show that the $T_A$'s are isomorphic to Bergman algebras.

Let $A\subseteq \Gamma$ a finite subset. We denote by $R\# \Gamma_A$ the subalgebra of $R\# \Gamma$ consisting of all sums of elements $rp_\beta$, where $r\in R_\alpha$ and $\alpha+\beta,\beta\in \Gamma_A$. Note that $R\# \Gamma_A$ is is a unital algebra with identity $\sum_{\gamma\in\Gamma_A}p_\gamma$.

\begin{lemma}
The algebra  $R\#\Gamma_A$ has the presentation 
 \begin{equation}\label{eqn6.19}
 \begin{split}
     R\#\Gamma_A=\Big\langle& rp_{\beta}\quad(r\in R_\alpha;~\alpha+\beta,\beta\in\Gamma_A)\mid{}\\
 &rp_\beta+sp_\beta=(r+s)p_\beta\quad(r,s\in R_{\alpha};~\alpha+\beta,\beta\in\Gamma_A),\\
 &rp_{\beta}sp_{\beta'}=rs_{\beta-\beta'} p_{\beta'}\quad(r\in R_{\alpha};~s\in R_{\alpha'};~\alpha+\beta,\beta, \alpha'+\beta',\beta'\in\Gamma_A)\Big\rangle.
 \end{split}
 \end{equation}   
\end{lemma}
\begin{proof}
See the proof of Lemma \ref{lemrelpress}.
\end{proof}

For any $\gamma\in A$ define the matrices 
\begin{align*}
d^{(\gamma)}&:=(d_{ij}^{(\gamma)})_{i,j\in J}=(d_{ij}p_{\gamma-\beta_j})_{i,j\in J}\in \Mat_{J\times J}(R\#\Gamma_A).
\end{align*}
A straightforward computation shows that the matrices $d^{(\gamma)}$ are idempotent. Hence they define idempotent endomorphisms of $\bigoplus_{j\in J}R\# \Gamma_A$. It follows that $P_A^{(\gamma)}:=(\bigoplus_{j\in J}R\# \Gamma_A)d^{(\gamma)}$ are non-zero finitely generated projective $R\#\Gamma_A$-modules. Define the $\K$-algebra
\begin{equation*}
  B_A:=R\#\Gamma_A\Big\langle f^{(\gamma)}:\overline{P_A^{(\gamma)}}\to\overline{P_A^{(\gamma)}};~(f^{(\gamma)})^2=f^{(\gamma)}~(\gamma\in A)\Big\rangle.
  \end{equation*}
We will show that $T_A\cong B_A$. By Lemma \ref{lempresS} and (\ref{eqn6.19}) we have
\begin{equation}\label{eqn6.20}
\begin{split}
B_A=\Big\langle& rp_{\beta}~(r\in R_\alpha;~\alpha+\beta,\beta\in\Gamma_A);~~~~f_{ij}^{(\gamma)}~(\gamma\in A;~i,j\in J)\mid{}\\
 &rp_\beta+sp_\beta=(r+s)p_\beta\quad(r,s\in R_{\alpha};~\alpha+\beta,\beta\in\Gamma_A),\\
 &rp_{\beta}sp_{\beta'}=rs_{\beta-\beta'} p_{\beta'}\quad(r\in R_{\alpha};~s\in R_{\alpha'};~\alpha+\beta,\beta, \alpha'+\beta',\beta'\in\Gamma_A),\\
&\sum_{k\in J}d_{ik}^{(\gamma)}f_{kj}^{(\gamma)}=f_{ij}^{(\gamma)}=\sum_{k\in J}f_{ik}^{(\gamma)}d_{kj}^{(\gamma)}\quad(i,j\in J;~ \gamma\in A),\\
&\sum_{k\in J}f_{ik}^{(\gamma)}f_{kj}^{(\gamma)}=f_{ij}^{(\gamma)}\quad(i,j\in J;~ \gamma\in A)\Big\rangle.
\end{split}
 \end{equation}

\begin{lemma}\label{lemmonomial*}
In $B_A$ the relations 

\begin{align}
 &f_{ij}^{(\gamma)}(1p_{\gamma-\beta_j})=f_{ij}^{(\gamma)}=(1p_{\gamma-\beta_i})f_{ij}^{(\gamma)}\quad(i,j\in J;~ \gamma\in A) \label{eqn6.21}
\end{align}
hold.
\end{lemma}
\begin{proof}
Clearly \begin{align*}
    f_{ij}^{(\gamma)}(1p_{\gamma-\beta_j})
        =&\Big(\sum_{k\in J}f_{ik}^{(\gamma)}d_{kj}^{(\gamma)}\Big)(1p_{\gamma-\beta_j})\\
    =&\sum_{k\in J}f_{ik}^{(\gamma)}(d_{kj}p_{\gamma-\beta_j})(1p_{\gamma-\beta_j})\\
=&\sum_{k\in J}f_{ik}^{(\gamma)}(d_{kj}p_{\gamma-\beta_j})\\
=&\sum_{k\in J}f_{ik}^{(\gamma)}d_{kj}^{(\gamma)}\\
    =&f_{ij}^{(\gamma)}.
\end{align*}
Similarly
\begin{align*}
    (1p_{\gamma-\beta_i})f_{ij}^{(\gamma)}
=& (1p_{\gamma-\beta_i})\Big (\sum_{k\in J}d_{ik}^{(\gamma)}f_{kj}^{(\gamma)}\Big)\\
=& \sum_{k\in J}(1p_{\gamma-\beta_i})(d_{ik}p_{\gamma-\beta_k})f_{kj}^{(\gamma)}\\
=& \sum_{k\in J}((d_{ik})_{\gamma-\beta_i-(\gamma-\beta_k)}p_{\gamma-\beta_k})f_{kj}^{(\gamma)}\\
=& \sum_{k\in J}(d_{ik}p_{\gamma-\beta_k})f_{kj}^{(\gamma)}\\
=& \sum_{k\in J}d_{ik}^{(\gamma)}f_{kj}^{(\gamma)}\\
=& f_{ij}^{(\gamma)}
\end{align*}
since $\deg(d_{ik})=\beta_k-\beta_i$.
\end{proof}

\begin{prop}\label{propbergsmash*}
Let $A\subseteq \Gamma$ be a finite subset. Then $T_A\cong B_A$.
\end{prop}
\begin{proof}
The proposition follows from (\ref{eqn6.18}), (\ref{eqn6.20}) and (\ref{eqn6.21}).
\end{proof}

The theorem below is the main result of this subsection. It follows directly from Proposition \ref{propbergsmash*} and the fact that $S\#\Gamma=\varinjlim_A T_A$, where $A$ runs through all finite subsets of $\Gamma$.
\begin{thm}\label{thmbergsmash*}
The $\K$-algebra $R\big\langle e:\overline{P}\to \overline{P};~e^2=e\big\rangle\#\Gamma$ is the direct limit of the Bergman algebras
\[R\#\Gamma_A\Big\langle f^{(\gamma)}:\overline{P_A^{(\gamma)}}\to\overline{P_A^{(\gamma)}};~(f^{(\gamma)})^2=f^{(\gamma)}~(\gamma\in A)\Big\rangle,\]
where $A$ runs through all finite subsets of $\Gamma$.
\end{thm}

We finish this subsection with a lemma which will be used in \S 7. For $\gamma\in \Gamma$ we define the $R\#\Gamma$-module $P^{(\gamma)}:=(\bigoplus_{j\in J}R\# \Gamma)d^{(\gamma)}$.

\begin{lemma}\label{lemtech_1*}
Let $\gamma\in A$. Then $P^{(\gamma)}\cong P(-\gamma)\#\Gamma$ as $R\#\Gamma$-modules.
\end{lemma}
\begin{proof}
See the proof of Lemma \ref{lemtech_1}.
\end{proof}

\section{The graded $\V$-monoid of a graded Bergman algebra}\label{sec7}
\subsection{Basic definitions and results}\label{subsec7.1}
Let $R$ be a not necessarily unital ring. Recall that a $R$-module $M$ is called {\it unital} if $RM=M$. We denote by $R$-$\Mod$ the category of unital $R$-modules. Furthermore, we denote by $R$-$\Mod_{\proj}$ the full subcategory of $R$-$\Mod$ whose objects are the projective objects of $R$-$\Mod$ that are finitely generated as a $R$-module. If $R$ has local units, we define
\[\V(R)=\{[P]\mid P\in R\text{-}\Mod_{\proj}\}\]
where $[P]$ denotes the isomorphism class of $P$ as a $R$-module. $\V(R)$ becomes an abelian monoid by defining $[P]+[Q]=[P\oplus Q]$. It is well-known that $\V$ is a functor that commutes with direct limits. There is a different, but equivalent definition of $\V$ using idempotent matrices over $R$, cf. \cite[Subsection 4A]{ara-hazrat-li-sims}.

Let $R$ now be a $\Gamma$-graded ring. Recall that a $R$-module $M$ is called {\it $\Gamma$-graded} if there is a decomposition $M=\bigoplus_{\gamma\in\Gamma}M_\gamma$ such that $R_\alpha M_\gamma\subseteq M_{\alpha\gamma}$ for any $\alpha, \gamma\in \Gamma$. We denote by $R$-$\GR$ the category of $\Gamma$-graded unital $R$-modules with morphisms the $R$-module homomorphisms that preserve grading. Moreover, we denote by $R$-$\GR_{\proj}$ the full subcategory of $R$-$\GR$ whose objects are the projective objects of $R$-$\GR$ that are finitely generated as a $R$-module. If $R$ has graded local units, we define
\[\V^{\gr}(R)=\{[P] \mid P \in R\text{-}\GR_{\proj}\}\]
where $[P]$ denotes the isomorphism class of $P$ as a graded $R$-module. $\V^{\gr}(R)$ becomes a $\Gamma$-monoid by defining $[P]+[Q]=[P\oplus Q]$ and $\gamma.[P]=[P(\gamma)]$. 

\begin{prop}[{\cite[Proposition 66]{Raimund2}}]\label{propsmash}
Let $R$ be a $\Gamma$-graded ring with graded local units. Then $R$-$\GR_{\proj}\cong R\#\Gamma$-$\Mod_{\proj}$ by an isomorphism that commutes with direct sums. It follows that $\V^{\gr}(R)\cong \V(R\#\Gamma)$.
\end{prop}

Let $R$ be a unital ring and $\epsilon\in \Mat_{n\times n}(R)$ an idempotent matrix. Then $\epsilon$ defines an endomorphism of the $R$-module $R^n$, which we also denote by $\epsilon$. Clearly the image $M(\epsilon)=\Ima\epsilon$ is a finitely generated projective $R$-module (since $R^n=\Ker \epsilon\oplus\Ima\epsilon$). Now let $M$ be a finitely generated projective $R$-module. Then $M\oplus N=R^n$ for some $R$-module $N$ and $n\in \N$. Define an endomorphism $\epsilon$ of $R^n$ by $\epsilon|_M=\id_M$ and $\epsilon|_N=0$. Then clearly $M=\Ima\epsilon$ and $\epsilon$ is idempotent. We denote the idempotent matrix in $\Mat_{n\times n}(R)$ corresponding to $\epsilon$ also by $\epsilon$. Clearly $M(\epsilon)=M$. Hence any finitely generated projective $R$-module equals $M(\epsilon)$ for some idempotent square matrix $\epsilon$ over $R$.

Now let $R$ be a $\Gamma$-graded ring and $\epsilon\in \Mat_{n\times n}(R)$ an idempotent matrix with homogeneous entries such $\deg(\epsilon_{ij})=\gamma_{j}-\gamma_i~(1\leq i,j\leq n)$ where $\gamma_1,\dots,\gamma_n\in \Gamma$. Then $\epsilon$ defines an endomorphism of the graded $R$-module $\bigoplus_{k=1}^nR(\gamma_k)$, which we also denote by $\epsilon$. Clearly the image $M(\epsilon)=\Ima\epsilon$ is a graded finitely generated projective  $R$-module (since $\bigoplus_{k=1}^nR(\gamma_k)=\Ker \epsilon\oplus\Ima\epsilon$). Now let $M$ be a graded finitely generated projective $R$-module. Then $M\oplus N=\bigoplus_{k=1}^nR(\gamma_k)$ for some graded $R$-module $N$ and $\gamma_1,\dots,\gamma_n\in \Gamma$. Define an endomorphism $\epsilon$ of $\bigoplus_{k=1}^nR(\gamma_k)$ by $\epsilon|_M=\id_M$ and $\epsilon|_N=0$. Then clearly $M=\Ima\epsilon$ and $\epsilon$ is idempotent. We denote the idempotent matrix in $\Mat_{n\times n}(R)$ corresponding to $\epsilon$ also by $\epsilon$. Clearly $\deg(\epsilon_{ij})=\gamma_{j}-\gamma_i~(1\leq i,j\leq n)$ and $M(\epsilon)=M$. Hence any graded finitely generated projective $R$-module equals $M(\epsilon)$ for some idempotent square matrix $\epsilon\in\Mat_{n\times n}(R)$ such that $\deg(\epsilon_{ij})=\gamma_{j}-\gamma_i~(1\leq i,j\leq n)$ where $\gamma_1,\dots,\gamma_n\in \Gamma$.

\begin{lemma}\label{lemidemsmash_1}
Let $R$ be a $\Gamma$-graded ring and $\epsilon\in \Mat_{n\times n}(R)$ an idempotent matrix with homogeneous entries such $\deg(\epsilon_{ij})=\gamma_{j}-\gamma_i~(1\leq i,j\leq n)$ where $\gamma_1,\dots,\gamma_n\in \Gamma$. Then $M(\epsilon)\#\Gamma\cong M(\epsilon\#\Gamma)$, where $\epsilon\#\Gamma\in \Mat_{n\times n}(R)$ is defined by $(\epsilon\#\Gamma)_{ij}=\epsilon_{ij}p_{-\gamma_j}~(1\leq i,j\leq n)$.
\end{lemma}
\begin{proof}
See the proof of Lemma \ref{lemtech_1}.
\end{proof}

\begin{lemma}\label{lemidemsmash_2}
Let $\phi:R\to S$ be a ring homomorphism and $\epsilon\in \Mat_{n\times n}(R)$ an idempotent matrix. Then $S\otimes_RM(\epsilon)\cong M(\phi(\epsilon))$, where $\phi(\epsilon)\in \Mat_{n\times n}(S)$ is obtained from $\epsilon$ by applying $\phi$ to each entry.
\end{lemma}
\begin{proof}
Recall that as an abelian monoid, $S\otimes_RM(\epsilon)$ has the presentation
\begin{equation}\label{defrel1}
\begin{split}
\Big\langle 
s\otimes m~(s\in S,m\in M(\epsilon))\mid{}& (s_1+s_2)\otimes m=s_1\otimes m+s_2\otimes m~(s_1,s_2\in S;~m\in M(\epsilon)),\\
&s\otimes (m_1+m_2)=s\otimes m_1+s\otimes m_2~(s\in S;~m_1,m_2\in M(\epsilon)),\\
& s.r\otimes m=s\otimes r.m ~(s\in S;~r\in R;~m\in M(\epsilon))
\Big\rangle.
\end{split}
\end{equation}
Clearly for any $k,l\in \N$, $\phi$ induces a map $\Mat_{k\times l}(R)\to \Mat_{k\times l}(S)$, which we also denote by $\phi$. We define a monoid homomorphism $\theta$ from the free abelian monoid $\big\langle s\otimes m~(s\in S,m\in M(\epsilon))\big\rangle$ to the abelian monoid $M(\phi(\epsilon))$ by 
\[\theta(s\otimes m)=s\phi(m)\]
for any $s\in S$ and $m\in M(\epsilon)$. Note that if $m\in M(\epsilon)$, then $m=x\epsilon$ for some $x\in R^n$ and hence $\phi(m)=\phi(x)\phi(\epsilon)\in M(\phi(\epsilon))$. One checks routinely that $\theta$ preserves the defining relations in presentation (\ref{defrel1}). Hence $\theta$ induces a monoid homomorphism $\eta:S\otimes_RM(\epsilon)\to M(\phi(\epsilon))$. 

First we show that $\eta$ surjective. For any $1\leq j\leq n$, let $\alpha_j$ be the element of $R^n$ whose $j$-th component is $1$ and whose other components are $0$. Define $\beta_1,\dots,\beta_n\in S^n$ similarly. Let now $m\in M(\phi(\epsilon))$. Then there is a $y\in S^n$ such that $m=y\phi(\epsilon)$. It follows that $m=y\phi(\epsilon)=\sum_{k=1}^ny_k\beta_k\phi(\epsilon)=\sum_{k=1}^n\eta(y_k\otimes \alpha_k\epsilon)$. Hence $\eta$ is surjective.

Next we show that $\eta$ is injective. Suppose that $\eta(s\otimes m)=\eta(s'\otimes m')$. Then $s\phi(m)=s'\phi(m')$ in $S^n$, i.e. $s\phi(m_k)=s'\phi(m'_k)~(1\leq k\leq n)$ in $S$. It follows that 
\begin{multline*}
    s\otimes m
    =s\otimes \sum_{k=1}^nm_k\alpha_k
    =\sum_{k=1}^ns\otimes m_k\alpha_k
=\sum_{k=1}^ns\phi(m_k)\otimes \alpha_k= \\ =\sum_{k=1}^ns'\phi(m'_k)\otimes \alpha_k
    =\sum_{k=1}^ns'\otimes m'_k\alpha_k
    =s'\otimes \sum_{k=1}^nm'_k\alpha_k
    s'\otimes m'.
\end{multline*}
Hence $\eta$ is injective. We leave it to the reader to show that $\eta$ preserves the action of $S$.
\end{proof}

\begin{lemma}\label{lemidemsmash_3}
Let $\phi:R\to S$ be a homomorphism of $\Gamma$-graded rings and $\epsilon\in \Mat_{n\times n}(R)$ an idempotent matrix with homogeneous entries such $\deg(\epsilon_{ij})=\gamma_{j}-\gamma_i~(1\leq i,j\leq n)$ where $\gamma_1,\dots,\gamma_n\in \Gamma$. Then $S\otimes_RM(\epsilon)\cong M(\phi(\epsilon))$ as graded $S$-modules, where $\phi(\epsilon)\in \Mat_{n\times n}(S)$ is obtained from $\epsilon$ by applying $\phi$ to each entry.
\end{lemma}
\begin{proof}
The proof of the previous lemma shows that there is an isomorphism $\eta:S\otimes_RM(\epsilon)\to M(\phi(\epsilon))$ of $S$-modules. We leave it to the reader to check that $\eta$ preserves the grading.
\end{proof}

\subsection{The graded $\V$-monoid of a graded Bergman algebra I}\label{gdhdhdh}
 Let $R$ be a $\Gamma$-graded $k$-algebra where $\Gamma$ is an abelian group and $\K$ is a graded field concentrated in degree zero. Moreover, let $P$ and $Q$ be non-zero graded finitely generated projective $R$-modules. Set 
\[S:=R\big\langle h,h^{-1}:\overline{P}\cong \overline{Q}\big\rangle.\]

\begin{thm}\label{thmbergV}
There is a monoid isomorphism 
\[\V^{\gr}(S) \cong \V^{\gr}(R)/\big \langle [P(\gamma)]=[Q(\gamma)], \gamma\in \Gamma\big \rangle,\] given by tensoring. 
\end{thm}
\begin{proof}
It follows from Proposition \ref{propsmash}, Theorem \ref{thmbergsmash}, \cite[Theorem 5.2]{bergman74} and Lemma \ref{lemtech_1} that
\begin{align*}
\V^{\gr}(S)
\cong~& \V(S\#\Gamma)\\
\cong~& \V(\varinjlim R\#\Gamma_A\Big\langle g^{(\gamma)}, (g^{(\gamma)})^{-1}:\overline{P_A^{(\gamma)}}\cong\overline{Q_A^{(\gamma)}}~(\gamma\in A)\Big\rangle)\\
\cong~& \varinjlim \V(R\#\Gamma_A\Big\langle g^{(\gamma)}, (g^{(\gamma)})^{-1}:\overline{P_A^{(\gamma)}}\cong\overline{Q_A^{(\gamma)}}~(\gamma\in A)\Big\rangle)\\
\cong~& \varinjlim \V(R\#\Gamma_A)/\big([P_A^{(\gamma)}]=[Q_A^{(\gamma)}]~~(\gamma\in A)\big)\\
\cong~& \V(R\#\Gamma)/\big([P^{(\gamma)}]=[Q^{(\gamma)}]~~(\gamma\in \Gamma)\big)\\
\cong~& \V(R\#\Gamma)/\big([P(\gamma)\#\Gamma]=[Q(\gamma)\#\Gamma]~~(\gamma\in \Gamma)\big)\\
\cong~& \V^{\gr}(R)/\big([P(\gamma)]=[Q(\gamma)]~~(\gamma\in \Gamma)\big).
\end{align*}
It remains to show that the isomorphism $\V^{\gr}(R)/\big([P(\gamma)]=[Q(\gamma)]~(\gamma\in \Gamma)\big)\cong\V^{\gr}(S)$ obtained above is given by tensoring. Let $M$ be a graded finitely generated projective $R$-module. By \S\ref{subsec7.1}, $M=M(\epsilon)$ for some idempotent matrix $\epsilon\in\Mat_{n\times n}(R)$ such that $\deg(\epsilon_{ij})=\gamma_{j}-\gamma_i~(1\leq i,j\leq n)$ where $\gamma_1,\dots,\gamma_n\in \Gamma$. Applying the isomorphisms above in reversed order to $[M]\in \V^{\gr}(R)/\big([P(\gamma)]=[Q(\gamma)]~~(\gamma\in \Gamma)\big)$ yields (in view of Lemmas \ref{lemidemsmash_1} and \ref{lemidemsmash_2})
\begin{multline*}
  [M(\epsilon)]\mapsto [M(\epsilon\#\Gamma)]\mapsto [M(\epsilon\#\Gamma_A)]\mapsto [M(\phi_A(\epsilon\#\Gamma_A))]\mapsto[M(\psi(\epsilon\#\Gamma))]=[M(\xi(\epsilon)\#\Gamma)]\mapsto[M(\xi(\epsilon))],  
\end{multline*}
where $A$ is chosen such that all entries of $\epsilon\#\Gamma$ lie in $R\#\Gamma_A$, $\epsilon\#\Gamma_A$ is just $\epsilon\#\Gamma$ viewed as matrix over $R\#\Gamma_A$, and 
\begin{align*}
\phi_A:{}&{}R\#\Gamma_A\to R\#\Gamma_A\Big\langle g^{(\gamma)}, (g^{(\gamma)})^{-1}:\overline{P_A^{(\gamma)}}\cong\overline{Q_A^{(\gamma)}}~(\gamma\in A)\Big\rangle,\\
\psi:{}&{}R\#\Gamma\to S\#\Gamma,\\
\xi:{}&{}R\to S   
\end{align*}
are the canonical homomorphisms. It follows from Lemma \ref{lemidemsmash_3} that the isomorphism \[\V^{\gr}(R)/\big([P(\gamma)]=[Q(\gamma)]~~(\gamma\in \Gamma)\big)\cong \V^{\gr}(S),\] obtained above maps $[M(\epsilon)]$ to $[M(\xi(\epsilon))]=[S\otimes_R M(\epsilon)]$.
\end{proof}

\begin{cor}\label{corbergV}
There is a $\Gamma$-monoid isomorphism $\V^{\gr}(S) \cong \V^{\gr}(R)/\big \langle [P]=[Q]\big \rangle$ given by tensoring. 
\end{cor}
\begin{proof}
By Theorem \ref{thmbergV} there is a monoid isomorphism $\phi:\V^{\gr}(R)/\big([P(\gamma)]=[Q(\gamma)]~(\gamma\in \Gamma)\big)\cong\V^{\gr}(S)$ given by tensoring. Clearly the monoid congruence on $\V^{\gr}(R)$ generated by the relations $[P(\gamma)]=[Q(\gamma)]~(\gamma\in \Gamma)$ is a $\Gamma$-monoid congruence. It follows that $\V^{\gr}(R)/\big([P(\gamma)]=[Q(\gamma)]~(\gamma\in \Gamma)\big)$ is a $\Gamma$-monoid since $\V^{\gr}(R)$ is a $\Gamma$-monoid. Since tensoring preserves the action of $\Gamma$, $\phi:\V^{\gr}(R)/\big([P(\gamma)]=[Q(\gamma)]~(\gamma\in \Gamma)\big)\cong\V^{\gr}(S)$ is an isomorphism of $\Gamma$-monoids. Clearly the relation $[P]=[Q]$ generates the same $\Gamma$-monoid congruence on $\V^{\gr}(R)$ as the relations $[P(\gamma)]=[Q(\gamma)]~(\gamma\in \Gamma)$. Hence $\V^{\gr}(R)/\big([P(\gamma)]=[Q(\gamma)]~(\gamma\in \Gamma)\big)=\V^{\gr}(R)/\big([P]=[Q]\big)$.
\end{proof}

\begin{example}
Consider the Leavitt algebra $L_\K(n,n+k)$. As already observed by Bergman~\cite{bergman74}, $$L_\K(n,n+k)\cong  \K \langle h, h^{-1}: \overline{\K^n} \cong \overline{\K^{n+k}} \rangle.$$ 
Then by \cite[Theorem 5.2]{bergman74} $$\mathcal V(L_\K(n,n+k))\cong \mathcal V(\K)/\langle [\K^n]=[\K^{n+k}]\rangle\cong M_{n,n+k}.$$ Here, for $n,k \in \mathbb{N}^+$, $M_{n,n+k}$ is the finite commutative monoid
$$M_{n,n+k} := \big \{ 0, x, 2x, \dots , nx , \dots, (n+k-1)x \big \} , \ \  \ \ \mbox{with relation} \ (n+k)x = nx.$$

Returning to the graded setting, we have 
$$L_\K(n,n+k)\cong_{\gr}  \K \langle h, h^{-1}: \overline{\K^n} \cong_{\gr} \overline{\K^{n+k}(-1)} \rangle,$$
and by Corollary~\ref{corbergV} 
$$\mathcal V^{\gr}(L_\K(n,n+k))\cong \mathcal V^{\gr}(\K)/\langle [\K^n]=[\K^{n+k}(-1)]\rangle\cong \big \langle 1, (n/n+k)^i, i\in \mathbb Z \big \rangle\subseteq \mathbb Q.$$ 

Since $[L_\K(n,n+k))]$ represents $1$ in this monoid, which is then a strongly order unit, by Proposition~\ref{mnhygh}, we conclude that $L_\K(n,n+k)$ is a strongly graded ring. (This can also be obtained from looking at the generators and relations of the algebra.) Therefore, an application of Dade's theorem, imply that  
$\mathcal V^{\gr}(L_\K(n,n+k)) \cong  \mathcal V(L_\K(n,n+k)_0)$, which in turn implies that the zero component ring $L_\K(n,n+k)_0$ has IBN. 

\end{example}

\subsection{The graded $\V$-monoid of a graded Bergman algebra II}
 Let $R$ be a $\Gamma$-graded $k$-algebra where $\Gamma$ is an abelian group and $k$ a field concentrated in degree zero. Moreover, let $P$ be a non-zero graded finitely generated projective $R$-modules. Set 
\[S:=R\big\langle e:\overline{P}\to \overline{P};~e^2=e\big\rangle.\]

\begin{thm}\label{thmbergV*}
There is a monoid isomorphism 
\[\V^{\gr}(S) \cong \V^{\gr}(R)\Big\langle[P_1^{(\gamma)}],[P_2^{(\gamma)}]~(\gamma\in A) \mid [P(\gamma)]=[P_1^{(\gamma)}]+[P_2^{(\gamma)}]~(\gamma\in \Gamma)\Big\rangle,\] given by tensoring.
\end{thm}
\begin{proof}
It follows from Proposition \ref{propsmash}, Theorem \ref{thmbergsmash*}, \cite[Theorem 5.1]{bergman74} and \ref{lemtech_1*} that
\begin{align*}
\V^{\gr}(S)
\cong~& \V(S\#\Gamma)\\
\cong~& \V(\varinjlim R\#\Gamma_A\Big\langle f^{(\gamma)}:\overline{P_A^{(\gamma)}}\to\overline{P_A^{(\gamma)}};~(f^{(\gamma)})^2=f^{(\gamma)}~(\gamma\in A)\Big\rangle)\\
\cong~& \varinjlim \V(R\#\Gamma_A\Big\langle f^{(\gamma)}:\overline{P_A^{(\gamma)}}\to\overline{P_A^{(\gamma)}};~(f^{(\gamma)})^2=f^{(\gamma)}~(\gamma\in A)\Big\rangle)\\
\cong~& \varinjlim \V(R\#\Gamma_A)\big\langle[P_1^{(\gamma)}],[P_2^{(\gamma)}]~(\gamma\in A)\mid [P_A^{(\gamma)}]=[P_1^{(\gamma)}]+[P_2^{(\gamma)}]~(\gamma\in A)\big\rangle\\
\cong~& \V(R\#\Gamma)\big\langle[P_1^{(\gamma)}],[P_2^{(\gamma)}]~(\gamma\in \Gamma)\mid [P^{(\gamma)}]=[P_1^{(\gamma)}]+[P_2^{(\gamma)}]~(\gamma\in \Gamma)\big\rangle\\
\cong~& \V(R\#\Gamma)\big\langle[P_1^{(\gamma)}],[P_2^{(\gamma)}]~(\gamma\in \Gamma)\mid [P(\gamma)\#\Gamma]=[P_1^{(\gamma)}]+[P_2^{(\gamma)}]~(\gamma\in \Gamma)\big\rangle\\
\cong~& \V^{\gr}(R)\big\langle[P_1^{(\gamma)}],[P_2^{(\gamma)}]~(\gamma\in \Gamma)\mid [P(\gamma)]=[P_1^{(\gamma)}]+[P_2^{(\gamma)}]~(\gamma\in \Gamma)\big\rangle.
\end{align*}
It remains to show that the isomorphism $\V^{\gr}(R)\big\langle[P_1^{(\gamma)}],[P_2^{(\gamma)}]~(\gamma\in A)\mid [P(\gamma)]=[P_1^{(\gamma)}]+[P_2^{(\gamma)}]~(\gamma\in \Gamma)\big\rangle\cong\V^{\gr}(S)$ is given by tensoring. But that follows from Lemmas \ref{lemidemsmash_1}, \ref{lemidemsmash_2} and \ref{lemidemsmash_3}, see the proof of Theorem \ref{thmbergV}.
\end{proof}

\begin{cor}\label{corbergV*}
There is a $\Gamma$-monoid isomorphism $\V^{\gr}(S) \cong \V^{\gr}(R)\big\langle[P_1],[P_2]\mid [P]=[P_1]+[P_2]\big\rangle$ given by tensoring.
\end{cor}
\begin{proof}
See the proof of Corollary \ref{corbergV}.
\end{proof}

\section{Realisation of $\Gamma$-monoids as non-stable $K$-theory of graded rings}\label{secmainres}

Throughout this section, $\Gamma$ is an abelian group, $\K$ is a $\Gamma$-graded field concentrated in degree zero, and $R$ is a graded $\K$-algebra.  Recall from \S\ref{gammamoni} that if $(M,i)$ is a pointed $\Gamma$-monoid with $i\in M$ an $\Gamma$-order unit, then $i$ is a strong order unit if for any $m \in M$, we have $m\leq ki$ for some $k\in \mathbb N$. Furthermore,   $i$ is an invariant order unit if ${}^\gamma i =i$ for any $\gamma \in \Gamma$.

\begin{thm}\label{mainalibm}
Let $(M, i)$ be a pointed conical $\Gamma$-monoid with a distinguished order unit $i$. Then there is a hereditary graded $\K$-algebra $R$ such that $\phi:  (M,i) \cong (\V^{\gr}(R),[R])$ as pointed $\Gamma$-monoids. Furthermore, $R$ has a weak universal property: if $S$ is a $\Gamma$-graded $\K$-algebra with $\Gamma$-monoid homomorphism $\psi:  (M,i) \rightarrow (\V^{\gr}(S),[S])$, then there is a (nonunique) graded $\K$-algebra homomorphism $R\rightarrow S$ such that the following diagram commutes. 
\[\xymatrix{
(M,i) \ar[rr]^{\phi}\ar[dr]_{\psi}& & (\V^{\gr}(R),[R]) \ar@{.>}[dl]^{-\otimes_R S}
	 \\
 & (\V^{\gr}(S),[S]) }\]

Furthermore, 
\begin{enumerate}[\upshape(1)]
    \item $R$ is a strongly graded ring if and only if 
$i$ is a strong order unit. 
    \item $R$ is a crossed-product ring if and only if $i$ is an invariant order unit. 
\end{enumerate}

\end{thm}
\begin{proof}
Let $X=\{i,p_\phi,q_\phi\mid \phi\in\Phi\}$ be a set of non-zero generators for the $\Gamma$-monoid $M$ such that for each $\phi\in\Phi$ there are $\gamma_{\phi,1},\dots,\gamma_{\phi,n_\phi}\in\Gamma$ such that 
\begin{equation}\label{classedfp}
  p_\phi+q_\phi=\sum_{k=1}^{n_\phi}{}^{\gamma_{\phi,k}}i.  
\end{equation}
Let $Y=\{u_\psi=v_\psi\mid \psi\in\Psi\}$ be a set of relations, including the relations (\ref{classedfp}), such that $M=\langle X\mid Y\rangle$. Since $M$ is conical, we may assume that $u_\psi,v_\psi\neq 0$ for any $\psi\in \Psi$. Let $Z$ be the set of all pairs $(A,B)$,  where $A$ is a finite subset of $\Phi$ and $B$ is a finite subset of $\Psi$ such that 
\begin{enumerate}[(I)]
\item in the relations $u_\psi=v_\psi~(\psi\in B)$ only the generators $i,p_\phi,q_\phi~(\phi\in A)$ appear,
\item the relations $u_\psi=v_\psi~(\psi\in B)$ include the relations $p_\phi+q_\phi=\sum_{k=1}^{n_\phi}{}^{\gamma_{\phi,k}}i~(\phi\in A)$.    
\end{enumerate}
Define a partial order $\leq$ on $Z$ by $(A,B)\leq (A',B')$ if $A\subseteq A'$ and $B\subseteq B'$. Clearly $(Z,\leq)$ is a directed set. For any $(A,B)\in Z$ let $M_{(A,B)}=\langle X_{(A,B)}\mid Y_{(A,B)}\rangle$,  where $X_{(A,B)}=\{i,p_\phi,q_\phi\mid\phi\in A\}$ and $Y_{(A,B)}=\{u_\psi=v_\psi\mid \psi\in B\}$. If $(A,B)\leq (A',B')$, then clearly there is a canonical $\Gamma$-monoid homomorphism $\alpha_{(A,B)}^{(A',B')}:M_{(A,B)}\to M_{(A',B')}$. One checks easily that $M$ is the direct limit of the direct system $\{M_{(A,B)},\alpha_{(A,B)}^{(A',B')}\mid (A,B)\leq (A',B')\in Z\}$ in the category of $\Gamma$-monoids. 

Now fix an $(A,B)\in Z$. Let $R'_{(A,B)}$ be the $\Gamma$-graded $\K$-algebra obtained from $\K$ by adjoining idempotent endomorphisms \[e_\phi:\overline{\bigoplus_{1\leq k\leq n_\phi}\K(\gamma_{\phi,k})}\longrightarrow \overline{\bigoplus_{1\leq k\leq n_\phi}\K(\gamma_{\phi,k})},\] for any $\phi\in A$. Then
\[\V^{\gr}(R'_{(A,B)})= \Big \langle I,P_\phi,Q_\phi~(\phi\in A)\mid P_\phi+Q_\phi=\sum_{k=1}^{n_\phi}{}^{\gamma_{\phi,k}} I~(\phi\in A)\Big \rangle,\] 
by Corollary \ref{corbergV*}. Let $R_{(A,B)}$ be the $\Gamma$-graded $\K$-algebra obtained from $R'_{(A,B)}$ by adjoining isomorphisms $\overline{U_\psi}\cong \overline{V_\psi}~(\psi\in B)$, where $U_{\psi}$ and $V_{\psi}$ are the graded $R'_{(A,B)}$-modules corresponding to $u_\psi$ and $v_\psi$, respectively. Then $\V^{\gr}(R_{(A,B)})\cong M_{(A,B)}$ by Corollary \ref{corbergV}, and this isomorphism maps $[R_{(A,B)}]$ to $i$. 

If $(A,B)\leq (A',B')$, then clearly there is a canonical graded $\K$-algebra homomorphism $\beta_{(A,B)}^{(A',B')}:R_{(A,B)}\to R_{(A',B')}$. Let $R$ be the direct limit of the direct system $\big \{R_{(A,B)},\beta_{(A,B)}^{(A',B')}\mid (A,B)\leq (A',B')\in Z\big\}$ in the category of graded $\K$-algebras. Since $\V^{\gr}$ commutes with direct limits, we obtain
\[\V^{\gr}(R)=\V^{\gr}(\varinjlim R_{(A,B)})=\varinjlim \V^{\gr}(R_{(A,B)})\cong\varinjlim M_{(A,B)}=M.\]
Clearly the isomorphism $\V^{\gr}(R)\cong M$ maps $[R]$ to $i$.

It remains to show that $R$ is hereditary. We closely follow the approach used in Theorem~3.7 in \cite{aragoodearl}.  Choose an ordinal $\lambda$ and a bijection $\alpha\to \phi_\alpha$ from $[0,\lambda)$ to $\Phi$. We build graded $\K$-algebras $R_\alpha$ for $\alpha\in [0,\lambda)$ as follows.
\begin{enumerate}[(a)]
    \item Start with $R_0=\K$.
    
    \item For $\alpha\in [0,\lambda)$, let $R_{\alpha+1}$ be obtained from $R_\alpha$ by adjoining an idempotent endomorphism \[e_{\phi_\alpha}:\overline{\bigoplus_{1\leq k\leq n_{\phi_\alpha}}\K(\lambda_{\phi_\alpha,k})}\to\overline{\bigoplus_{1\leq k\leq n_{\phi_\alpha}}\K(\lambda_{\phi,k})}.\]
    
    \item If $\beta\leq \lambda$ is a limit ordinal and $R_\alpha$ has been defined for all $\alpha < \beta$, take $R_\beta$ to be the
direct limit of $(R_\alpha)_{\alpha<\beta}$.
\end{enumerate}
By \cite[Theorem 3.4]{bergman78} the algebra $R_\lambda$ is hereditary. Now choose an ordinal $\mu>\lambda$ and a bijection $\alpha\to \psi_\alpha$ from $[\lambda,\mu)$ to $\Psi$. We build graded $\K$-algebras $R_\alpha$ for $\alpha\in [\lambda,\mu)$ as follows.
\begin{enumerate}[(a)]
\setcounter{enumi}{3}
    \item Start with the algebra $R_\lambda$ constructed above.
    
    \item For $\alpha\in [\lambda,\mu)$, let $R_{\alpha+1}$ be obtained from $R_\alpha$ by adjoining an isomorphism $\overline{U_{\psi_\alpha}}\cong \overline{V_{\psi_\alpha}}$, where $U_{\psi_\alpha}$ and $V_{\psi_\alpha}$ are the graded $R_{\alpha}$-modules corresponding to $u_{\psi_\alpha}$ and $v_{\psi_\alpha}$, respectively. 
    
    \item If $\beta\in[\lambda,\mu)$ is a limit ordinal and $R_\alpha$ has been defined for all $\alpha < \beta$, take $R_\beta$ to be the
direct limit of $(R_\alpha)_{\lambda\leq \alpha<\beta}$. 
\end{enumerate}
By \cite[Theorem 3.4]{bergman78} the algebra $R_\mu$ is hereditary. Clearly $R_\mu\cong R$.

In order to prove the weak universality of the $\K$-algebra $R$, suppose that $S$ is a $\Gamma$-graded $\K$-algebra with $\Gamma$-monoid homomorphism $\psi:  (M,i) \rightarrow (\V^{\gr}(S),[S])$. Consider the following diagram, which gives a $\Gamma$-monoid homomorphism, $f_{(A,B)}$, for any $(A,B)\in Z$. 
\begin{equation}\label{hfgfhs12}
\xymatrix{
\varinjlim  (M_{(A,B)},i) \cong (M,i) \ar[rr]^{\psi}& & (\V^{\gr}(S),[S]) \\
 & (M_{(A,B)},i)\ar[ul]
 \ar@{.>}[ur]_{f_{(A,B)}}}
 \end{equation}

Since  $M_{(A,B)}\cong \V^{\gr}(R_{(A,B)})$, and $R_{(A,B)}$ is the universal ring such that
$\overline{P_\phi} \oplus \overline{Q_\phi}=\sum_{k=1}^{n_\phi}{}^{\gamma_{\phi,k}} I~(\phi\in A)$ and $\overline{U_\psi}\cong \overline{V_\psi}~(\psi\in B)$ and the images of these modules in $\V^{\gr}(S)$ give the same relations, it follows that there is a $\Gamma$-graded homomorphism $R_{(A,B)} \rightarrow S$ which induces 
$f_{(A,B)}= -\otimes_{R_{(A,B)}} S$ on the level of monoids. Since the direct limit of $R_{(A,B)}$, $(A,B)\in Z$ is $R$ and $\V^{\gr}$ commutes with direct limits, 
\[\xymatrix{
\varinjlim  R_{(A,B)} \cong R \ar[rr]^{\psi_1}& & S \\
 & R_{(A,B)}\ar[ul]
 \ar@{.>}[ur]_{f_{(A,B)}}}\]
 we obtain that $\overline {\psi_1}: \V^{\gr}(R)\rightarrow \V^{\gr}(S)$ coincides with $\psi$ in Diagram~\ref{hfgfhs12}.

The second part of the theorem, statements (1) and (2), immediately follow from~Proposition~\ref{mnhygh}.
\end{proof}

If $M$ is generated by a finite number of elements whose sum is the order unit $i$, then the construction of the algebra $R$ in the proof of Theorem \ref{mainalibm} above simplifies a bit, as it is demonstrated in the proof of Theorem~\ref{mainalibm_2} below.

\begin{thm}\label{mainalibm_2}
Let $(M, i)$ be a pointed conical $\Gamma$-monoid with a distinguished order unit $i$. Suppose that $M$ is generated by finitely many elements whose sum is $i$. Then there is a hereditary unital hyper Leavitt path algebra $R=L_\K(H)$ with a $\Gamma$-grading induced by a weight map, such that $\phi:  (M,i) \cong (\V^{\gr}(R),[R])$ as pointed $\Gamma$-monoids. Furthermore, $R$ has a weak universal property. 
\end{thm}
\begin{proof}
Let $X=\{p_\phi\mid \phi\in\Phi\}$, where $\Phi$ is a finite set, be a set non-zero generators for $M$ such that $\sum_{\phi\in\Phi} p_\phi=i$. Let $Y=\{u_\psi=v_\psi\mid \psi\in\Psi\}$ be a set of relations such that $M=\langle X\mid Y\rangle$. Since $M$ is conical, we may assume that $u_\psi,v_\psi\neq 0$ for any $\psi\in \Phi$. Let $R'$ be the $\Gamma$-graded $\K$-algebra $\K^\Phi$ which is concentrated in degree $0$. Then
\[\V^{\gr}(R')= \big \langle P_\phi~(\phi\in\Phi) \rangle\] 
is freely generated by generators $P_\phi~(\phi\in\Phi)$ whose sum is $[R']$. Let $Z$ be the set of all finite subsets of $\Psi$. Clearly $Z$ is a directed set together with the inclusion relation. For any $A\in Z$ let $M_{A}=\langle X\mid Y_{A}\rangle$ where $Y_{A}=\{u_\psi=v_\psi\mid \psi\in B\}$. If $A\subseteq A'\in Z$, then clearly there is a canonical $\Gamma$-monoid homomorphism $\alpha_{A}^{A'}:M_{A}\to M_{A'}$. One checks easily that $M$ is the direct limit of the direct system $\{M_{A},\alpha_{A}^{A'}\mid A\subseteq A'\in Z\}$ in the category of $\Gamma$-monoids. 

Now fix an $A\in Z$. Let $R_{A}$ be the $\Gamma$-graded $\K$-algebra obtained from $R'$ by adjoining isomorphisms $\overline{U_\psi}\cong \overline{V_\psi}~(\psi\in A)$ where $U_{\psi}$ and $V_{\psi}$ are the graded $R'$-modules corresponding to $u_\psi$ and $v_\psi$, respectively. Then $\V^{\gr}(R_{A})\cong M_{A}$ by Corollary \ref{corbergV}, and this isomorphism maps $[R_{A}]$ to $i$. 

If $A\subseteq A'\in Z$, then clearly there is a canonical graded $\K$-algebra homomorphism $\beta_{A}^{A'}:R_{A}\to R_{A'}$. Let $R$ be the direct limit of the direct system $\{R_{A},\beta_{A}^{A'}\mid A\subseteq A'\in Z\}$ in the category of graded $\K$-algebras. Since $\V^{\gr}$ commutes with direct limits, we obtain
\[\V^{\gr}(R)=\V^{\gr}(\varinjlim R_{A})=\varinjlim \V^{\gr}(R_{A})\cong\varinjlim M_{A}=M.\]
Clearly the isomorphism $\V^{\gr}(R)\cong M$ maps $[R]$ to $i$. By Lemma \ref{lemhyperberg}, $R$ is graded isomorphic to a unital hyper Leavitt path algebra $L_\K(H)$ whose grading is induced by a weight map. That $R$ is hereditary and has a weak universal property follow as in the proof of Theorem~\ref{mainalibm}.
\end{proof}


\section{Application to graph  algebras}
\label{lpaaplicaton}

The Graded  Classification Conjecture for Leavitt path algebras states that the graded Grothendieck group, $K^{\operatorname{gr}}_0$,  is a complete invariant for 
these algebras~(\cite{willie, haz2013}, \cite[\S7.3.4]{lpabook}):  

\begin{conj}\label{conjalg}
Let $E$ and
$F$ be finite graphs. 
\begin{enumerate}[\upshape(1)]

\item For any order preserving $\Z[x,x^{-1}]$-module homomorphism $\phi: K^{\gr}_0(L_\K(E)) \rightarrow K^{\gr}_0(L_\K(F))$ with 
$\phi([L_\K(E)])=L_\K(F)$, there exists a unital $\mathbb Z$-graded $\K$-homomorphism $\psi: L_\K(E) \rightarrow L_\K(F)$ such that $K^{\gr}_0(\psi) = \phi$.

    \item For any order preserving  $\Z[x,x^{-1}]$-module
isomorphism $\phi: K_0^{\gr}(L_\K(E)) \rightarrow K_0^{\gr}(L_\K(F))$ with 
$\phi([L_\K(E)]=[L_\K(F)]$, there exists a unital $\mathbb Z$-graded $\K$-isomorphism $\psi: L_\K(E) \rightarrow L_\K(F)$ such that $K^{\gr}_0(\psi) = \phi$.

\end{enumerate}

\end{conj}

Here the order preserving $\Z[x,x^{-1}]$-module isomorphism $K_0^{\gr}(L_\K(E)) \cong K_0^{\gr}(L_\K(F))$ should give that these algebras are graded Morita equivalent (see Conjecture~\ref{conji1}). 

Since the graded $K$-theory of Leavitt path algebras coincides with equivariant $K$-theory of graph $C^*$-algebras one can extend the conjecture to the setting of $C^*$-algebras.  
Denote by $\gamma_E$ the gauge circle actions on $C^*(E)$ and $K^\T_0(C^*(E))$ the equivariant $K$-theory of $C^*(E)$~\cite{chrisp}. There are canonical order preserving isomorphisms of $\Z[x,x^{-1}]$-modules (see~\cite[p. 275]{haz2013}, \cite{haziso} and \cite[Proof of Theorem A]{eilers2}). 
\begin{equation}\label{copenhag}
K^{\gr}_0(L(E)) \cong K_0(L(E \times \mathbb Z)) \cong  K_0(C^*(E \times \mathbb Z))\cong K_0^{\mathbb T}(C^*(E)).
\end{equation}

Thus one can pose the analytic version of Conjecture~\ref{conjalg} as follows. 

\begin{conj}\label{conjanal}
Let $E$ and $F$ be finite graphs. Then there is an order preserving $\mathbb Z[x,x^{-1}]$-module
isomorphism
\[ \phi: K^\T_0(C^*(E)) \longrightarrow K_0^\T(C^*(F)),\] 
with 
$\phi([C^*(E)])=[C^*(F)]$ 
 if and only if
 $C^*(E) \cong C^*(F)$ which respect the gauge action. 
\end{conj}

In fact in Conjecture~\ref{conjalg} if the graded Grothendieck groups are isomorphic, one should have that the isomorphism between the Leavitt path algebras is indeed a $*$-isomorphism. If this is the case, then Conjecture~\ref{conjalg} implies Conjecture~\ref{conjanal}. For, if $K^\T_0(C^*(E)) \cong K_0^\T(C^*(F))$, then by (\ref{copenhag}), $K_0^{\gr}(L_{\mathbb C}(E)) \cong K_0^{\gr}(L_{\mathbb C}(F))$, so $L_{\mathbb C}(E)\cong_{\gr} L_{\mathbb C}(F)$ via a $*$-isomorphism. This implies that $C^*(E) \cong C^*(F)$ which respects the gauge action (see~\cite[Theorem~4.4]{abramstomforde}).

 The notion of \emph{talented monoids} allows us to write the classification conjectures for Leavitt and $C^*$-algebras in a unified manner. The positive cone of the graded Grothendieck group of a Leavitt path algebra $L_\K(E)$ can
be described purely based on the underlying graph, via the so-called talented monoid $T_E$ of $E$~\cite{hazli}. The benefit of talented monoids is that they give us control over elements of the monoids (such as minimal elements, atoms, etc.)
and consequently on the ``geometry'' of the graphs such as the number of cycles, their lengths, etc. (see~\cite{hazli,Luiz}).

\begin{deff}\label{talentedmon}
Let $E$ be a row-finite directed graph. The \emph{talented monoid} of $E$, denoted $T_E$, is the commutative 
monoid generated by $\{v(i) \mid v\in E^0, i\in \mathbb Z\}$, subject to
\[v(i)=\sum_{e\in s^{-1}(v)}r(e)(i+1)\]
for every $i \in \mathbb Z$ and every $v\in E^{0}$ that is not a sink. The additive group $\mathbb{Z}$ of integers acts on $T_E$ via monoid automorphisms by shifting indices: For each $n,i\in\mathbb{Z}$ and $v\in E^0$, define ${}^n v(i)=v(i+n)$, which extends to an action of $\mathbb{Z}$ on $T_E$. Throughout we will denote elements $v(0)$ in $T_E$ by $v$.
\end{deff}

The crucial ingredient for us is the action of $\mathbb Z$ on the monoid $T_E$. The general idea is that the monoid structure of $T_E$ along with the action of $\mathbb Z$ resemble the graded ring structure of a Leavitt path algebra $L_{\mathsf k}(E)$.
Thus the conjectures above roughly state that isomorphism of talented monoids can lift to the isomorphisms of graded and equivariant $K$-theory of respected Leavitt and graph $C^*$-algebras. We therefore can formulate both conjectures as follows:

\begin{conj}\label{conji1}
Let $E$ and $F$ be finite graphs and $\K$ a field. Then the following are equivalent.

\begin{enumerate}[\upshape(1)]
\item  The talented monoids \(T_E\) and \(T_F\) are \(\mathbb{Z}\)-isomorphic;

\item  The Leavitt path algebras \(L_\K(E)\) and \(L_\K(F)\) are graded Morita equivalent;

\item The graph $C^*$-algebras $C^*(E)$ and $C^*(F)$ are equivariant Morita equivalent. 
\end{enumerate}

Furthermore,  the following are equivalent.

\begin{enumerate}[\upshape(1)]
\item  The talented monoids \(T_E\) and \(T_F\) are pointed \(\mathbb{Z}\)-isomorphic;

\item  The Leavitt path algebras \(L_\K(E)\) and \(L_\K(F)\) are graded isomorphic;

\item The graph $C^*$-algebras $C^*(E)$ and $C^*(F)$ are equivariant isomorophic. 
\end{enumerate}
\end{conj}

 G. Arnone in \cite{arnone} and L. Vas in \cite{vas} answered part (2) of Conjecture~\ref{conjalg}, independently and with different approaches,  in positive. In fact Arnone shows that the lifting map can be diagonal preserving graded $*$-homomorphism and Vas shows that the graphs can consist of infinite edges. 

Using theory developed here, we can in fact show that for any $\mathbb Z$-graded $\K$-algebra $A$ and any pointed $\mathbb Z$-homomorphism $\phi:\mathcal V^{\gr}(L_\K(E)) \rightarrow \mathcal V^{\gr}(A)$, we can lift $\phi$ to a graded $\K$-algebra homomorphism  $\psi: L_\K(E) \rightarrow A$, which preserves the injectivity.

\begin{thm}\label{arnonevas}
Let $E$ be a finite graph and let $A$ be a $\mathbb Z$-graded $\K$-algebra. Let $\phi: T_E \rightarrow \mathcal V^{\gr}_0(A)$ be a $\mathbb Z$-monoid homomorphism with  
$\phi(1_E)=[A]$. Then there exists a unital $\mathbb Z$-graded $\K$-algebra homomorphism $\psi: L_\K(E) \rightarrow A$ such that $\mathcal V^{\gr}_0(\psi) = \phi$. Furthermore, if $\phi(v)\not = 0$, for all $v\in E^0$, then $\psi$ is monomorphism. 

\end{thm}
\begin{proof}
Consider the semisimple $\K$-algebra $R=\prod_{E^0}\K$, the product of $|E^0|$-copies  of the field $\K$. 
Then 
\begin{equation}\label{timemm}
\mathcal V^{\gr}(R) \cong \big \langle v(i)\mid v\in E^0, \, i \in \mathbb Z \big \rangle,
\end{equation}
with the $\mathbb Z$-action ${}^n v(i)=v(i+n)$, for $n,i \in \mathbb Z$. 
Denote $p_v(i)$, where $v\in E^0$ and $i\in \mathbb Z$, the graded finitely generated projective $R$-module with $\K(i)$ appears in $v$-th component of $R$ and zero elsewhere. Throughout, we write $p_v$ for $p_v(0)$. Note that $\bigoplus_{v\in E^0}p_v \cong R$ as graded $R$-module and the isomorphism classes $[p_v(i)]$ correspond to $v(i)$ in (\ref{timemm}). Next we consider the pairs of graded finitely generated projective $R$-modules 
$p_v$ and $\bigoplus_{e\in s^{-1}(v)}p_{r(e)}(1)$, for $v\in E^0$ that are not sinks. We construct the graded ring
\begin{equation}\label{sequi}
    S:=R\big\langle h_v,h_v^{-1}:\overline{p_v}\cong \overline{\bigoplus_{e\in s^{-1}(v)}p_{r(e)}(1)}\big\rangle,
    \end{equation}for vertices $v$ which are not sink. This amounts to adjoining matrices  to $R$ with relations which define the $\K$-algebra $L_\K(E)$ (see Theorem~\ref{levbalcon}), and thus 
    \begin{equation}\label{sequjj}
        S\cong_{\gr} L_\K(E).
    \end{equation}

On the other hand, by Theorem~\ref{thmbergV}, the monoid $\mathcal V^{\gr}(L_\K(E))$ is obtained by $\mathcal V^{\gr}(R)$ subject to relations \[[p_v(i)]\cong \bigoplus_{e\in s^{-1}(v)}[p_{r(e)}(i+1)],\] for $i\in \mathbb Z$. In~(\ref{timemm}) these relations translate to 
 $v(i)=\sum_{e\in s^{-1}(v)}r(e)(i+1)$. Thus $\mathcal V^{\gr}(L_\K(E))$  is precisely the talented monoid $T_E$, where the graded $\K$-algebra homomorphism $R\rightarrow L_\K(E)$, induces 
 $\mathcal V^{\gr}(R) \rightarrow \mathcal V^{\gr}(L_\K(E)), [p_v] \mapsto [L_K(E)v]$.

 Now suppose $\phi: T_E \rightarrow \mathcal V^{\gr}(A)$ is a pointed $\mathbb Z$-monoid homomorphism. Since $\phi(1_E)=\sum_{v\in E^0}\phi(v)=[A]$, we obtain a graded $A$-module isomorphism $A\cong \bigoplus_{v\in E^0} q_v$, where $q_v$ are graded finitely generated $A$-modules with $\phi(v)=[q_v]$.   It follows that $A = \bigoplus_{v\in E^0} Ae_v$, where $e_v$'s are pairwise orthogonal idempotents of homogeneous degree zero in $A$, with $Ae_v \cong q_v$ as graded $A$-modules. Thus there is a natural graded $\K$-algebra homomorphism $\eta: R=\prod_{E^0}\K \rightarrow A$, making the ring $A$ a $\mathbb Z$-graded $R\text{-ring}_{\K}$. Since $A \otimes_R p_v \cong q_v$ (use $A \otimes_R R/I \cong A/AI$, for an ideal $I$ of $R$), we obtain the following commutative diagram of $\mathbb Z$-monoids:
 \[\xymatrix{
\mathcal V^{\gr}(R) \ar[rr]^{A \otimes_R -}   \ar[dr]& & \V^{\gr}(A)\\
 & T_E \ar[ur]_{\phi}}\]

Thus 
\[[q_v]=\phi(v)=\phi\big(\sum_{e\in s^{-1}(v)}r(e)(i+1)\big)=\sum_{e\in s^{-1}(v)}[q_{r(e)}(i+1)],\] implying graded $A$-module isomorphisms  
\begin{align*}
q_v &\cong \bigoplus_{e\in s^{-1}(v)}q_{r(e)}(1)\\
A\otimes_R p_v &\cong A \otimes_R \big (\bigoplus_{e\in s^{-1}(v)}p_{r(e)}(1)\big),
\end{align*}
for vertices $v$ which are not sink. 

Since $L_\K(E)$ is the universal ring providing these isomoprhisms (see~\ref{sequi} and \ref{sequjj}), it follows that there is a graded $\K$-algebra homomorphism $\psi: L_\K(E) \rightarrow A$, such that the following diagram is commutative: \[\xymatrix{
\mathcal V^{\gr}(R) \ar[rr]^{A\otimes_R -}   \ar[dr]& & \V^{\gr}(A)\\
 & \mathcal V^{\gr}(L_\K(E)) \ar[ur]_{\; \phi= A \otimes_{L_\K(E)} -}}\]

Finally, suppose $\phi(v)\not = 0$, for all $v\in E^0$. Since $\phi(v)=[A \otimes_{L_\K(E)} L_\K(E)v]=[A\psi(v)]$, it follows that  $\psi(v)\not = 0$, for all $v\in E^0$. Now the graded uniqueness theorem for Leavitt path algebras (\cite[Theorem 2.2.15]{lpabook}) implies that $\psi$ is injective. 
\end{proof}

\begin{cor}\label{linvco}
Let $E$ and $F$ be finite graphs and  $\phi: K^{\gr}_0(L_\K(E)) \rightarrow K^{\gr}_0(L_\K(F))$ a preordered $\mathbb Z[x,x^{-1}]$-module homomorphism with 
$\phi([L_\K(E)])=L_\K(F)$. Then there exists a unital $\mathbb Z$-graded homomorphism $\psi: L_\K(E) \rightarrow L_\K(F)$ such that $K^{\gr}_0(\psi) = \phi$. Furthermore if $\phi$ is injective, so is $\psi$.
\end{cor}
\begin{proof}
    Since the positive cone of $K^{\gr}_0(L_\K(E))$ is precisely $\mathcal V^{\gr}(L_\K(E))$ (see~\cite{ara-hazrat-li-sims}) and  $\phi$ is preordered, we obtain a $\mathbb Z$-monoid homomorphism  $\phi: \mathcal V^{\gr}(L_\K(E)) \rightarrow \mathcal V^{\gr}(L_\K(F))$. The Corollary now follows from Theorem~\ref{arnonevas}.
\end{proof}

\begin{rmk}
Although throughout the paper, by setting the grade group $\Gamma$ to be trivial, we recover the non-graded universal constructions,  Corollary~\ref{linvco}, however,  is not valid for the (non-graded) Grothendieck group $K_0$. This is because the positive cone of $K_0(L(E))$ is not necessarily $\mathcal V(L(E))$. As an example, since  $K_0(L_\K(1,2))=0$ and $K_0(L(2,3))=0$, we have a  preordered pointed monoid homomorphism $K_0(L_\K(1,2))\rightarrow K_0(L(2,3))$, however there does not exist any unital ring homomorphsim $L_\K(1,2) \rightarrow L_\K(2,3).$
\end{rmk}

We finish this section by  giving a new proof for the graded 
uniqueness theorem for Leavitt path algebras (See~\cite[Theorem 2.2.15]{lpabook} for an element-wise proof). Although this proof can be extended to arbitrary graphs, for simplicity we work with Leavitt path algebras associated to finite graphs, where all our applications are concerned.  

\begin{thm}\label{uniquethm}
Let $E$ be a finite graph, $A$ a $\mathbb Z$-graded ring and $\phi:L_\K(E)\rightarrow A$ a graded ring homomorphism. If $\phi(v)\not = 0$ for all $v\in E^0$, then $\phi$ is injective.
\end{thm}
\begin{proof}
The ring homomorphism $\phi$ induces a $\Gamma$-monoid homomorphism 
\begin{align*}
\overline \phi: \mathcal V^{\gr}(L_\K(E)) &\longrightarrow \mathcal V^{\gr}(A),\\
[L_\K(E)v]&\longmapsto [A\phi(v)]
\end{align*}
Since any graded  finitely generated projective  $L_\K(E)$-module, is generated by some   $L_\K(E)v$, $v\in E^0$  (see the first part of the proof of Theorem~\ref{arnonevas}) and since $\phi(v)\not = 0$ for all $v\in E^0$, then $\overline \phi^{-1}(0)=\{0\}.$ 
Suppose $\phi(a)=0$, for some $0 \not = a\in L_\K(E)$. Without the  loss of generality we can assume $a$ is homogeneous. Consider the left ideal $L_\K(E)a$. Since $L_\K(E)$ is graded von Neumann regular, $L_\K(E)a=  L_\K(E)e$, for some homogeneous idempotent $e$. Thus $L_\K(E)a$ is a graded finitely generated projective module and $\overline \phi([L_\K(E)a])=[L_\K(E)\phi(a)]=0$, which is a contradiction. 
\end{proof}

\noindent{{\bf Acknowledgements}}
We would like to thank Gene Abrams whose germane questions  led us to try to understand Bergman's paper. Our thanks to George Bergman who provided us with invaluable inputs which improved the presentation of the paper.

\end{document}